\documentclass[letterpaper,11pt]{amsart}

\usepackage{amsmath}
\usepackage{amsthm}
\usepackage{mathrsfs}
\usepackage{amssymb,amscd}
\usepackage[all]{xy}
\usepackage{mathtools}
\usepackage{fullpage}
\usepackage{comment}
\usepackage{tikz-cd}

\usepackage{xcolor}
\usepackage[colorlinks=true,linkcolor=blue,breaklinks,pagebackref]{hyperref}

\usepackage[lite,abbrev,msc-links,alphabetic]{amsrefs}

\newcommand{\fa}{\mathfrak{a}}

\newcommand{\fg}{\mathfrak{g}}

\newcommand{\fj}{\mathfrak{j}}

\newcommand{\fm}{\mathfrak{m}}
\newcommand{\fn}{\mathfrak{n}}

\newcommand{\fs}{\mathfrak{s}}

\newcommand{\fu}{\mathfrak{u}}

\newcommand{\fX}{\mathfrak{X}}

\newcommand{\bA}{\mathbb{A}}

\newcommand{\G}{\mathbb{G}}
\newcommand{\bG}{\mathbb{G}}

\newcommand{\C}{\mathbb{C}}
\newcommand{\Q}{\mathbb{Q}}
\newcommand{\R}{\mathbb{R}}

\newcommand{\Z}{\mathbb{Z}}

\newcommand{\cA}{\mathcal{A}}
\newcommand{\cB}{\mathcal{B}}
\newcommand{\cC}{\mathcal{C}}

\newcommand{\cF}{\mathcal{F}}

\newcommand{\cH}{\mathcal{H}}

\newcommand{\cO}{\mathcal{O}}
\newcommand{\cP}{\mathcal{P}}
\newcommand{\cS}{\mathcal{S}}
\newcommand{\cT}{\mathcal{T}}
\newcommand{\cU}{\mathcal{U}}

\newcommand{\cZ}{\mathcal{Z}}

\newcommand{\rd}{\mathrm{d}}

\newcommand{\sfc}{\mathsf{c}}

\newcommand{\tS}{\mathtt{S}}

\newcommand{\bs}{\backslash}

\newcommand{\herm}{\mathsf{Herm}}

\newcommand{\tp}[1]{\prescript{t}{}{#1}}

 \DeclareMathOperator{\supp}{supp}
 \DeclareMathOperator{\vol}{vol}

\DeclareMathOperator{\Gal}{Gal}

\providecommand{\abs}[1]{\lvert#1\rvert}
\providecommand{\aabs}[1]{\lVert#1\rVert}
\providecommand{\Abs}[1]{\left\lvert#1\right\rvert}
\newcommand{\valP}[1]{\left|#1\right|}

\DeclareMathOperator{\Tr}{Tr}

\DeclareMathOperator{\Trace}{Trace}

\DeclareMathOperator{\GL}{GL}

\DeclareMathOperator{\Sp}{Sp}

\DeclareMathOperator{\Mp}{Mp}

\DeclareMathOperator{\U}{U}

\DeclareMathOperator{\Res}{Res}

\DeclareMathOperator{\Ind}{Ind}

\DeclareMathOperator{\Hom}{Hom}

\DeclareMathOperator{\Lie}{Lie}

\DeclareMathOperator{\id}{\mathbf{1}}

\DeclareMathOperator{\disc}{disc}

\DeclareMathOperator{\BC}{BC}

\DeclareMathOperator{\Ad}{Ad}

\DeclareMathOperator{\otimeshat}{\widehat{\otimes}}

\DeclareFontFamily{U}{mathx}{\hyphenchar\font45}
\DeclareFontShape{U}{mathx}{m}{n}{
      <5> <6> <7> <8> <9> <10>
      <10.95> <12> <14.4> <17.28> <20.74> <24.88>
      mathx10
      }{}
\DeclareSymbolFont{mathx}{U}{mathx}{m}{n}
\DeclareMathAccent{\widecheck}{\mathalpha}{mathx}{"71}

\makeatletter
\def\Ddots{\mathinner{\mkern1mu\raise\p@
\vbox{\kern7\p@\hbox{.}}\mkern2mu
\raise4\p@\hbox{.}\mkern2mu\raise7\p@\hbox{.}\mkern1mu}}
\makeatother

\theoremstyle{definition}
\newtheorem{definition}{Definition}[section]

\theoremstyle{plain}
\newtheorem{theorem}[definition]{Theorem}
\newtheorem{prop}[definition]{Proposition}
\newtheorem{lemma}[definition]{Lemma}

\theoremstyle{remark}
\newtheorem{remark}[definition]{Remark}

\numberwithin{equation}{section}

\begin{document}

\title{The global Gan--Gross--Prasad conjecture for Fourier--Jacobi periods on unitary groups I: Coarse expansions of the relative trace formulae}

\author{Paul Boisseau}

\author{Weixiao Lu}

\author{Hang Xue}

\address{Paul Boisseau, Max Planck Institute for Mathematics, Vivatsgasse 7, 53111 Bonn, Germany}
\email{boisseau@mpim-bonn.mpg.de}

\address{Weixiao Lu, Aix Marseille Univ, CNRS, I2M, Marseille, 13009, France}
\email{weixiao.lu@univ-amu.fr}

\address{Hang Xue, Department of Mathematics, The University of Arizona, Tucson, AZ, 85721, USA}

\email{xuehang@arizona.edu}

\date{\today}

\begin{abstract}
This is the first of a series of three papers where we prove the Gan--Gross--Prasad conjecture for Fourier--Jacobi periods on unitary groups and an Ichino--Ikeda type refinement. Our strategy is based on the comparison of relative trace formulae formulated by Liu. The goal of this first paper is to introduce the relative trace formulae and establish the coarse expansions.
\end{abstract}

\maketitle

\tableofcontents

\section{Introduction}

In the 1990s, Gross and Prasad~\cites{GP1,GP2} formulated some conjectures on the
restriction problems for orthogonal groups. These conjectures were later
extended to all classical groups by Gan, Gross and Prasad in their
book~\cite{GGP}. Ichino and Ikeda~\cite{II} gave a refinement of the
conjecture of Gross and Prasad, and this refinement was further extended to
other cases~\cites{NHarris,Liu2,Xue2,Xue7}. These conjecture are usually
referred to as the Gan--Gross--Prasad (GGP) conjectures and have attracted
significant amount of research. These conjectures describe the relation
between certain period integrals of automorphic forms on classical groups and
the central values of some $L$-functions.

There are two kinds of period integrals in the GGP conjecture: Bessel periods
and Fourier--Jacobi periods. Fourier--Jacobi periods usually involve a theta
functions while Bessel periods do not. Specifying to the case of unitary
groups, Bessel periods are on the unitary groups $\U(n) \times \U(m)$ where
$n-m$ is odd, while the Fourier--Jacobi case is when $n-m$ is even. The
cases of Bessel periods on unitary groups, together with their refinements,
are now completely settled, by the combination of the work of many people,
cf.~\cites{JR,Yun,Zhang1,Zhang2,Xue3,BP,BP1,BPLZZ,BPCZ,BPC22} for an
incomplete list. The proof follows the approach of Jacquet and Rallis
via the comparison of relative trace formulae.

Inspired by the work of Jacquet and Rallis, Liu~\cite{Liu} proposed a
relative trace formula approach towards the Fourier--Jacobi case of the conjecture. Following
this approach, we previously in~\cites{Xue1,Xue2} proved some cases of the
GGP conjecture for $\U(n) \times \U(n)$ under various local conditions. The
goal of this series of three papers is to work out the comparison of these relative trace
formulae in general, and prove the GGP conjecture for Fourier--Jacobi periods
on unitary groups. The goal of this series of papers is to prove, among many other things, the following two conjectures.

\begin{enumerate}
\item The global GGP conjecture for $\U(n) \times \U(m)$, where $n-m$ is
    even, as stated in~\cite{GGP}*{Conjecture~26.1}.

\item An exact analogue of a conjecture of Ichino and Ikeda stated in~\cite{II}*{Conjecture~1.4} 
    in the context of Fourier--Jacobi periods, which is
    a refinement of the global GGP conjecture.
\end{enumerate}

This is the first paper in this series. Its goal is to introduce the relative trace formulae and establish the coarse expansions. In the rest of this introduction, we explain what we are going to prove in this series of papers in a little more detail. This serves as an introduction to the whole series of the papers.

\subsection{Fourier--Jacobi periods}

Let $E/F$ be a quadratic extension of number fields, and $\bA_E$ and $\bA$
be their rings of adeles respectively. Denote by $\mathsf{c}$ the nontrivial
element in the Galois group $\Gal(E/F)$. For any positive integer $k$, set
$G_k:= \Res_{E/F} \GL_k$, where $\Res_{E/F}$ is the Weil restriction of scalars. Fix a nontrivial additive character
$\psi: F \bs \bA \to \C^\times$. Let $\eta: F^\times \bs \bA^\times \to
\{ \pm 1 \}$ be the quadratic character attached to the extension $E/F$ via  global class field theory. Let $\mu: E^\times \bs \bA_E^\times
\to \C^\times$ be a character such that $\mu|_{\bA^\times} = \eta$.

Let $(V, q_V)$ be a nondegenerate $n$-dimensional skew
$\mathsf{c}$-Hermitian space, with
skew-Hermitian form $q_V$. Let $\Res V$ be the $F$-symplectic space whose
underlying vector space is $V$ viewed as a $F$-vector space, and
whose symplectic pairing is given by $\Tr_{E/F} \circ q_V$. Let $\Res V = L +
L^\vee$ be a polarization, i.e. $L$ and $L^\vee$ are maximal isotropic subspaces
of $V$ such that the pairing $\Tr_{E/F} \circ q_V$ is
nondegenerate when restricted to $L \times
L^\vee$. We have a Weil representation $\omega =
\omega_{\psi, \mu}$, realized on the space of Schwartz functions $\cS(L^\vee(\bA))$,
cf.~Subsection~\ref{subsec:theta_series_U}. It depends on the characters
$\psi$ and $\mu$. We denote by $\omega^\vee = \omega_{\psi^{-1}, \mu^{-1}}$ the dual
representation of $\omega$ which we also realize on $\cS(L^\vee(\bA))$. For $\phi \in \cS(L^\vee(\bA))$ we may form the theta
function
    \[
    \theta(g, \phi) = \sum_{x \in L^\vee(F)} (\omega(g) \phi)(x), \quad
    g \in \U(V)(\bA).
    \]
We define similarly $\theta^\vee(g, \phi)$ when $\omega^\vee$ is used.

Put $\U_V := \U(V) \times \U(V)$ and $\U_V' := \U(V)$, viewed as a subgroup of $\U_V$
by the diagonal embedding. Let $\pi$ be an irreducible cuspidal automorphic representation of $\U_V(\bA)$, and take $\varphi \in \U_V(\bA)$. Let $\phi \in \cS(L^\vee(\bA))$ be a Schwartz function. We introduce the Fourier--Jacobi period  
\begin{equation}
\label{eq:FJ_corank_0_intro}
    \cP(\varphi,\phi)=\int_{\U_V'(F) \backslash \U_V'(\bA)}  \varphi(h)  \theta^\vee(h,\phi) \rd h.
\end{equation}
Here the theta function $\theta^\vee$ instead of $\theta$ is used here for compatibility with the choice in~\cite{GGP}.

The linear form $(\varphi, \phi) \mapsto \cP(\varphi,
\phi,\lambda)$ belongs to
    \[
    \Hom_{\U_V'(\bA)}(\sigma \otimes \omega^\vee, \C).
    \]
Thus in order to have a nonzero Fourier--Jacobi period $\cP$, this $\Hom$
space has to be nonzero in the first place. By the multiplicity one
theorems~\cites{Sun,SZ} its dimension is at most one. Whether this space is of dimension zero or one is the subject of study of the local GGP conjecture, which is now estalished in~\cites{GI2,Xue6}. 

We denote by $\BC$ the weak base change of $\pi$ to $\GL_n(\bA_E) \times \GL_n(\bA_E)$, and assume that it is a hermitian Arthur parameter, cf.~\cite{BPCZ}*{Section~1} for an explanation of this notion. We will not need this until the third paper in this series, and it will be carefully explained there.

The global GGP conjecture for Fourier--Jacobi periods on unitary groups, in its simplest form, can be stated as follows.

\begin{theorem}\label{thm:GGP_intro}
Let the notation and assumptions be as above. Assume that
    \[
    \Hom_{\U_V(\bA)}(\pi \otimes \omega^\vee, \C)
    \not=0.
    \]
The following are equivalent.
\begin{enumerate}
\item $\cP$ is not identically zero.

\item $L(\frac{1}{2}, \BC(\pi) \otimes \mu^{-1})
    \not=0$.
\end{enumerate} 
\end{theorem}

There is a refinement of this theorem, which states that if $\pi$ is tempered, then $\abs{\cP(\varphi,\phi)}^2$ equals
    \[
    \frac{L(\frac{1}{2}, \BC(\pi) \otimes \mu^{-1})}{L(1, \pi, \Ad)} \times \cdots
    \]
where $\cdots$ are some explicit and elementary constants.  Both the theorem and its refinement will be proved in the third paper in this series.

In general, we may consider a pair of skew-hermitian spaces $W \subset V$, such that $\dim V - \dim W$ is even. Let $\pi$ be an irreducible cuspidal automorphic representations of $\U(V)(\bA) \times \U(W)(\bA)$. Then we may define Fourier--Jacobi periods for $\pi$, and there is an analogous GGP conjecture for these periods. We will also prove these conjectures and their refinement in the third paper in this series.

\subsection{The relative trace formulae}
Our approach is through the comparison of relative trace formulae (RTF) proposed by Liu~\cite{Liu}, one on unitary groups which gives Fourier--Jacobi periods, and the other on general linear groups which gives central $L$-values.

We explain the RTF on unitary groups. Let $f \in \cS(\U_V(\bA))$ be a test function, and 
    \[
    K_f(x, y) = \sum_{\gamma \in \U_V(F)}
    f(x^{-1} \gamma y) 
    \]
the usual automorphic kernel function on $\U_V(\bA)$. Let $\phi_1, \phi_2 \in \cS(L^\vee(\bA))$ be Schwartz functions. We introduce the distribution
    \begin{equation}    \label{eq:J_intro}
    J(f, \phi_1, \phi_2) = \iint_{(\U_V'(F) \bs \U_V'(\bA))^2}
    K_f(x, y) \theta^\vee(x, \phi_1)
    \theta(y, \phi_2) \rd x \rd y.
    \end{equation}
If everything is absolutely convergent, e.g. when $V$ is anisotropic, then the distribution $J$ has a spectral expansion
    \[
    \sum_{\pi} \sum_{\varphi}
    \cP(\pi(f) \varphi, \phi_1) \overline{\cP(\varphi, \phi_2)},
    \]
where $\pi$ ranges over all cuspidal automorphic representations of $\U_V(\bA)$ and $\varphi$ ranges over an orthonormal basis of $\pi$. This expansion is called the spectral side of the RTF.

Again assuming convergence, the distribution $J$ unfolds to orbital integrals. Liu introduced in~\cite{Liu} a partial Fourier transform
    \[
    -^\ddag: \cS(L^\vee(\bA)) \otimes \cS(L^\vee(\bA)) \to \cS(V(\bA_F)),
    \]
and defined a function $\varphi \in \cS(\U(V)(\bA_F) \times V(\bA_F))$ by
    \[
    \varphi(g, v) = \int_{\U(V)(\bA_F)}  f(h^{-1}, h^{-1} g)
    (\omega^\vee(h) \phi_1 \otimes \phi_2)^\ddag(v).
    \]
The distribution $J$ then unfolds to 
    \[
    J(f \otimes \phi_1 \otimes \phi_2) = \sum_{(\delta, v) }
    \int_{\U(V)(\bA_F)} \varphi(h^{-1} \delta h, h^{-1} v) \rd h,
    \]
where $(\delta, v)$ ranges over all the orbits of $\U(V)(F) \times V(F)$ under the action of $\U(V)(F)$ given by
    \[
    h \cdot (\delta, v) = 
    (h \delta h^{-1}, h v).
    \]
This is usually referred to as the geometric side of the RTF.

There is a striking similarity between the geometric side of Liu's RTF and the Jacquet--Rallis RTF which tackles the GGP conjecture for the Bessel periods on unitary groups. In fact the infinitesimal version of the two geometric sides are essentially the same. The work~\cite{Xue1} is largely based on this observation, and Liu's RTF under the convergence assumptions has been worked out. To extend the work of~\cite{Xue1} and prove the GGP conjecture in complete generality, extra work is needed. In particular we need to truncate the RTF.

\subsection{The truncation}
The way we set up and truncate the RTFs contains some of our major innovations in this paper. We explain the ideas behind it here, in the setting of unitary groups.

There seems to be many ways to truncate the right hand side of~\eqref{eq:J_intro}. We may replace $K_f(x, y)$ by Arthur's modified kernel, or apply mixed truncation to $K_f(x, y)$. We may also apply some sort of truncation operators to the theta functions. There does not seem to be a priori a ``canonical'' way to truncate the RTF.

The key observation, and actually the starting point of this series of papers, is that the expression~\eqref{eq:J_intro} can be written as a single kernel function on the Jacobi group. 
Let $S(V) = \Res V + F$ be the Heisenberg group $J(V) = S(V) \rtimes \U(V)$ the Jacobi group, cf.~Subsection~\ref{subsec:u_Jacobi} for a description of these groups. Take $f \in \cS(\U(V)(\bA_F) \times \U(V)(\bA_F))$, $\phi_1, \phi_2 \in \cS(L^\vee(\bA))$. Define $\widetilde{f} \in \cS(\U(V)(\bA) \times J(V)(\bA))$ by
    \[
    \widetilde{f}(g_1, \widetilde{g_2}) =
    f(g_1, g_2)  \langle \omega(\widetilde{g_2}) \phi_1,
    \overline{\phi_2}\rangle,
    \]
where $g_1 \in \U(V)(\bA_F)$, $\widetilde{g_2} \in J(V)(\bA_F)$ and its image in $\U(V)(\bA)$ is $g_2$. If $\widetilde{g_2} = ((v, 0), g_2)$, $(v, 0) \in S(V)$, then
    \[
    \widetilde{f}(g_1, \widetilde{g_2}) =
    f(g_1, g_2)
    (\omega(\widetilde{g_2}) \phi_1 \otimes \phi_2)^\ddag(v) .
    \]
That is, the Fourier transform introduced in~\cite{Liu} produces from a $f, \phi_1, \phi_2$ a test function on the group $\U(V)(\bA) \times J(V)(\bA)$. A little computation then gives that
    \[
    K_{\widetilde{f}}(h_1, h_2) = K_f(h_1, h_2) \theta^\vee(h_1, \phi_1)
    \theta(h_2, \phi_2),
    \]
if $h_1, h_2 \in \U_V'(\bA)$, if we define the kernel function $K_{\widetilde{f}}(h_1, h_2)$ on $\U(V)(\bA) \times J(V)(\bA)$ the same way as $K_f$ on the unitary groups. If we were able to work with the Jacobi groups just like those reductive groups, then there would be standard way to define modified kernels on $\U(V)(\bA) \times J(V)(\bA)$. 

A large part of the current paper is to make sense of these ideas. We introduce a special kind of subgroups of Jacobi group, which behave exactly like parabolic subgroups of reductive group. Though we tend to call them parabolic subgroups, but they are not parabolic subgroups in the usual sense, cf.~\cite{Springer}*{Section~6.2}. Instead they are defined using a description in~\cite{Springer}*{Section~13.4} via coroots. To distinguish, we call them D-parabolic subgroups. We develop the theory of constant terms of automorphic forms on Jacobi groups along these D-parabolic subgroups, and prove the crucial result on ``approximation by constant terms'' for Jacobi groups. These are carried out in Sections~3 and~4 of this paper. Once we have this, truncating the RTF~\eqref{eq:J_intro} follows a standard process. We model this process on the previous work~\cite{BPCZ}. It turns out that the D-parabolic subgroups of the Jacobi groups that we consider here and the parabolic subgroups that appear in the truncation process of the Jacquet--Rallis RTF are closely related. Thus we are able to truncate Liu's RTF, in such a way that the similarity between its geometric side and that of Jacquet--Rallis RTF is preserved under the truncation process. This similarity will be further developed in the second paper in this series.

\subsection{Organization of this series of papers}
In the first of this series of papers, we develop the coarse expansion of the RTFs, following the ideas we outlined above. 

The second of this series of paper is devoted to comparing the geometric side of the RTFs. In particular we define the notation of matching of test functions, and prove what is usually called ``smooth transfer'' and ``fundamental lemma''. By descending the geometric sides to their infinitesimal versions, we achieve a full comparison of the geometric sides. Then we prove a local character identity, which in turn characterizes the matching of test functions.

All these preparations lead to the final proof of the GGP conjectures and their refinement in the last paper in this series. We compute explicitly some terms on the spectral side of the RTFs. Following~\cites{BPCZ,BPC22}, we not only compute those terms corresponding to the cuspidal representations, but also some ``regular'' Eisenstein series. In this way, we are able to prove the global GGP conjectures and the refinement for Fourier--Jacobi periods in all codimensions.

\subsection{Acknowledgments}
We thank Rapha\"el Beuzart-Plessis and Wei Zhang for many helpful discussions.
PB was partly funded by the European Union ERC Consolidator Grant, RELANTRA, project number 101044930. Views and opinions expressed are however those of the author only and do not necessarily reflect those of the European Union or the European Research Council. Neither the European Union nor the granting authority can be held responsible for them. WL was partially supported by the National Science Foundation under Grant No. 1440140, while he was in residence at the Mathematical Sciences Research Institute in Berkeley, California, during the semester of Spring 2023. HX is partially supported by the NSF grant DMS~\#2154352.

\section{Notation and conventions}

This section contains some notation which will be used throughout the paper. Specific sets of notation will be fixed in each part.

\subsection{General notation}

If $R$ is a ring, $R^n$ will denote the set of $n$ dimensional column vectors
with coefficients in $R$, and $R_n$ the set of row vectors.

If $f$ and $g$ are functions on some space $X$, we write $f\ll g$ if there is a
constant $C$ such that $f(x) \leq C g(x)$ for all $x \in X$. We say $f$ and
$g$ are equivalent, denoted by $f \sim g$,  if $f\ll g$ and $g\ll f$.

Let $X$ be a measure space. We denote by $\langle-,-\rangle_{L^2}$ the $L^2$-inner product and $\aabs{\cdot}_{L^2}$ the $L^2$-norm. We sometimes also write $\langle-,-\rangle_X$ to emphasize the space $X$.

If $G$ is a group and $f$ is a complex valued function on $G$. We put $f^\vee(g) = f(g^{-1})$ and $f^*(g) = \overline{f(g^{-1})}$.

When $G$ is a group and $\cF$ is a space of functions on $G$ which is invariant by right
(resp. left) translation, we denote by $\mathrm{R}$ (resp. $\mathrm{L}$) the
corresponding representation of $G$ on $\cF$. If $G$ is a Lie group and the
representation is differentiable, we will also denote by the same letter the
induced action of the Lie algebra or of its associated enveloping algebra. If
$G$ is a topological group equipped with a bi-invariant Haar measure, we
denote by $*$ the convolution product of functions on $G$
(whenever it is well-defined).

All of the topological vector spaces in this paper will be Banach, Hilbert,
Fr\'{e}chet, or LF spaces. By an LF space, we mean a topological vector space which can be
expressed as a countable direct limit of Fr\'{e}chet spaces in the category of locally convex spaces. Note that it is not
necessarily complete or even Hausdorff. If $V = \varinjlim_n V_n$ is an LF space, we say that it is a strict LF space if the maps $V_n \to V_{n+1}$ are all closed embeddings.
Strict LF spaces are complete, cf.~\cite{Treves}*{Theorem~13.1} (note the LF
spaces in~\cite{Treves} are by definition what we call strict LF spaces).
The uniform boundedness principle and
the closed graph theorem hold for LF spaces, cf.~\cite{BPCZ}*{Appendix~A}. We
use the notation $\otimeshat$ to denote the projective completed tensor
product of two locally convex topological vector spaces.

\subsection{Fields}
\label{subsec:field}

We fix a quadratic extension of number fields $E/F$ unless otherwise specified. We denote by
$\textsf{c}$ the nontrivial Galois conjugation in $\Gal(E/F)$. Let $E^-$ be the purely imaginary elements in $E$,
i.e. $E^- = \{x \in E \mid x^\mathsf{c} = -x\}$, and fix a nonzero $\tau \in
E^-$. We denote by $\bA$ and $\bA_E$ the ring of adeles of $F$ and $E$
respectively. We write $\bA_{f}$
the ring of finite adeles of $\bA$. If $\tS$ is a finite set of places of $F$, we set $F_\tS := \prod_{v \in S} F_v$. When $\tS = V_{F,\infty}$ is the set of Archimedean places of $F$, we also write $F_{\infty} := F_{V_{F,\infty}}$.

We denote by $\abs{\cdot}$ and $\abs{\cdot}_E$ the normalized
absolute values on $\bA^\times$ and $\bA_E^\times$ respectively. They satisfy
$\abs{x}_E = \abs{x x^{\mathsf{c}}}$ for all $x \in \bA_E^\times$, and in particular
$\abs{x}_E = \abs{x}^2$ if $x \in \bA^\times$.

We fix a nontrivial additive character of $\psi$ of $F\bs \bA$, and let
$\psi^E$ and $\psi_E$ be the nontrivial additive characters of $\bA_E$ given
respectively by
    \[
    \psi^E(x) = \psi(\Tr_{E/F}(\tau x)), \quad \psi_E(x) = \psi(\Tr_{E/F} x).
    \]

Let $\eta$ be the quadratic character of $F^\times \bs \bA^\times$
associated to the extension $E/F$ by global class field theory. We fix an idele class character $\mu: \bA_E^\times/E^{\times} \to \C^\times$ whose restriction to $\bA^\times$ is $\eta$.

\subsection{Groups} \label{subsec:notation_intro_groups}

We denote by $\bG_a$ and $\bG_m$ the additive group and the multiplicative group
over $F$ respectively.  If $F'/F$ is a field extension,
we denote by $G_{F'}$ the extension of scalars $G \times_F F'$. We denote by
$\fg = \mathrm{Lie}(G)$ the Lie algebra of $G$ (over $F$).
We write $G_\infty$ for $G(F_\infty)$, $\fg_\infty$ for $\mathrm{Lie}(G_\infty)
\otimes_{\R} \C$, $\cU(\fg_{\infty})$ for the universal enveloping algebra of
$\fg_{\infty}$ and $\cZ(\fg_{\infty})$ for its center.
If $T$ is a torus over $F$, we let $T_1$ be the maximal split torus of
$\Res_{F/\Q} T$ and $T^\infty$ be the neutral component of $T_1(\R)$.

\subsubsection{Hermitian spaces and unitary groups}
Let $V$ be a vector space over $E$. We take the convention that a $\sfc$-Hermitian
or skew $\sfc$-Hermitian form $q_V$ on $V$ is linear in the first variable and anti-linear
in the second variable. A vector space with a nondegenerate $\sfc$-Hermitian or
skew $\sfc$-Hermitian form is called a $\sfc$-Hermitian or skew $\sfc$-Hermitian space (or simply Hermitian or skew-Hermitian). For fixed $n$, we denote by $\cH$ the set of all isomorphism classes of nondegenerate
skew-Hermitian spaces over $E$ of dimension $n$. If $v$ is a place of $F$, we denote by $\cH_v$ the set of all isomorphism classes of nondegenerate
skew-Hermitian spaces over $E_v = E \otimes_F F_v$ of dimension $n$. More generally, for a finite set $\tS$ of places, we denote by $\cH_{\tS}$ the set of isomorphism classes of nondegenerate
skew-Hermitian spaces over $E_\tS = E \otimes_F F_{\tS}$ of dimension $n$. We also denote by $\cH^{\tS}$ the set of $V \in \cH$ such that $V \otimes_{E} E_v$ has a self-dual lattice for all non-Archimedean $v \not \in \tS$.

If $V$ is a Hermitian or skew-Hermitian space of dimension $n$, we choose a basis $v_1,
\hdots, v_n$ and put
    \[
    \disc V = \det (q_V(v_i, v_j))_{1 \leq i, j \leq n} \in E^\times.
    \]
If $V$ is Hermitian then $\disc V \in F^\times$, while if $V$ is
skew-Hermitian then $\tau^n \disc V \in F^\times$. The image of $\disc V$
(when $V$ is Hermitian) or $\tau^n \disc V$ (when $V$ is skew-Hermitian) in
$F^\times /\mathrm{Nm}_{E/F}(E^\times)$ is independent of the choice of the basis.

\subsection{Measures}    \label{subsubsec:measures}

We now specify measures on the local and adelic points of algebraic groups. In this article, the precise choice of measures is not crucial, as we will not perform any explicit computations involving them. However, the choice of measures will play an important role in the subsequent articles \cite{BLX2, BLX3}, which concern relative character identities, spectral expansions, and related topics. For the sake of consistency, we therefore fix the measures here. 

For every place $v$ of $F$, let $\rd_{\psi_v} x_v$ be the unique Haar measure on $F_v$ which is self-dual with respect to $\psi_v$, the local component of the additive character $\psi$ of $\bA$.

Let $G$ be a connected linear group over $F$.
The choice of a right-invariant rational volume
form $\omega$ on $G$ together with the measure $\rd_{\psi_v} x_v$ determine a right invariant Haar measure $\rd_{\psi_v} g_v$ on each $G(F_v)$ (\cite{Wei}).
By~\cite{Gro97}, there is an Artin--Tate $L$-function $L_{G}(s)=\prod_v
L_{G,v}(s)$, and more generally for $\tS$ a finite set of places its partial counterpart $L_{G}^{\tS}(s)=\prod_{v \notin \tS}
L_{G,v}(s)$. Define $\Delta_{G}^*$ and
$\Delta_{G}^{\tS,*}$ to be the leading coefficient of the Laurent expansion
at $s=0$ of $L_{G}(s)$ and $L_{G}^{\tS}(s)$ respectively. For each place $v$,
set $\Delta_{G,v}:=L_{G,v}(0)$. We equip $G(\bA)$ with the Tamagawa
measure $\rd g$ defined as $\rd g=\rd_\psi g_{\tS} \times \rd_\psi g^{\tS}$ where $\rd_\psi
g_{\tS}=\prod_{v \in \tS} \rd_{\psi_v} g_v$ and $\rd_\psi g^{\tS}=(\Delta_{G}^{\tS,*})^{-1}\prod_{v
\notin \tS} \Delta_{G,v} \rd_{\psi_v} g_v$. Note that for any model of $G$ over
$\cO_F^{\tS}$ we have for almost all $v$
\begin{equation}
\label{eq:Delta}
    \vol(G(\cO_v),\rd_{\psi_v} g_v)=\Delta_{G,v}^{-1}.
\end{equation}
Although $\rd_{\psi_v} g_v$ depends on various choices, the Tamagawa measure $\rd g$ does not. Note that if $G$ is a unipotent group, the measure we just picked satisfies $\vol([G]) = 1$.

If $G=\GL_n$, we will take the form
\begin{equation*}
    \omega=( \det g)^{-n}  \wedge_{i,j} \rd g_{ij},
\end{equation*}
so that \eqref{eq:Delta} is satisfied for every non-Archimedean place $v$ of $F$ where $\psi_v$ is unramified.

\section{Preliminaries}

\subsection{Groups, parabolic subgroups} \label{subsubsec:D-parabolic}

Let $G$ be a connected linear algebraic group over a number field $F$. We introduce a class of
subgroups of $G$ which will play the role of parabolic subgroups for
reductive groups. The group $G$ we have in mind is either a reductive group,
or a Jacobi group which will be introduced in
Section~\ref{sec:Jacobi_groups}.

By a theorem of Mostow~\cite{Conrad}*{Proposition 5.4.1}, the group $G$ can
be written in the form $ U \rtimes H$, where $H$ is reductive and $U$ is the
unipotent radical of $G$. The group $H$ is then called a Levi subgroup of $G$. We fix such a decomposition, choose a maximal split
torus $A_0$ of $H$ and choose $P_0 \subset H$ a minimal parabolic subgroup
that contains $A_0$. We say that a parabolic subgroup of $H$ is semistandard if it
contains $A_0$, and say that it is standard if it contains $P_0$. Note that $A_0$ is
also a maximal split torus of $G$. Denote by $X^*(A_0)$ the group of rational characters of $A_0$, and by $X_*(A_0)$ the group of rational cocharacters of $A_0$. Denote by $\langle -,- \rangle : X^*(A_0) \times X_*(A_0) \to \Z$ the canonical pairing between these groups.

The torus $A_0$ has an adjoint action on the Lie algebra $\fg$ of $G$, thus a
root-space decomposition
    \[
    \fg =
    \bigoplus_{\alpha \in \Phi_G \subset X^*(A_0)} \fg_\alpha,
    \]
where $\fg_\alpha$ satisfies
    \[
    \fg_\alpha = \{ x \in \fg \mid \mathrm{Ad}(a) x
    = \alpha(a)(x) \text{ for all }a \in A_0 \}.
    \]
The set $\Phi_G$ is the set of \emph{roots}. It consists of those $\alpha \in X^*(A_0)$ such that $\fg_\alpha
\ne 0$. For $\alpha \in \Phi_G$, $\fg_\alpha$ is called a \emph{root space} of $\fg$. We say that $G$ has the symmetric root property, or property SR for short, if
the following condition holds:
    \begin{equation} \tag{SR} \label{eq:(SR)}
         \alpha \in \Phi_G  \iff -\alpha \in \Phi_G.
    \end{equation}

As maximal split tori in $H$ are conjugate to each other, the condition
$\eqref{eq:(SR)}$ is independent of the choice of $A_0$. Moreover, the Levi subgroups of $G$ are conjugate so that $\eqref{eq:(SR)}$ also does not depend on the choice of $H$. Note that $\eqref{eq:(SR)}$ holds for all connected reductive groups.

A subsemigroup of $X^*(A_0)$ is by definition a subset $\Gamma$ of $X^*(A_0)$ closed under
addition. We recall the following root subgroup construction
from~\cite{CGP}*{Proposition 3.3.6, Theorem 3.3.11}.

\begin{prop} \label{prop:root subgroup}
For any subsemigroup $\Gamma \subset X^*(A_0)$, there exists a unique connected
$F$-subgroup $H_\Gamma(G)$ of $G$ such that
    \[
    \mathrm{Lie}(H_\Gamma(G)) =
    \bigoplus_{\alpha \in \Gamma \cap \Phi_G} \fg_\alpha.
    \]
Moreover, we have the following properties.
    \begin{enumerate}
        \item $H_\Gamma(G)$ is stable under conjugation by $A_0$.

        \item If $0 \not \in \Gamma$, then $H_\Gamma(G)$ is unipotent.
            \end{enumerate}

        Moreover, assume that $G$ is solvable and suppose that there is a disjoint union decomposition $\Phi_G=\sqcup_{i=1}^n \Phi_i$ such that $\Phi_i$ is disjoint from the semigroup $\Gamma_j$ generated by $\Phi_j$ whenever $i \neq j$. Then the map induced by group
            multiplication
            \[
             H_{\Gamma_1}(G) \times \cdots \times H_{\Gamma_n}(G) \to G
            \]
        is an isomorphism of $F$-schemes.
\end{prop}

\begin{remark}
Since $\mathrm{char}(F)=0$, a connected subgroup is determined by its Lie
algebra. Therefore, we do not need the $A_0$-stable condition in ~\cite{CGP}*{Proposition 3.3.6} but
have it as a property instead.
\end{remark}

Let $\lambda$ be a cocharacter of $A_0$. It determines three subsemigroups of
$X^*(A_0)$ defined by the equations $\{ \langle \alpha ,\lambda\rangle \ge 0\}$,
$\{ \langle  \alpha ,\lambda \rangle > 0\}$ and $\{  \langle  \alpha ,
\lambda \rangle = 0\}$ respectively. The resulting root subgroups are denoted
by $P(\lambda)$, $U(\lambda)$ and $Z(\lambda)$ respectively. They are the
dynamical subgroups described in \cite{Springer}*{Section~13.4}. We have
the semidirect product decomposition
    \begin{equation}    \label{eq:D_levi}
    P(\lambda) = Z(\lambda) \ltimes U(\lambda).
    \end{equation}

\begin{lemma} \label{lem:unique Levi decomposition}
If $G$ has the property \eqref{eq:(SR)}, then the
decomposition~\eqref{eq:D_levi} only depends on $P(\lambda)$ and not on
$\lambda$.
\end{lemma}

\begin{proof}
Let $S \subset \Phi_G$ be the set of roots for $S(\lambda)$. If $G$ has the property \eqref{eq:(SR)}, then for $\alpha \in S$, we have
\begin{equation*}
    \Phi_{U(\lambda)} = \{ \alpha \in S \mid -\alpha \not \in S \},
\end{equation*}
which only depends on $S$.
\end{proof}

\begin{remark}
Lemma~\ref{lem:unique Levi decomposition} does not hold in general without the condition~\eqref{eq:(SR)}. The mirabolic subgroup in $\GL_3$ is a counterexample.
\end{remark}

From now on we will always assume that $G$ has the property ~\eqref{eq:(SR)}.

We say that a subgroup of the form $P(\lambda)$ for so $\lambda \in X_*(A_0)$ is a (semistandard) \emph{D-parabolic
subgroup}, and we call the decomposition~\eqref{eq:D_levi} the \emph{D-Levi
decomposition} of $P(\lambda)$. When $P=P(\lambda)$, we denote $M_P :=
Z(\lambda)$, $N_P:=U(\lambda)$, $\fm_P := \mathrm{Lie}(M_P)$ and
$\fn_P:=\mathrm{Lie}(N_P)$. By Proposition ~\ref{prop:root subgroup} (2),
$N_P$ is always unipotent. When there is no confusion, we will omit
the subscript $P$ in $M_P$ and $N_P$, and whenever we write the equality $P=MN$
we will always mean that it is the D-Levi decomposition of $P$ in the above sense. The D-parabolic
subgroup $P$ is called \emph{standard} if $\lambda$ is a dominant cocharacter
(with respect to $P_0$), or equivalently if $P(\lambda) \cap H$ is a standard
parabolic subgroup of $H$.

Note that for any D-parabolic subgroup $P$ of $G$, $M_P$ also has the property ~\eqref{eq:(SR)}. Moreover, when $G$ is reductive a D-parabolic subgroups is the same as a semistandard parabolic
subgroup, and the D-Levi decomposition is the same as the semistandard Levi
decomposition (that is, $A_0 \subset M_P$).

We denote by $\Psi_P$ the set of roots of the adjoint action of $A_0$ on $\fn_P$.
If $P=P(\lambda)$ for $\lambda \in X_*(A_0)$, then $\Psi_P = \{ \alpha \in \Phi_G \mid \langle \alpha,\lambda \rangle > 0 \}$.

\begin{lemma}
Let $P,Q$ be two D-parabolic subgroups of $G$. If $P \subset Q$ then $N_Q
\subset N_P$.
\end{lemma}

\begin{proof}
It suffices to check that $\fn_Q$ is a subspace of $\fn_P$. Write
$P=P(\lambda)$ and $Q=P(\mu)$. Let $\alpha \in X^*(A_0)$ with $\langle \alpha, \mu \rangle > 0$. Assume for contradiction that $\langle
\alpha, \lambda \rangle \le 0$. Then $\langle -\alpha, \lambda \rangle \ge
0$ and $-\alpha \in \mathrm{Lie}(P) \subset \mathrm{Lie}(Q)$. This
contradicts $\langle -\alpha, \mu  \rangle < 0$.
\end{proof}

Let $P=P(\lambda)$ and $Q=P(\mu)$ with $P \subset Q$. Since $N_Q$ is a normal
subgroup of $Q$, it is a normal subgroup of $N_P$. Set $\Psi_P^Q=\Psi_P \setminus \Psi_Q$. Let $N_P^Q$ be the
subgroup corresponding to the semigroup $\{ \alpha \in \Phi_G \mid \langle \alpha, \mu
\rangle = 0, \; \langle \alpha , \lambda \rangle > 0 \}$. By the last assertion of Proposition
~\ref{prop:root subgroup}, we have
    \[
    N_P = N_Q \rtimes N_P^Q.
    \]
The Lie algebra $\fn_P^Q \cong \fn_P/\fn_Q$ of $N_P^Q$ is $\oplus_{\alpha \in \Psi_P^Q} \fg_\alpha$. In particular the group $N_P^Q$ only depends on $P$
and $Q$ and not on the choices of $\lambda$ and $\mu$.

The following lemma provides a filtration on $N_P^Q$ which will be useful in the
proof of Theorem~\ref{thm:approximation_by_constant_term}.

\begin{lemma} \label{lem:filtration_of_NPQ}
    There exists an increasing filtration of unipotent subgroups
        \[
            N_0=\{0\} \subset N_1 \subset \cdots \subset N_k=N_P^Q
        \]
    of $N_P^Q$ such that the following properties hold.
        \begin{itemize}
            \item Each $N_i$ is normal in $N_P^Q$.

            \item Put $\fn_i = \Lie(N_i)$. There is a complementary subspace $\fn_{i+1}^i$ of
            $\fn_i$ in $\fn_{i+1}$ which is contained in a root space. In
            particular, the action of $A_0$ on the Lie algebra of each
            quotient $N_{i+1}/N_{i}$ $(0 \le i \le k-1)$ is by a character
            $\alpha_i \in X^*(A_0)$.

            \item $N_{i+1}/N_{i}$ is isomorphic to a product of $\bG_a$'s.
        \end{itemize}
\end{lemma}

\begin{proof}
Assume that $P=P(\lambda)$ and $Q=P(\mu)$. Take a
cocharacter $\nu$ of $A_0$ such that the numbers $\langle \alpha,\nu \rangle$, as
$\alpha$ runs through $\Psi_P^Q$, are positive and distinct. For each $n>0$,
let $\Gamma_{\ge n}$ be the subsemigroup of $X^*(A_0)$ defined by $\{ \alpha
\mid \langle \alpha, \nu \rangle  \ge n, \langle \alpha,\mu \rangle = 0, \langle \alpha, \lambda \rangle >0 \} $. For $n_0$ large enough, the filtration $\{0\}=H_{\Gamma_{\geq n_0}} \subset \hdots \subset H_{\Gamma_{\geq 1}}=N_{P}^Q$ satisfies the first two conditions. One then further refines this filtration to make it satisfy the third (e.g. take the derived series of each $H_{\Gamma_{\ge n}}/H_{\Gamma_{\ge (n+1)}}$).
\end{proof}

Let $P$ be a D-parabolic subgroup. We put
    \[
    [G]_P = N_P(\bA)M_P(F) \bs G(\bA).
    \]
When $P = G$ we simply write $[G]= [G]_G$. There is a natural map $P(F) \bs
G(\bA) \to [G]_P$ whose fibers are $N(F) \bs N(\bA)$-torsors and hence are compact.

We write $A_G$ for the maximal central split torus of $G$. More generally, for a D-parabolic subgroup $P$ of $G$, we write $A_P = A_{M_P}$ for the maximal central split torus of $M_P$.

\subsection{Heights and weights}
\label{subsection:Heights_weights}

\subsubsection{Heights on adelic points of algebraic groups} For any $F$-variety $X$, there is a height function $\| \cdot \|_{X(\bA)}$ on $X(\bA)$ introduced in~\cite{BP}*{Appendix A}. When $G$ is a connected linear algebraic group over $F$. We fix an embedding $\iota: G \hookrightarrow \GL_N$ for some $N>0$ the height function on
$G(\bA)$ can be described by
    \begin{equation} \label{eq:height_funtion_on_group}
            \aabs{g} = \prod_v \max_{1 \leq i, j \leq N}
    \{ \abs{\iota(g)_{ij}}_v, \abs{\iota(g)^{-1}_{ij}}_v\},
    \end{equation}
where the product runs over all places of $F$. Note that for another choice of embedding $\iota'$ yielding a height $\aabs{\cdot}'$, there exists $r_0>0$ such that $\aabs{g}^{1/r_0} \ll \aabs{g}' \ll \aabs{g}^{r_0}$ for $g \in G(\bA)$. The equivalence class of $\aabs{\cdot}$ (in the preceding sense) is therefore independent of the choice of $\iota$. Note that if $G = F_n$ or $F^n$ is a vector space, then we may take the height function given, for $g = (x_1, \hdots, x_n) \in G(\bA)$, by
    \begin{equation}    \label{eq:height_vector_spaces}
    \aabs{g} = \prod_v \max\{1, \abs{x_1}_v, \hdots, \abs{x_n}_v\}.
    \end{equation}

The height function $\| \cdot \|$ on $G(\bA)$ induces a height function $\| \cdot \|_\infty$ on $G(F_\infty)$ by the embedding $G(F_\infty) \hookrightarrow G(\bA)$. It is explicitly given by
    \[
    \aabs{g}_\infty = \prod_{v \mid\infty} \max_{1 \leq i, j \leq N}
    \{ \abs{\iota(g)_{ij}}_v, \abs{\iota(g)^{-1}_{ij}}_v\},
    \]
where the product runs over all Archimedean places of $F$. This is called an
algebraic scale on $G(F_\infty)$ in~\cite{BK}.

\subsubsection{Heights modulo a central unipotent subgroup}
\label{subsubsec:central_unip}

In some circumstances we will need to work with the group $G$ modulo a central unipotent subgroup $Z$ (see Subsection ~\ref{subsec:spaces_of_function}). In this setting, we will equip $G(\bA)$ with the pull-back of the height function $\| \cdot \|$ on the group $(G/Z)(\bA)$, defined in~\eqref{eq:height_funtion_on_group}, by the projection $G \to G/Z$. The resulting function will still be denoted $\| \cdot \|$ and be called a height function on $G(\bA)$. We take the convention that, whenever $G$ is equipped with a fixed central unipotent subgroup $Z$, the notation $\| \cdot \|$ always designates the function constructed this way. If necessary, the height function on $G(\bA)$ defined in ~\eqref{eq:height_funtion_on_group} will be denoted by $\| \cdot \|'$.

Note that by the equivalence of unipotent groups and
nilpotent Lie algebra, $Z$ is necessarily isomorphic to product of copies of
$\bG_a$. Since $H^1_\mathrm{fppf}(R,Z) = 0$ for any $F$-algebra $R$ ,
we have $(G/Z)(R) = G(R)/Z(R)$.

The following two lemmas and Remark~\ref{rem:[G]_P_height_equivalent}, where we assume the above setting so that $G$ is equipped with a fixed unipotent central subgroup $Z$, show that the distinction between $\| \cdot \|$ and $\| \cdot \|'$ will be mostly inconsequential on adelic quotients.

\begin{lemma} \label{lem:norm_equivalent_quotient_center}
There exist $r_1,r_2>0$, such that
        \[
            \|g\|^{r_1} \ll  \inf_{z \in Z(\bA)} \|zg\|'  \ll \|g\|^{r_2}, \quad g \in G(\bA).
        \]
\end{lemma}
In other words, the quotient maps $G \to G/Z$ satisifes the ``norm-descent" property in~\cite{Kottwitz05}*{Section 18}

\begin{proof}
By \cite{BP}*{Proposition A.1.1(ii)}, the first inequality holds, so that we now
prove the second. The projection map $G \to G/Z$ is a $Z$ torsor in the fppf
topology. Since $G/Z$ is affine we have $H^1_\mathrm{fppf}(G/Z,Z) = 0$, and thus the
torsor is a trivial torsor. It follows that there exists a section $s:G/Z \to G$. By
\cite{BP}*{Proposition A.1.1(ii)} again, there exists $r>0$ such that $\|s(g)\|' \ll \|g\|^{r}$ for all $g \in (G/Z)(\bA)$, which implies the second inequality.
\end{proof}

Let $P$ be a D-parabolic subgroup of $G$. Note that it contains $Z$. We define a height function on $[G]_P$ by
    \begin{equation} \label{eq:aabs_P}
         \aabs{g}_P = \inf_{\gamma \in P(F)} \aabs{\gamma g},
        \quad g \in [G]_P.
    \end{equation}

\begin{lemma} \label{lem:norm_on_[G]_P equivalent}
There exist $r_1,r_2>0$ such that
    \[
    \| g \|_{P}^{r_1}
       \ll \inf_{\gamma \in P(F)} \| \gamma g\|'
    \ll \| g \|_P^{r_2}, \quad g \in [G]_P.
    \]
\end{lemma}

\begin{proof}
This directly follows from ~\ref{lem:norm_equivalent_quotient_center}, since
$Z(F) \subset P(F)$ and $Z(F) \backslash Z(\bA)$ is compact.
\end{proof}

\begin{remark} \label{rem:[G]_P_height_equivalent}
    Lemma ~\ref{lem:norm_on_[G]_P equivalent} implies that, if we temporarily put
\[
    \| x \|'_P = \inf_{\gamma \in P(F)} \| \gamma x \|',
\]
there exists $r_0>0$ such that $\| g \|_P^{1/r_0} \ll \| g \|'_P \ll \| g \|_P^{r_0}$ for all $g \in [G]_P$. In particular, since the spaces of functions that we will define in Subsection~\ref{subsec:spaces_of_function} are mostly insensitive to raising the heights to a power, we may interchange the two constructions.
\end{remark}

\begin{remark} \label{rem:height_on_[G]_P_and_[H]_P_H}
\label{rk:heights_H}
    Recall that $H$ is our fixed Levi subgroup of $G$. Set $P_H=P \cap H$, which is parabolic subgroup of $H$. Using the projection $G \to H$ and the embedding $H \hookrightarrow G$, by \cite{BP}*{Proposition~A.1.1} we see that there exist $r_1,r_2 > 0$ such that for $x \in H(\bA)$, we have
    \[
        \| x \|_{P_H}^{r_1} \ll  \| x \|_{P} \ll \| x \|_{P_H}^{r_2}.
    \]
\end{remark}

\subsubsection{Weights}
We keep the setting of Subsection~\ref{subsubsec:central_unip}, so that $G$ is equipped with a central unipotent subgroup $Z$ and that $P$ is a D-parabolic subgroup.

By a weight on $[G]_{P}$ we mean a positive measurable function on $[G]_{P}$
such that there exist a positive number $N$ and a constant $C$ such that
    \begin{equation} \label{eq:weight_on_[G]_P}
        w(xg) \leq C w(x) \aabs{g}_P^N, \quad
   x \in [G]_P, \; g\in G(\bA).
    \end{equation}
For example, $\| \cdot \|_{P}$ is a weight function on $[G]_{P}$.

We say that a function $w$ on $P(F) \bs G(\bA)$ is a weight on
$P(F) \bs G(\bA)$ if it is the composition of a weight on $[G]_P$ with the
map $P(F) \bs G(\bA) \to [G]_P$. In particular $\| \cdot \|_P$ gives a
weight on $P(F) \bs G(\bA)$, and we say that it is a height function on $P(F) \bs
G(\bA)$.

\begin{remark}
The definition of the weight in~\cite{BPCZ} differs slightly from ours. The
definition there requires that for all compact subgroups $J$ of $G(\bA)$ we have
    \[
    w(x) \sim w(xk), \quad \text{for all $x \in [G]_P$, $k \in J$}.
    \]
If $G$ is reductive, which is the case considered in~\cite{BPCZ}, the two
definitions are equivalent by~\cite{BPCZ}*{Lemma~2.4.3.1}. In general our
definition imposes a stronger condition. Indeed
exponential functions on affine spaces satisfy the condition in~\cite{BPCZ}
while our condition rules them out.
\end{remark}

\subsection{Space of functions} \label{subsec:spaces_of_function}
Let $G$ be a connected linear algebraic group over $F$.

\subsubsection{Representations}

Following~\cites{BK,BPCZ}
we introduce certain nice categories of representations. We first consider
representations of $G(F_\infty)$. A Fr\'{e}chet representation $V$ of
$G(F_\infty)$ is called an F-representation if its topology is induced by a
countable family of $G(F_\infty)$-continuous $\aabs{\cdot}_{\infty}$-bounded
semi-norms. Here a semi-norm $\nu$ on $V$ is called
$G(F_\infty)$-continuous if the map
    \[
    G(F_\infty) \times (V, \nu) \to (V, \nu)
    \]
is continuous, where $(V, \nu)$ is the vector space $V$ endowed with the
topology induced from $\nu$. The semi-norm $\nu$ is called
$\aabs{\cdot}_\infty$-bounded if there is a positive number $A$ such that
    \[
    \sup_{v \in V,\, \nu(v) \not=0} \frac{\nu(g\cdot v)}{\nu(v)}
    \leq \aabs{g}_\infty^A
    \]
for all $g \in G(F_\infty)$. By~\cite{BK}*{Lemma~2.10}, the Fr\'{e}chet
representation $V$ is an F-representation if and only if it is of
moderate growth, i.e.
if for any continuous semi-norm $\nu$ on $V$ there exists another continuous
semi-norm $\nu'$ on $V$ and a positive number $A$ such that
    \[
    \nu(g \cdot v) \leq \aabs{g}_\infty^A \nu'(v)
    \]
for all $g \in G(F_\infty)$ and $v \in V$. An F-representation $V$
of $G(F_\infty)$ is called smooth, or is said to be an SF-representation if
for each $v \in V$ the map
    \[
    G(F_\infty) \to V, \quad g \mapsto g \cdot v
    \]
is smooth, and if for every $X \in \cU(\fg_{\infty})$ the resulting map
    \[
    V \to V, \quad v \mapsto X \cdot v
    \]
is continuous.

\begin{remark}
The $\aabs{\cdot}_\infty$-bounded condition is not included
in~\cite{BPCZ}*{Section~2.5.3}. This is because only reductive groups are
considered there, where $\aabs{\cdot}_{\infty}$ is equivalent to the
``maximal scale'', cf.~\cite{BK}*{Section~2.1}, and thus the condition of
$\aabs{\cdot}_\infty$-boundedness is automatic. We however need to work with
representations
of nonreductive groups, and $\aabs{\cdot}_\infty$-boundedness is crucial.
\end{remark}

We now consider representations of $G(\bA)$. An SLF-representation of $G(\bA)$
is a vector space $V$ equipped with a $G(\bA)$-action, with the following
properties.
\begin{itemize}
\item For each open compact subgroup $J$ of $G(\bA_f)$ the space $V^J$ is an
    SF-representation of $G(F_\infty)$.

\item We have
    \[
    V = \bigcup_{J \subset G(\bA_f)} V^J,
    \]
    where $J$ runs over all open compact subgroups of $G(\bA_f)$.

\item If $J' \subset J$ then the natural map $V^J \to V^{J'}$ is a closed embedding.
\end{itemize}

\subsubsection{Smooth functions and Schwartz space}

We now let $Z \subset G$ be a
central unipotent subgroup.

Let $\psi:Z(F) \backslash Z(\bA) \to \C^\times$ be a character. We define
$C^\infty(G(\bA),\psi)$ to be the vector space of functions $f:G(\bA) \to \C$ such
that
\begin{itemize}
    \item $f$ is right invariant under a compact open subgroup $J \subset G(\bA_{f})$,
    \item for all $g_f \in G(\bA_f)$, the function $g_\infty \mapsto f(g_fg_\infty)$ is a smooth function on the Lie group $G(F_\infty)$,
    \item $f(zg)=\psi(z)f(g)$ for all $z \in Z(\bA)$ and $g \in G(\bA)$.
\end{itemize}
Remark that we do not require
$f$ to be left invariant under some compact open subgroup of $G(\bA)$. We denote by $C_c^\infty(G(\bA),\psi)$ the subspace of compactly supported functions modulo $Z(\bA)$. If $\psi$ is trivial we omit it from the notation
in all function spaces.

Choose a height function $\aabs{\cdot}$ on $G(\bA)$ which is the pullback of a height function on $(G/Z)(\bA)$, as described in Subsection~\ref{subsubsec:central_unip}. For a compact open subset $C$ of $G(\bA_{f})$ and a
compact open subgroup $J$ of $G(\bA_{f})$, we denote by
$\cS(G(\bA),C,J,\psi)$ the subspace of $C^\infty(G(\bA), \psi)$ consisting
of functions $f$ that are bi-invariant under $J$, supported
in $G(F_\infty) \times CZ(\bA)$ and such that for all $N>0$ and $X,Y \in
\cU(\fg_\infty)$ the semi-norm
    \[
    \|f\|_{X,Y,N} := \sup_{g \in G(\bA)}  \aabs{g}^N
    \abs{\mathrm{R}(X)\mathrm{L}(Y)f(g)}
    \]
is finite. This family of semi-norms gives $C^\infty(G(\bA),\psi)$ the structure of a Fr\'{e}chet space. Define the space of Schwartz functions $\cS(G(\bA),\psi)$ to be the union
of all such $\cS(G(\bA),C,J,\psi)$ as $C$ and $J$ vary, endowed with the natural topology of an LF space.

\begin{remark} \label{rmk:Schwartz_space_equiv_norm}
    Note that $\mathrm{L}(Y)$ and $\mathrm{R}(Y)$ differ by the adjoint action, which is an algebraic representation on $G$. Therefore, the topology of $\cS(G(\bA),C,J,\psi)$ is also generated by the semi-norms $\| \cdot \|_{X,N}$ where
    \[
         \|f\|_{X,N} := \sup_{g \in G(\bA)}  \aabs{g}^N
    \abs{\mathrm{R}(X) f(g)} .
    \]
\end{remark}

\subsubsection{The subgroup $\Gamma$}
\label{subsubsec:gamma}
Let $\Gamma$ be a subgroup of $G(\bA)$ such that its intersection with $Z(\bA)$ equals $Z(F)$. We let $C^\infty(\Gamma \bs G(\bA), \psi)$ be the subspace of $C^\infty(G(\bA), \psi)$ of left $\Gamma$-invariant functions. For any open compact subgroup $J \subset G(\bA_f)$, we denote by $C^\infty(\Gamma \bs G(\bA), \psi)^J$ its subspace of right $J$-invariant functions.

We will say that $\Gamma$ satisfies the condition (SL) (for sufficiently large) if $\Gamma$ contains $P(F)$ for some D-parabolic $P$. We will assume that $\Gamma$ satisfies this condition for the rest of this section.

In later text, the group $G$ will be a reductive group or a Jacobi group (see Section ~\ref{sec:Jacobi_groups}). In these cases, if $\Gamma$ satisfies the property (SL) then for any open compact subgroup
$J \subset G(\bA_f)$, there exists an open compact subset $C \subset G(\bA_f)$ such
that the support of any $f \in C^\infty(P(F) \bs G(\bA), \psi)^J$ is contained in
$\Gamma (C \times G(F_\infty))$. Let $\cC$ be a set of representatives of $P(F)\bs G(\bA_f)/J$.
Then we see that there is a finite subset $\cC_0$ of $\cC$ such that $C^\infty(P(F) \bs G(\bA), \psi)^J$ can be identified with a space of smooth functions on $\cup_{\xi \in \cC_0}(P(F) \cap \xi J \xi^{-1}) \bs G(F_{\infty})$, and $C^\infty(\Gamma \bs G(\bA), \psi)^J$ is a subspace of it. When $G$ is reductive, the existence of the compact subset $C$ follows from finiteness of class number.

We slightly extend the definition of heights and weights to the case of $\Gamma \bs G(\bA)$ where $\Gamma$ is as above. We define a height $\| \cdot \|_{\Gamma}$ on $\Gamma
\backslash G(\bA)$ by
    \[
        \| g \|_\Gamma := \inf_{\gamma \in \Gamma} \| \gamma g \|.
    \]
Note that $\| g \|_\Gamma \ge 1$ for all $g$. We define a weight on $\Gamma \backslash G(\bA)$ to be a positive measurable function on $\Gamma \backslash G(\bA)$
such that there exist a positive number $N$ and a constant $C$ such that
    \[
    w(xg) \leq C w(x) \aabs{g}_P^N, \quad
    \text{for all $x \in [G]_P$ and $g\in G(\bA)$}.
    \]
When $\Gamma = P(F)$ or $M(F)N(\bA)$, these definitions coincide with the definition of weight on $[G]_P$ given in ~\eqref{eq:weight_on_[G]_P}.

\subsubsection{Spaces of Schwartz functions}

We keep $\Gamma$ to be as in Subsection~\ref{subsubsec:gamma}, so that it satisfies condition (SL).

Let $\cS^0(\Gamma \backslash G(\bA),\psi)$ be the space of measurable
functions on $\varphi$ on $\Gamma \backslash G(\bA)$ such that
$\varphi(zx)=\psi(z) \varphi(x)$ for almost all $(z,x) \in [Z] \times \Gamma
\backslash G(\bA)$, and such that for any $N>0$
    \[
    \| \varphi \|_{\infty,N} :=
    \sup_{x \in \Gamma \backslash G(\bA)} \|x\|_{\Gamma}^N |\varphi(x)|
    < \infty.
    \]
It is equipped with the family of seminorms $\| \cdot \|_{\infty,N}$
which gives it the structure of Fr\'{e}chet space. Let $\cS^{00}(\Gamma
\backslash G(\bA),\psi)$ be the closed subspace of $\cS^0(\Gamma \backslash
G(\bA),\psi)$ consisting of continuous functions.

Let $\cS(\Gamma \backslash G(\bA),\psi)$ be the space of \emph{Schwartz
functions on $\Gamma \backslash G(\bA)$}. It is the subspace of
$C^\infty(\Gamma \backslash G(\bA),\psi)$ consisting of $\varphi$ satisfying
that for all $X \in \cU(\fg_\infty)$
and integers $N>0$,
    \[
    \aabs{\varphi}_{X,N,\infty} :=
    \sup_{g \in \Gamma \backslash G(\bA)} \aabs{g}_{\Gamma}^N
    \abs{\mathrm{R}(X)\varphi(g)}
    < \infty.
    \]
For each compact open subgroup $J \subset G(\bA_{f})$, $\cS(\Gamma
\backslash G(\bA),\psi)^J$ is a Fr\'{e}chet space under the seminorms $\| \cdot
\|_{X,N,\infty}$, and $\cS(\Gamma \backslash G(\bA),\psi)$ is an SLF representation of $G(\bA)$ for the right translation $\mathrm{R}$.

Let $w$ be a weight on $\Gamma \bs G(\bA)$. For each $N \ge 0$, let $\cS_{w, N}(\Gamma \backslash G(\bA), \psi)$ be the
space of smooth functions $\varphi \in C^\infty(\Gamma \backslash
G(\bA),\psi)$ satisfying
    \[
    \aabs{\varphi}_{\infty, -N, w^r, X} :=
    \sup_{g \in \Gamma \backslash G(\bA)} \aabs{g}_{\Gamma}^{-N} w(g)^{r}
    \abs{\mathrm{R}(X) \varphi(g)} < \infty
    \]
for all $r \geq 0$ and $X \in \cU(\fg_{\infty})$. For each open compact
subgroup $J$ of $G(\bA_{f})$, the space $\cS_{w, N}(\Gamma \backslash
G(\bA), \psi)^J$ is a Fr\'{e}chet space, and hence $\cS_{w, N}(\Gamma
\backslash G(\bA), \psi)$ is a strict LF space. Put
    \[
    \cS_{w}(\Gamma \backslash G(\bA), \psi)
    = \bigcup_{N \geq 0} \cS_{w, N}(\Gamma \backslash G(\bA),
    \psi),
    \]
which is naturally a non-strict LF space. This is the $\psi$-equivariant \emph{weighted Schwartz
space}. In particular, when $w=\| \cdot \|_\Gamma$, we recover the definition
of $\cS(\Gamma \backslash G(\bA),\psi)$.

\subsubsection{Space of functions of uniform moderate growth}

We retain the subgroup $\Gamma$ from the previous subsection. Fix a weight $w$ on $\Gamma \backslash G(\bA)$, and denote by $\cT^0_w(\Gamma
\backslash G(\bA),\psi)$ the space of complex Radon measures $\varphi$ on
$\Gamma \backslash G(\bA)$ with the properties that
    \[
    \int_{\Gamma \backslash G(\bA)} f(g) \varphi(g)  =
    \psi(z) \int_{\Gamma \backslash G(\bA)} f(zg) \varphi(g)
    \]
for any $f \in C_c(\Gamma \backslash G(\bA))$ and $z
\in [Z]$, and moreover that
    \[
        \aabs{ \varphi }_{1,w} :=
     \int_{\Gamma \backslash G(\bA)} w(g) \abs{ \varphi(g)} < \infty.
    \]
The space $\cT_w^0(\Gamma \backslash G(\bA),\psi)$ is naturally a Banach
space with the norm $\aabs{\cdot}_{1,w}$. We write $\cT^0_N(\Gamma \backslash G(\bA),\psi)$ for the space
$\cT^0_{\| \cdot \|^{-N}_\Gamma}(\Gamma \backslash G(\bA),\psi)$,
and set
    \[
    \cT^0(\Gamma \backslash G(\bA),\psi) :=
    \bigcup_{N>0} \cT^0_{N}(\Gamma \backslash G(\bA),\psi),
    \]
which is naturally an LF space.

We set
 \begin{equation*}
     \cT(\Gamma
\backslash G(\bA),\psi) := \cS_1(\Gamma \backslash G(\bA),\psi)=\bigcup_{N>0} \cT_N(\Gamma \backslash G(\bA),\psi),
 \end{equation*}
 where $\cT_N(\Gamma \backslash G(\bA), \psi):=\cS_{1,N}(\Gamma \backslash G(\bA),\psi)$.
It is the
\emph{space of $\psi$-equivariant functions of uniform moderate growth}. It is equipped with the corresponding LF topology.

There is a natural pairing
    \begin{equation}
    \label{eq:natural_pairing}
         \cS^{00}(\Gamma \backslash G(\bA),\psi)
    \times \cT^0(\Gamma \backslash G(\bA),\psi) \to \C,
    \quad (\varphi,\varphi') \mapsto
   \langle \varphi,\varphi' \rangle= \int_{\Gamma \backslash G(\bA)} \varphi(x) \overline{\varphi'(x)} \rd x.
    \end{equation}
It identifies $\cT^0(\Gamma \backslash G(\bA),\psi)$
 with the topological dual of $\cS^{00}(\Gamma \backslash G(\bA),\psi)$.

If $\varphi \in \cT^0(\Gamma \bs G(\bA), \psi)$ and $f\in \cS(G(\bA),\psi^{-1})$,
we define a function $\mathrm{R}(f)\varphi$ on $G(\bA)$ by
    \[
    \mathrm{R}(f)\varphi(x) =
    \int_{Z(\bA)\backslash G(\bA)} f(x^{-1}y) \varphi(y).
    \]

\begin{lemma}   \label{lem:translation_weighted_Schwartz}
Let $f \in \cS(G(\bA),\psi^{-1})$, and $\varphi \in
\cT^0(\Gamma \bs G(\bA), \psi)$. Let $w$ be a weight on $\Gamma \bs G(\bA)$. If $w$ is bounded on the
support of $\varphi$, then $\mathrm{R}(f) \varphi \in
\cS_w(\Gamma \bs G(\bA), \psi)$.
\end{lemma}

\begin{proof}
The proof is the same as~\cite{BPCZ}*{Lemma~2.5.1.1}.
\end{proof}

\subsubsection{Spaces of weighted $L^2$-functions on reductive groups}
\label{subsubsec:L2}
We now assume that $G$ is reductive and let $P$ be a standard parabolic subgroup of $G$. By taking $\Gamma=M_P(F)N_P(\bA)$ and $Z=1$, we get various spaces of functions on $[G]_P$, including $\cS([G]_P), \cT([G]_P)$, etc. These spaces are related to smooth vectors in weighted $L^2$ spaces as we now explain.

The space $[G]_P=M_P(F) N_P(\bA) \backslash G(\bA)$ is equipped with the quotient of the Tamagawa measure on $G(\bA)$ by the product of the counting measure on $M_P(\bA)$ with the Tamagawa measure on $N_P(\bA)$ (see Subsection~\ref{subsubsec:measures}). We denote by $L^2([G]_P)$ the space of $L^2$ functions on $[G]_P$. It is a Hilbert space when equipped with the scalar product
\begin{equation*}
    \langle \varphi_1, \varphi_2 \rangle_P = \int_{[G]_P} \varphi_1(g) \overline{\varphi_2(g)} \rd g.
\end{equation*}
More generally, if $w$ is a weight on $[G]_P$, we write $L^2_w([G]_P)$ for the Hilbert space of square-integrable functions with respect to the measure $w(g) \rd g$. We denote the resulting norm $\aabs{\cdot}_{w,P,L^2}$. If $N \in \R$, we will simply write $L^2_N([G]_P)$ for $L^2_{\aabs{\cdot}_P^N}([G]_P)$. It is equipped with a continuous (non-unitary) representation $\mathrm{R}$ of $G(\bA)$ by right-translation.

The subspace $L^2_w([G]_P)^\infty$ of smooth vectors consist of smooth functions $\varphi : [G]_P \to \C$ such that for every $X \in \cU(\fg_\infty)$ we have $\mathrm{R}(X) \varphi \in L^2_w([G]_P)$. It is equipped with the family of semi-norms $\aabs{R(X)f}_{w,P,L^2}$. For every open compact subgroup $J \subset G(\bA_f)$, the subspace $(L^2_w([G]_P)^\infty)^J$ is a Frechet space, and we equip $L^2_w([G]_P)^\infty=\bigcup_J (L^2_w([G]_P)^\infty)^J$ with the corresponding structure of LF-space.

By the Sobolev inequality (\cite{BPCZ}*{(2.5.5.4)}) and the open mapping theorem, we have an equality of SLF representations (where the right-hand side is equipped with the locally convex projective limit topology)
\begin{equation*}
    \cS([G]_P) = \bigcap_{N>0} L^2_N([G]_P)^\infty.
\end{equation*}
Moreover, for every weight $w$ on $[G]_P$, $\cS([G]_P)$ is dense in $L^2_w([G]_P)^\infty$.

By another use of the Sobolev inequality, we also have the equality of LF spaces
\begin{equation*}
    \cT([G]_P)=\bigcup_{N>0}  L^2_{-N}([G]_P)^\infty.
\end{equation*}
It follows that $\cS([G]_P)$ is dense in $\cT([G]_P)$ (but it is not dense in any $\cT_N([G]_P)$ in general). Finally, for every $N>0$ there exist $N'>0$ and a continuous inclusion $L^2_{-N}([G]_P)^\infty \subset \cT_{N'}([G]_P)$.

\subsection{Pseudo-Eisenstein series and constant terms}

Let $G$ be a connected linear algebraic group over $F$ that satisfies~\eqref{eq:(SR)}. Recall that in Subsection~\ref{subsubsec:D-parabolic} we have fixed a Levi
decomposition $G= U \rtimes H$, a maximal split torus $A_0 \subset H$ and a
minimal parabolic subgroup $P_0$ of $H$ which contains $A_0$. Let $P,Q$ be two
semistandard D-parabolic subgroups of $G$ with $P \subset Q$. We keep $Z$ to be a central unipotent subgroup of $G$, and $\psi$ a character of $[Z]$. Note that $M_P(F)N_P(\bA) \cap Z(\bA)=Z(F)$. We may therefore take $\Gamma = M_P(F) N_P(\bA)$ or $\Gamma = P(F)$
in the definitions of the various function spaces of Subsection~\ref{subsec:spaces_of_function}. In particular, we get spaces of $\psi$-equivariant functions on $P(F) \backslash G(\bA)$ or $[G]_P$. The same holds for $Q$.

Like in the case of reductive groups, we have
the construction of pseudo-Eisenstein series and the constant terms, which we
now explain. For $\varphi \in \cS^0(P(F) \bs G(\bA),\psi)$, we define a
pseudo-Eisenstein series
    \[
    E_P^Q \varphi(g) = \sum_{\gamma \in P(F) \backslash Q(F)} \varphi(\gamma g).
    \]

\begin{lemma}   \label{lemma:convergence_of_pseudo_eisenstein}
The pseudo-Eisenstein series $\varphi \mapsto E_P^Q \varphi$ defines a
continuous linear map
    \[
    E_P^Q: \cS^{?}(P(F) \bs G(\bA),\psi) \to \cS^?(Q(F) \bs G(\bA),\psi),
    \]
where $? = 0$, $00$, or empty.
\end{lemma}

Note that by restriction we also obtain a continuous linear map $E_P^Q : \cS^{?}([G]_P,\psi) \to \cS^?([G]_Q,\psi)$ for $? = 0$, $00$, or empty.

\begin{proof}
It is more convenient to use the height function $\| \cdot \|'$ on $G(\bA)$ instead of the one pulled back from $G/Z(\bA)$ (see Remark ~\ref{rem:[G]_P_height_equivalent}). So $\| \cdot \|$ in the proof below will denote the height function on $G(\bA)$. Let $d_1, d_2>0$. For every $\gamma \in Q(F)$ we have $\aabs{\gamma}_P^{-d_2} \ll \sum_{\delta \in P(F)} \aabs{ \delta \gamma}^{-d_2}$. Moreover, it follows from~\cite{BP1}*{Theorem A.1.1(1)} that there exists $c>0$ such that $\aabs{xy}^c \ll \aabs{x} \aabs{y}$. Therefore for every $\varphi \in \cS^0(P(F) \bs G(\bA),\psi)$ and $g \in G(\bA)$ we have

\begin{align*}
   \sum_{\gamma \in P(F) \backslash Q(F)} \left|\varphi(\gamma g)  \right| \ll\sum_{\gamma \in P(F) \backslash Q(F)} \aabs{\gamma g}_P^{-d_1-d_2}  &\ll \aabs{g}^{-d_2}_Q  \sum_{\gamma \in P(F) \backslash Q(F)} \aabs{\gamma g}_P^{-d_1}  \\
   & \ll \aabs{g}^{-d_2}_Q  \sum_{\gamma \in Q(F)} \aabs{\gamma g}^{-d_1} \\
    & \ll \aabs{g}^{-d_2}_Q  \aabs{g}^{d_1} \sum_{\gamma \in Q(F)} \aabs{\gamma}^{- cd_1}.
\end{align*}
By ~\cite{BP1}*{Proposition A.1.1~(v)}, the sum $\sum_{\gamma \in Q(F)}\|\gamma\|^{-cd_1}$ is finite for $d_1$ large enough. As the LHS only depends on the class of $g$ in $Q(F) \backslash G(\bA)$, we see that
\begin{equation*}
    \sum_{\gamma \in P(F) \backslash Q(F)} \left|\varphi(\gamma g)  \right| \ll  \aabs{\varphi}_{d_1+d_2,\infty} \aabs{g}_Q^{-d_2+d_1}, \; \varphi \in \cS^0(P(F) \bs G(\bA),\psi).
\end{equation*}
By taking $d_2$ large enough, this shows that $E_P^Q$ sends $\cS^0(P(F) \bs G(\bA),\psi)$ to $\cS^0(Q(F) \bs
G(\bA),\psi)$, and that it is continuous. By a similar estimate we see that $E_P^Q$ restricts a to continuous
map from $\cS^{?}(P(F) \bs G(\bA),\psi)$ to $\cS^{?}(Q(F) \bs G(\bA),\psi)$,
where $? = 00$ or empty.
\end{proof}

\begin{remark}
If $G$ is reductive and $\varphi \in \cS^0([G]_P)$ then we have $E_P^Q\varphi
\in \cS^0([G]_Q)$. This is the usual definition of pseudo-Eisenstein series
in~\cite{BPCZ}*{Section~2.5.13}.
\end{remark}

We now define the constant terms. If $\varphi$ is a Radon measure on $P(F)
\bs G(\bA)$, we define its constant term $\varphi_P$ to be its pushforward to
$[G]_P$ along the projection map $P(F) \bs G(\bA) \to [G]_P$. Since the
fibers of $P(F) \bs G(\bA) \to [G]_P$ are compact, if $\varphi \in \cT^0(P(F)
\bs G(\bA), \psi)$, then $\varphi_P \in \cT^0([G]_P, \psi)$. Note that this map is
the transpose of the pullback of functions $\cS^{00}([G]_P, \psi^{-1}) \to
\cS^{00}(P(F) \bs G(\bA), \psi^{-1})$ under the pairing \eqref{eq:natural_pairing}, where we use the compactness of the fibers again to
ensure that a function on $[G]_P$ of rapid decay is of rapid decay on
$P(F) \bs G(\bA)$. On the subspace $\cT(P(F) \bs G(\bA),\psi)$, the constant
term is given by
    \[
    \varphi_P(x) = \int_{[N_P]} \varphi(nx) \rd n,
    \]
and $\varphi_P \in \cT([G]_P, \psi)$. Moreover $\varphi \mapsto \varphi_P$
induces continuous linear maps $\cT(P(F) \bs G(\bA), \psi) \to \cT([G]_P,
\psi)$ and $\cT_N(P(F) \bs G(\bA), \psi) \to \cT_N([G]_P, \psi)$ for all $N$.

By Lemma~\ref{lemma:convergence_of_pseudo_eisenstein}, the pseudo-Eisenstein
series $E_P^Q$ gives a continuous linear map $\cS^{00}(P(F) \bs G(\bA), \psi^{-1})
\to \cS^{00}(Q(F) \bs G(\bA), \psi^{-1})$. Its transpose gives a continuous
linear map
    \[
    \cT^0(Q(F) \bs G(\bA), \psi) \to \cT^0(P(F) \bs G(\bA),\psi).
    \]
Therefore we have a series of continuous linear maps
    \begin{equation}
    \label{eq:constant_term_defi}
           \cT^0([G]_Q, \psi) \to \cT^0(Q(F) \bs G(\bA), \psi) \to
    \cT^0(P(F) \bs G(\bA), \psi) \to \cT^0([G]_P, \psi).
    \end{equation}
We denote the composition also by $\varphi \mapsto \varphi_P$ and again call
it the constant term. If $G$ is reductive, this is the definition of the
constant term given in~\cite{BPCZ}*{Section~2.5.13}. The composition in \eqref{eq:constant_term_defi} restricts to a continuous map $\cT([G]_Q,\psi) \to \cT([G]_P,\psi)$.

By construction we have
    \begin{equation}    \label{eq:adjuction_pseudo_constant}
    \langle \varphi, E_P^Q \varphi' \rangle =
    \langle \varphi_P, \varphi' \rangle, \quad
    \varphi \in \cS^0([G]_Q, \psi^{-1}), \ \varphi' \in
    \cT^0([G]_P, \psi),
    \end{equation}
where $\langle-,-\rangle$ stands for the pairing between $\cS^{00}$ and
$\cT^0$ defined in \eqref{eq:natural_pairing}.

\subsection{Reductive groups, reduction theory}
\label{subsec:reduction_theory}

In this section, we let $G$ be a connected reductive group over $F$. 

\subsubsection{Main notations}
Let $X^*(G)$ be the group of rational characters of $G$. Put $\fa_G = \Hom_{\Z}(X^*(G), \R)$ and $\fa_G^* = X^*(G) \otimes_{\Z} \R$. We
have a canonical pairing
    \[
    \langle-,-\rangle: \fa_G^* \times \fa_G \to \R.
    \]
Let $H_G: G(\bA) \to \fa_G$ the Harish-Chandra map, with the defining
property that for any $\chi \in X^*(G)$, we have
\begin{equation}
\label{eq:HC_defi}
     \log\abs{\chi(g)} = \langle  \chi, H_G(g) \rangle.
\end{equation}
Recall that we have fixed a maximal split torus $A_0$ of $G$ and a minimal parabolic $P_0=M_0 N_0$. Let $P = MN$ be a standard parabolic subgroup and $K$ be a maximal compact subgroup of $G(\bA)$ in good position with respect to $P_0$, so that we have the Iwasawa decomposition $G(\bA) = P(\bA)K$. We put $\fa_P = \fa_M$ and $\fa_P^* =
\fa_M^*$. For any $mnk \in G(\bA)$ with $m \in M(\bA)$, $n \in N(\bA)$ and $k \in K$, set
\begin{equation*}
    H_P(mnk) = H_M(m) \in \fa_P,
\end{equation*}
where $H_M$ is the Harish-Chandra map defined in \eqref{eq:HC_defi} with respect to the reductive group $M$. We define $G(\bA)^1$ to be the kernel of $H_G$. More generally, we define $G(\bA)_P^1$ be the fiber of $0$ of $H_P$, and it descends to a subset $[G]_P^1$ of $[G]_P$. We also put $P(\bA)^1 := P(\bA) \cap G(\bA)^1_P$, which is a subgroup of $P(\bA)$.

We put $\fa_0 = \fa_{P_0}$, $\fa_0^* = \fa_{P_0}^*$ and $H_0 = H_{P_0}$ to shorten notation.

Let $P = MN$ be a parabolic subgroup. We denote by $\delta_P: P(\bA) \to \R_{>0}$ the modulus character, and by $\rho_P$ the half sum of roots of $A_P$ in $N$. Then $\rho_P \in \fa_P^*$ and satisfies
    \[
    \delta_P(m) = e^{\langle 2\rho_P, H_M(m) \rangle}, \quad m \in M(\bA).
    \]

Let $P \subset Q$ be parabolic subgroups. Then we have $A_Q \subset A_P$. The restriction $X^*(Q) \to X^*(P)$ induces
the maps $\fa_Q^* \to \fa_P^*$ (which is an injection) and $\fa_P \to \fa_Q$, whose kernel is denoted by $\fa_P^{Q}$. The restriction $X^*(A_P) \to X^*(A_Q)$ induces the maps $\fa_Q \to \fa_P$ (which is an injection) and $\fa_P^* \to \fa_Q^*$, whose kernel is denoted by $\fa_P^{Q, *}$. We have canonical
decompositions
    \[
    \fa_P = \fa_Q \oplus \fa_P^Q, \quad  \fa_P^* = \fa_Q^* \oplus \fa_P^{Q, *}.
    \]
Define $\Delta_P^Q \subset \fa_P^{Q, *}$ to be the set of simple roots of
$A_P$ in $M_Q \cap P$. We have the set of coroots $\Delta_P^{Q, \vee} \subset
\fa_P^Q$. By duality, we also have the set of simple weights
$\widehat{\Delta}_P^Q \subset \fa_P^{Q, *}$. The sets $\Delta_P^Q$ and
$\widehat{\Delta}_P^Q$ define open cones in $\fa_0$ whose characteristic
functions are denoted by $\tau_P^Q$ and $\widehat{\tau}_P^Q$ respectively. If
$P = P_0$ then we replace the subscript $P_0$ in the notation by $0$. If $Q =
G$ then we omit the superscript $G$.

Let $W$ be the Weyl group of $(G,A_0)$, that is the quotient of the normalizer of $A_0$ in $G(F)$ by $M_0$ the Levi subgroup of $P_0$. For $P=M_P N_P$ and $Q=M_Q N_Q$ two standards parabolic subgroups of $G$, denote by $W(P,Q)$ the set $w \in W$ such that $w \Delta_0^P=\Delta_0^Q$. In particular, for $w \in W(P,Q)$ we have $wM_P=M_Q$.

\subsubsection{Haar measures}
\label{subsubsec:measures_Lie}
We equip $\fa_P$ with the Haar measure
giving the lattice $\Hom(X^*(P),\Z)$ covolume $1$, and $i\fa_P^*$ with the
dual measure. If $Q \supset P$ is another parabolic subgroup, then $\fa_P^Q
\cong \fa_P / \fa_Q$ and $i\fa_P^{Q,*} \cong i\fa_P^* / \fa_Q^*$ are equipped
with the quotient Haar measures.

\subsubsection{Truncation parameters}

The fixed minimal parabolic subgroup $P_0$ determine a positive chamber in
$\fa_0$. By ``$T \in \fa_0$ is sufficiently positive'' we mean that ``for $T$
such that $\inf_{\alpha \in \Delta_0} \alpha(T) \geq \max\{\epsilon \aabs{T},
C\}$'', where $\aabs{\cdot}$ is an arbitrary norm on the real vector space
$\fa_0$, $C> 0$ is a large enough constant and $\epsilon>0$ is an arbitrary
(but in practice small enough) constant. We say that $T \in \fa_0$ is sufficiently
negative if $-T$ is sufficiently positive.

For $T \in \fa_0$ and a standard parabolic subgroup $P$, we write $T_P$ for the image of $T$ under the projection map $\fa_0 \to \fa_P$. If $P$ is more generally semistandard, we write $T_P$ for the image of $w \cdot T$ under the projection map $\fa_0 \to \fa_P$, where $w$ is any element in the Weyl group satisfies $wP_0w^{-1} \subset P$.

\subsubsection{Reduction theory}

Let $\omega_0 \subset P_0(\bA)^1$ be a compact subset such that $P_0(\bA)^1
= \omega_0 P_0(F)$. Let $P=MN$ be a standard parabolic subgroup. By a Siegel domain $\fs^P$ of $[G]_P$ we mean a subset of
$[G]_P$ of the form
    \[
    \fs^P = \omega_0 \{ a \in A_0^\infty \mid
    \langle \alpha, H_0(a) - T_- \rangle \geq 0, \ \alpha \in \Delta_0^P \} K,
    \]
where $T_- \in \fa_0$, and such that $G(\bA) = M(F)N(\bA)\fs^P$. We
assume that for different parabolic subgroups of $G$, their Siegel domains
are defined by the same $T_-$. In particular if $P \subset Q$ then $\fs^P
\supset \fs^Q$.

For $T,T_- \in \fa_0$, we define
    \[
    A_0^{P,\infty}(T_-,T) =
    \left\{ a \in A_0^\infty \mid \langle \alpha,H_0(a) \rangle
    \ge \langle \alpha,T_- \rangle, \forall \alpha \in \Delta_0^P
    \text{ and } \langle \varpi,H_0(a) \rangle \le \langle \varpi,T \rangle,
    \forall \varpi \in \widehat{\Delta}_0^P \right\}.
    \]
 We put $\fs^P(T) = \fs^P(T_-,T,\omega_0,K) = \omega_0
A_0^{P,\infty}(T_-,T)K$ and call it the truncated Siegel set. Let
$F^P(\cdot,T)$ be the characteristic function on $[G]_P$ of the set
$M_P(F)N_P(\bA) \fs^P(T)$. For $T$ sufficiently positive and $T_-$
sufficiently negative, we have the Langlands partition formula \cite{Arthur3}*{Lemma 6.4}:
    \begin{equation}    \label{eq:langlands_partition}
    \sum_{P_0 \subset Q \subset P} \sum_{\gamma \in Q(F) \backslash P(F)}
    F^Q(\gamma x,T) \tau_Q^P(H_0(\gamma x)-T)=1.
    \end{equation}

Let $\lambda \in \fa_0^*$ and $P$ be a parabolic subgroup. A weight $d_{P, \lambda}$ on $[G]_P$ is introduced in~\cite{BPCZ}*{Section~2.4.3}. If $g \in \fs^P$, then we have
    \[
    d_{P, \lambda}(g) \sim e^{\langle \lambda, H_P(g) \rangle}.
    \]
If $P \subset Q$ are parabolic subgroups, two weights $d_P^Q$ and $d_Q^P$ on $[G]_P$ and $[G]_Q$ respectively are introduced in~\cite{BPCZ}*{Section~2.4.4} and are given by
    \[
    d_{P}^Q(g) = \min_{\lambda \in \Psi_P^Q}
    d_{P, \lambda}(g), \quad
    d_{Q}^P(g) = \min_{\lambda \in \Psi_P^Q}
    d_{Q, \lambda}(g),
    \]
where $g \in [G]_P$ and $[G]_Q$ respectively. Consider the projections
    \begin{equation}    \label{eq:projections_P_Q}
    [G]_P \xleftarrow{\pi_P^Q}
    P(F) N_Q(\bA) \bs G(\bA) \xrightarrow{\pi_Q^P} [G]_Q.
    \end{equation}
For $C>0$ define
    \[
    \omega_P^Q[>C] = \{ g \in P(F) N_Q(\bA) \bs G(\bA)
    \mid d_P^Q(\pi_P^Q(g)) > C\}.
    \]

We recall~\cite{BPCZ}*{Lemma~2.4.4.1} which summarizes some classical results
from reduction theory.

\begin{lemma}   \label{lemma:classical_reduction_theory}
We have the following assertions.
\begin{enumerate}
\item There is an $\epsilon>0$ such that $\pi_Q^P$ maps
    $\omega_P^Q[>\epsilon]$ onto $[G]_Q$.

\item For any $\epsilon>0$, we have
    \[
    d_P^Q(g) \sim d_Q^P(g), \quad \aabs{g}_P \sim \aabs{g}_Q
    \]
    for all $g \in \omega_P^Q[>\epsilon]$.

\item For all $\epsilon>0$, the restriction of $\pi_Q^P$ to
    $\omega_P^Q[>\epsilon]$ has uniformly bounded fibers.

\item For all $\epsilon>0$, there is a $C>0$ such that if $(g_1, g_2) \in
    \omega_P^Q[>\epsilon] \times \omega_P^Q[>C]$ and $\pi_Q^P(g_1) =
    \pi_Q^P(g_2)$, then $g_1 = g_2$.

\item The map $\pi_Q^P$ is a local homeomorphism that locally preserves the
    measures. The map $\pi_P^Q$ is proper and the pushforward of the
    invariant measure on $P(F) N_Q(\bA) \bs G(\bA)$ by it is the
    invariant measure on $[G]_P$.
\end{enumerate}
\end{lemma}

\subsection{Automorphic forms and Eisenstein series}
\label{subsec:automorphic_form}

We assume that $G$ is reductive and connected in this subsection.

\subsubsection{Spaces of automorphic forms}
Let $P = MN$ be a standard parabolic subgroup of $G$. Equip $A_P^\infty$ with the Haar measure $\rd a$ such that the isomorphism $H_P: A_P^\infty \to \fa_P$ is measure preserving (see Subsection~\ref{subsubsec:measures_Lie}). Set $[G]_{P,0}=A_P^{\infty} \backslash [G]_P$. It is equipped with the quotient of the Tamagawa invariant measure on $[G]_P$. More precisely, it is the right-invariant functional on the space of continuous functions $f$ on $[G]_P$ such that $f(ag)=\delta_P(a)f(g)$ for any $(a,g) \in A_P^{\infty} \times [G]_P$, such that for any $f \in C_c([G]_P)$ we have
\begin{equation*}
  \int_{[G]_{P,0}}  \int_{A_P^{\infty}} f(ag) \rd a \rd g = \int_{[G]_P} f(g) \rd g.
\end{equation*}

We define the space of
automorphic forms $\cA_P(G)$ to be the subspace of $\cZ(\fg_\infty)$-finite
functions in $\cT([G]_P)$. We define $\cA_{P, \mathrm{cusp}}(G)$ (resp. $\cA_{P, \mathrm{disc}}(G)$) to be the subspace
of cuspidal automorphic forms (resp. discrete automorphic forms), i.e. functions $\varphi \in \cA_{P}(G)$ such that
$\varphi_Q = 0$ for all $Q \subsetneqq P$ (resp. such that $\valP{\varphi} \in L^2([G]_{P,0})$). We will simply drop the subscripts $P$ when $P
= G$.

A cuspidal (resp. discrete) automorphic representation $\pi$ of $M(\bA)$ is a topologically
irreducible subrepresentation of $\cA_{\mathrm{cusp}}(M)$ (resp. of $\cA_{\mathrm{disc}}(M)$). Let $\pi$ be a
cuspidal (resp. discrete) automorphic representation of $M(\bA)$. We define $\cA_{\pi,
\mathrm{cusp}}(M)$ (resp. $\cA_{\pi,
\mathrm{disc}}(M)$) to be the $\pi$-isotypic component of
$\cA_{\mathrm{cusp}}(M)$ (resp. $\cA_{\mathrm{disc}}(M)$), and set
    \[
    \Pi = \Ind_{P(\bA)}^{G(\bA)} \pi, \quad
    \cA_{P, \pi, \mathrm{cusp}}(G) = \Ind_{P(\bA)}^{G(\bA)}
    \cA_{\pi, \mathrm{cusp}}(M), \quad \left( \text{resp. } \cA_{P, \pi, \mathrm{disc}}(G) = \Ind_{P(\bA)}^{G(\bA)}
    \cA_{\pi, \mathrm{disc}}(M) \right).
    \]
Here $\Ind$ stands for the normalized smooth induction. These spaces have a natural topology described in~\cite{BPCZ}*{\S 2.7} which gives them the structure of SLF representations of $G(\bA)$. We identify
$\Pi$ (resp. $\cA_{P, \pi, \mathrm{cusp}}(G)$, $\cA_{P, \pi, \mathrm{disc}}(G)$) with the space of
forms $\varphi \in \cA_{P}(G)$ such that the function
    \[
    m \mapsto e^{-\langle \rho_P, H_P(m) \rangle} \varphi(mg), \quad
    m \in [M]
    \]
belongs to $\pi$ (resp. $\cA_{\pi, \mathrm{cusp}}(M)$, $\cA_{\pi, \mathrm{disc}}(G)$) for every $g \in [G]_P$.

Let $\lambda \in \fa_{P, \C}^*$. We define the twist $\pi_{\lambda}$ as the
space of functions of the form
    \[
    m \mapsto e^{\langle \lambda, H_P(m) \rangle} \varphi(m), \quad
    m \in M(\bA), \quad \varphi \in \pi.
    \]
If $\pi$ is cuspidal, for $f \in \cS(G(\bA))$ we denote by
$I(\lambda, f)$ the action of $f$ on $\cA_{P, \pi_{\lambda},
\mathrm{cusp}}(G)$ obtained by transporting the action on $\cA_{P,
\pi, \mathrm{cusp}}(G)$ through the identification
    \[
    \cA_{P, \pi, \mathrm{cusp}}(G) \to
    \cA_{P, \pi_{\lambda}, \mathrm{cusp}}(G), \quad
    \varphi \mapsto e^{\langle \lambda, H_P(\cdot) \rangle} \varphi(\cdot).
    \]
In the same way, if $\pi$ is discrete we also get an action on $\cA_{P, \pi_{\lambda},
\mathrm{disc}}(G)$ still denoted by $I(\lambda,f)$.

If the central character of $\pi$ is unitary, we equip $\Pi$ and $\cA_{P, \pi, \mathrm{cusp}}(G)$ (resp. $\cA_{P, \pi, \mathrm{disc}}(G)$) with the Petersson inner product
\begin{equation}
\label{eq:Petersson}
     \langle \varphi_1, \varphi_2 \rangle_{\mathrm{Pet}} =
    \int_{[G]_{P,0}} \varphi_1(g) \overline{\varphi_2(g)} \rd g.
\end{equation}

For every $\varphi \in \cA_{P, \mathrm{disc}}(G)$, $\lambda \in \fa_{P,
\C}^*$, we have the Eisenstein series
    \[
    E(g, \varphi, \lambda) = \sum_{\gamma \in P(F) \bs G(F)}
    \varphi(\gamma g) e^{\langle \lambda, H_P(\gamma g) \rangle}.
    \]
This sum is absolutely convergent when $\Re \lambda$ is in a certain cone,
and $E(g, \varphi, \lambda)$ has a meromorphic continuation to all $\lambda$ which is regular on $i \fa_P^*$. By~\cite{Lap}*{Theorem~2.2}, for every discrete automorphic representation $\pi$ of $M(\bA)$ and for every $\lambda \in i
\mathfrak{a}_P^*$, the map
$\varphi \mapsto E( \cdot, \varphi, \lambda)$ induces a continuous map $\Pi \to
\mathcal{T}([G])$ that actually factors through $\cT_N([G])$ for some $N>0$. The resulting map $\Pi \to \cT_N([G])$ is continuous by the closed graph theorem (see Remark~\ref{remark:closed_graph_theorem}).

\subsection{Cuspidal data and Langlands decompositions}    \label{subsec:langlands_decomposition}

\subsubsection{Cuspidal data}
We continue to assume that $G$ is reductive in this subsection. Let $\underline{\fX}(G)$ be the set of pairs $(M_P, \pi)$ where
\begin{itemize}
    \item $P = M_P N_P$ is a standard parabolic subgroup of $G$,
    \item $\pi$ is an (isomorphism class of a) cuspidal automorphic representation of $M_P(\bA)$ whose central character is trivial on $A_P^\infty$.
\end{itemize}
Two elements $(M_P, \pi)$ and $(M_Q, \tau)$ of
$\underline{\fX}(G)$ are equivalent if there
is a $w \in W(P, Q)$ such that $w\pi w^{-1} = \tau$. We define a cuspidal
datum to be an equivalence class of such $(M_P, \pi)$ and denote by $\fX(G)$
the set of all cuspidal data. If $\chi \in \fX(G)$ is represented by $(M_P, \pi)$ we define $\chi^\vee$ to be the cuspidal datum represented by $(M_P, \pi^\vee)$. Note that the natural inclusion $\underline{\fX}(M)
\subset \overline{\fX}(G)$ descends to a finite-to-one map $\fX(M) \to \fX(G)$.

\subsubsection{Coarse Langlands decomposition}
For $(M_P,\pi) \in \underline{\fX}(G)$, let $\cS_\pi([G]_P)$ be the space of $\varphi \in \cS([G]_P)$ such that
\begin{equation*}
    \varphi_\lambda(x):=\int_{A_P^\infty} e^{-\langle \rho_P+\lambda,H_P(a) \rangle} \varphi(ax) \rd a, \; x \in [G]_P,
\end{equation*}
belongs to $\cA_{P,\pi_\lambda,\mathrm{cusp}}(G)$ for every $\lambda \in \fa_{P,\C}^*$.

Let $P \subset G$ be a standard parabolic subgroup, $\chi \in \fX(G)$ be a cuspidal datum and $\{ (M_{Q_i}, \pi_i) \; | \; i \in I\} $ be the inverse image of $\chi$ in $\underline{\fX}(M_P)$. Denote by $L^2_\chi([G]_P)$ the closure in $L^2([G]_P)$ of the subspace
\begin{equation}
\label{eq:pseudo_subspace}
    \sum_{i \in I} E_{Q_i}^P(\cS_{\pi_i}([G]_{Q_i})).
\end{equation}
 We define similarly $L^2_{\chi}([G]_{P,0}) \subset L^2([G]_{P,0})$. The coarse Langlands decomposition (\cite{MW95}*{Proposition~II.2.4}) states that we have decompositions in orthogonal direct sums
\begin{equation*}
    L^2([G]_P)=\widehat{\bigoplus_{\chi \in \fX(G)}}L^2_{\chi}([G]_P) \quad \text{and} \quad L^2([G]_{P,0})=\widehat{\bigoplus_{\chi \in \fX(G)}}L^2_{\chi}([G]_{P,0}).
\end{equation*}

For any subset $\fX \subset \fX(G)$, set
\begin{equation*}
    L^2_{\fX}([G]_P)=\widehat{\bigoplus_{\chi \in \fX}}L^2_{\chi}([G]_P),
\end{equation*}
and define
\begin{equation*}
    \cS_{\fX}([G]_P)=\cS([G]_P) \cap L^2_{\fX}([G]_P).
\end{equation*}
Let $w$ be a weight on $[G]_P$. For any $\cF \in \{L^2_w, \cT_N, \cT, \cS_{w,N}, \cS_w\}$, define $\cF_{\fX}([G]_P)$ to be the orthogonal of $\cS_{\fX^c}([G]_P)$ in $\cF([G]_P)$, where $\fX^c$ is the complement of $\fX$ in $\fX(G)$. By \cite{BPCZ}*{Section~2.9.4}, for $\cF \in \{ L^2_w, \cT, \cS \}$ there are canonical projections
\begin{equation}
    \label{eq:chi_proj}
    \cF([G]_P) \to \cF_{\fX}([G]_P), \quad \varphi \mapsto \varphi_{\fX}.
\end{equation}
The following proposition is contained in~\cite{BPCZ}*{Theorem~2.9.4.1}.

\begin{prop}    \label{prop:absolutely_summable_chi}
There exists an integer $N_0$, such that for all $\varphi \in \cT_w([G]_P)$
(resp. $\cS_{w, N}([G]_P)$), the family $(\varphi_{\chi})_{\chi \in \fX(G)}$
is absolutely summable with the sum $\varphi$ in $\cT_{w, N}([G]_P)$ (resp.
$\cS_{w, N+N_0}([G]_P)$).
\end{prop}

Moreover, we have the following result from~\cite{BPCZ}*{Section~2.9.5} (which follows from the density of $\cS([G]_P)$ in $L^2_{w}([G]_P)^\infty$ stated in Subsection~\ref{subsubsec:L2} and the continuity of the projections $\varphi \mapsto \varphi_\fX$).
\begin{prop}
\label{prop:density_schwartz}
    Let $\fX$ be a subset of $\fX(G)$. The space $\cS_{\fX}([G]_P)$ is dense in $L^2_{w,\fX}([G]_P)^\infty$.
\end{prop}

\subsubsection{Automorphic kernel functions}

The right convolution by $f \in \cS(G(\bA))$ on the spaces $L^2_{\chi}([G]_P)$ and $L^2([G]_P)$ gives rise to integral operators whose kernel are denoted by $K_{f, P, \chi}$ and $K_{f, P}$ respectively. If $P = G$ we omit the subscript $G$.
These kernel functions satisfy the following estimates,
cf.~\cite{BPCZ}*{Lemma~2.10.1.1}.

\begin{lemma}   \label{lemma:estimate_kernel}
There exists $N_0>0$ such that for every weight $w$ on $[G]_P$ and every
continuous seminorm $\aabs{\cdot}_{w, N_0}$ on $\cT_{w, N_0}([G]_P)$, there
exists a continuous seminorm $\aabs{\cdot}_{\cS}$ on $\cS(G(\bA))$ such that for $f \in \cS(G(\bA))$ we have
    \[
    \sum_{\chi \in \fX(G)} \aabs{K_{f, P, \chi}(\cdot, y)}_{w, N_0} \leq
    \aabs{f}_{\cS} w(y)^{-1}, \quad y \in [G]_P.
    \]
In particular
    \[
    \sum_{\chi \in \fX(G)}
    \abs{K_{f, P, \chi}(x, y)} \leq
    \aabs{x}_P^{N_0} w(x) w(y)^{-1} \aabs{f}_{\cS}, \quad
    x, y \in [G]_P.
    \]
\end{lemma}

\subsection{Approximation by constant terms}
If $G$ is reductive, functions of uniformly moderate growth are approximated
by their constant terms in a precise sense. This is the ``approximation by
constant terms'', cf.~\cite{BPCZ}*{Theorem 2.5.14.1~(1)}. We extend this to
the case of possibly nonreductive groups satisfying the
condition~\eqref{eq:(SR)}.

Let $G = U \rtimes H$ be a connected algebraic group satisfying the
condition~\eqref{eq:(SR)} as in Section~\ref{subsubsec:D-parabolic}.
Let $P$ be a standard D-parabolic subgroup. Put $P_H = P \cap H$, which is a parabolic subgroup of $H$. For $P \subset Q$, we define a weight
function $d_P^Q$ on $[H]_{P_H}$ by
    \begin{equation}  \label{eq:weight_d_P_Q}
    d_P^Q(h) = \min_{\alpha \in \Psi_P^Q} d_{P_H,\alpha}(h), \quad h \in [H]_{P_H}.
    \end{equation}
We recall that $\Psi_P^Q = \Psi_P \bs \Psi_Q$ is the set of roots of $A_0$ action on $\fn_P^Q = \fn_P/\fn_Q$. Note that $P, Q$ are not parabolic subgroups of $H$, and the weight $d_P^Q$ is not to be confused with $d_{P_H}^{Q_H}$ which we recall is defined by
    \[
    d_{P_H}^{Q_H}(h) = \min_{\alpha \in \Psi_{P_H}^{Q_H}} d_{P_H,\alpha}(h), \quad h \in [H]_{P_H}.
    \]
However we have the following relation.

\begin{lemma}   \label{lemma:weight_restriction}
We have $d_P^Q \le d_{P_H}^{Q_H}$.
\end{lemma}

\begin{proof}
By definition, we have $N_P \cap H = N_{P_H}$ and $N_Q \cap H =
N_{Q_H}$, so that $\fn_{P_H}/\fn_{Q_H}$ embeds into $\fn_P/\fn_Q$ as a subspace. It follows that $\Psi_{P_H}^{Q_H} \subset \Psi_{P}^Q$ which implies $d_P^Q \le d_{P_H}^{Q_H}$.
\end{proof}

The ``approximation by constant terms'' refers to the following theorem.

\begin{theorem} \label{thm:approximation_by_constant_term}
Assume that $G$ satisfies condition~\eqref{eq:(SR)}. Let $P \subset Q$. Let $N>0$, $r \ge 0$ and $X \in \cU(\fg_\infty)$. There exists a continuous
seminorm $\| \cdot \|=\| \cdot \|_{N,X,r}$ such that for any $\varphi \in
\cT_N([G]_Q, \psi)$ and $x \in P_H(F) N_{Q_H}(\bA) \bs H(\bA)$, we have
    \begin{equation} \label{eq:thm_approxiamtion_by_constant_term}
      \left| \mathrm{R}(X) \varphi(x) - \mathrm{R}(X) \varphi_P(x) \right|
    \le \|x\|_{P}^N d_P^Q(x)^{-r} \| \varphi \|.
    \end{equation}
\end{theorem}

By Remark ~\ref{rem:height_on_[G]_P_and_[H]_P_H}, we can also use $\| x \|_{P_H}$ in right hand side of ~\ref{eq:thm_approxiamtion_by_constant_term}.
It is however of critical importance that the weight $d_P^Q$ appears on the right hand side, not $d_{P_H}^{Q_H}$.

We begin with an elementary lemma.

\begin{lemma} \label{lem:Fourier_series_estimate}
Let $V$ be a finite dimensional vector space over $F$ and $U \subset
V(\bA)$ be an open compact subgroup. Then there exists a homogeneous $X \in
\cU(\mathrm{Lie}(V(F_\infty)))$ such that for any $f \in
C^\infty([V])^U$ and any sufficiently large integer $r$, we have
    \begin{equation}    \label{eq:constant_term_approx}
    \| f - \int_{[V]} f(v) \rd v  \|_{L^\infty}
    \ll_r \| \mathrm{R}(X^r) f \|_{L^\infty}.
    \end{equation}
Note that, by our convention on measures, the measure on $[V]$ is the Tamagawa measure, i.e. $\mathrm{vol}([V])=1$.
\end{lemma}

\begin{proof}
Let $\Lambda = V(F)
\cap U$ which is a lattice in $V(F_\infty)$. Since $V(F)$ is dense in
$V(\bA)$, the natural embedding induces a $V(F_\infty)$-equivariant
isomorphism
    \[
    \Lambda \backslash V(F_\infty) \cong V(F) \backslash V(\bA) /U.
    \]
Thus we are reduced to the same problem for $\Lambda \backslash V(F_\infty)$.
We pick a suitable basis and identify it with $\Z^n \backslash \R^n$ for $n =
\dim_\Q V$. By classical Fourier series theory, $ X = \sum_{i=1}^n
\frac{\partial^2}{\partial x_i^2} $ suffices.
\end{proof}

\begin{proof}[Proof of Theorem~\ref{thm:approximation_by_constant_term}]
Up to replacing $\mathrm{R}(X)\varphi$ by $\varphi$, we can assume that $X=1$. Since
the constant term map $\varphi \mapsto \varphi_P$ is continuous from $\cT_N([G]_{Q},\psi)$ to
$\cT_N(P(F)N_Q(\bA) \backslash G(\bA),\psi)$, there exists a continuous semi-norm $\| \cdot \|$ on
$\cT_N([G]_Q,\psi)$ such that for all $x \in P(F)N_Q(\bA) \backslash G(\bA)$, we have
    \[
    \left| \mathrm{R}(X)\varphi(x)-\mathrm{R}(X)\varphi_P(x) \right| \le
    \|x\|_P^N \| \varphi \|.
    \]
Thus we only need to consider those $x$ such that $d_P^Q(x)>C$ for some $C$,
hence $x \in \omega_{P_H}^{Q_H}[>C]$. For $C$ large enough, $P_H(F) \fs^{Q_H}
\subset \omega_{P_H}^{Q_H}[>C]$, hence we can assume $x \in \fs^{Q_H}$ at the
beginning.

Note that
    \[
    \varphi_P(x) = \int_{[N_P^Q]} \varphi(nx) \rd n.
    \]
Take a filtration of $N_P^Q$
    \[
    \{0\} = N_0 \subset N_1 \subset \cdots \subset N_k=N_P^Q
    \]
as in Lemma ~\ref{lem:filtration_of_NPQ}. For each $0 \le i \le k-1$, set
    \[
    \varphi_i(x) := \int_{[N_i]} \varphi(nx) \rd n.
    \]
Then $\varphi_0=\varphi$ and $\varphi_k=\varphi_P$. It suffices
to show that, for each $i$ with $0 \le i \le k-1$, there exists a continuous
semi-norm $\| \cdot \|_i$ on $\cT_N([G]_Q,\psi)$ such that for $x \in \fs^{Q_H}$,
we have
    \[
    \left| \varphi_i(x) - \varphi_{i+1}(x) \right| \ll
    \|x\|_P^N d_P^Q(x)^{-r} \| \varphi \|_i.
    \]
Once we have this, the semi-norm $\| \cdot\| = \sup_i \| \cdot \|_i$ satisfies the condition.

Since $N_{i+1}/N_i$ is a vector space, up to enlarging $r$ (which is possible as $d_P^Q(x) >C$), by Lemma
~\ref{lem:Fourier_series_estimate} there exists an $\bar{X} \in
\fn_{i+1,\infty}/\fn_{i,\infty}$ such that we have
    \begin{equation} \label{eq:phi_i-phi_i+1}
    \left| \varphi_i(x) - \varphi_{i+1}(x) \right| =
    \left| \varphi_i(x) - \int_{[N_{i+1}/N_i]} \varphi_{i}(nx) \rd n \right|
    \ll \left\| \mathrm{R}(\bar{X}^r)
    \mathrm{R}(x) \varphi_i \right\|_{L^\infty([N_{i+1}])}.
    \end{equation}
Take the complementary subspace $\fn_{i+1}^{i}$ of $\fn_i$ in $\fn_{i+1}$ as
in the Lemma ~\ref{lem:filtration_of_NPQ}. Let $X \in \fn_{i+1}^i$ with image
$\overline{X}$ , then
    \[
    \mathrm{R}(\overline{X}^r) \mathrm{R}(x) \varphi_i =
    \mathrm{R}(X^r) \mathrm{R}(x) \varphi_i =
    \mathrm{R}(x) \mathrm{R}\left( \mathrm{Ad}(x_\infty^{-1})(X^r)\right)
    \varphi_i.
    \]
Therefore, as $[N_{i+1}]$ is compact,
    \begin{equation} \label{eq:approx_constant_term_adjoint}
    \left\| \mathrm{R}(\bar{X}^r) \mathrm{R}(x) \varphi_i \right\|_{L^\infty([N_{i+1}])}
    \ll \|x\|_P^N \sup_{y \in [G]_Q} \left( \|y\|_P^{-N} \cdot \valP{\mathrm{R}\left( \mathrm{Ad}(x_\infty^{-1})(X^r)\right)
    \varphi_i(y) }\right).
    \end{equation}

For $x \in \fs^{Q_H}$, $x_\infty$ can be written in the form $x_\infty=ac$,
where $a \in A_0^Q(T_-)$, and $c$ lies in a compact subset of $H(F_\infty)$. Thus
    \[
    \mathrm{Ad}(x_\infty^{-1})X = \mathrm{Ad}(c^{-1})
    \mathrm{Ad}(a^{-1})X = \mathrm{Ad}(c^{-1})
    e^{-\langle \alpha,H_0(a) \rangle} X,
    \]
    where $\alpha$ is the only root of $A_0$ acting on $\fn_{i+1}^{i}$. Pick a basis $E_i$ of $\cU(\fg_\infty)^{\le r}$, the elements in
$\cU(\fg_\infty)$ of degree $\le r$. If we write
    \begin{equation} \label{eq:approx_constant_term_d_P^Q_coefficient}
     \mathrm{Ad}(x_\infty^{-1})X^r = \sum_i c_i(x)E_i,
    \end{equation}
then we deduce that $|c_i(x)| \le d_{{P_H},\alpha}(x)^{-r}$, since we recall that we have
    \[
    d_{P_H,\alpha}(x) \sim e^{\langle \alpha, H_0(x) \rangle}.
    \]
when $x \in \fs^{P_H}$.
The theorem then follows from~\eqref{eq:phi_i-phi_i+1}, ~\eqref{eq:approx_constant_term_adjoint} and ~\eqref{eq:approx_constant_term_d_P^Q_coefficient}.
\end{proof}

\subsection{A mild extension of Arthur's truncation operator}
\label{subsec:arthur's_truncation}
Let $G$ be a connected reductive group over $F$. Fix a minimal parabolic subgroup $P_0$.
Let $\cF$ be the set of standard parabolic subgroups of $G$. We define a space of functions $\cT_{\cF}(G)$ as
    \[
    \cT_\cF(G) = \left\{ ({}_P \varphi) \in \prod_{P \in \cF}
    \cT(P(F) \bs G(\bA))
    \mid {}_P\varphi - {}_Q\varphi \in \cS_{d_P^Q}(P(F) \bs G(\bA))
    \text{ for any }  P \subset Q   \right\}.
    \]
This space embeds as a closed subspace of the LF space,
    \[
    \prod_{P \in \cF} \cT(P(F) \bs G(\bA)) \times
    \prod_{\substack{P \subset Q \\ P,Q \in \cF}}
    \cS_{d_P^Q}(P(F) \bs G(\bA)),
    \]
and as such inherits an LF topology.

We define $\cS^0([G]_P^1)$ to be the Banach space of measurable functions on $[G]_P^1$ such that for any $N$,
\begin{equation} \label{eq:norm_Schwartz_GP1}
       \|f\|_{\infty,N} := \sup_{x \in [G]_P^1} \|x\|_P^N \lvert f(x) \rvert < \infty.
\end{equation}

\begin{prop} \label{prop:relative_truncation}
Let
$\underline{\varphi}=({}_P \varphi) \in \cT_{\cF}(G)$ be a collection of functions. For $g \in [G]_P^1$ and $T \in \fa_0$, define
    \[
    \Lambda^T \underline{\varphi}(g) =
    \sum_{P \in \cF_0} \epsilon_P
    \sum_{\gamma \in P(F) \backslash G(F)}
    \widehat{\tau}_{P}(H_{P}(\gamma g)-T_{P}) {}_P \varphi(\gamma g).
    \]
and
    \[
    \Pi^T \underline{\varphi}(g) = F^{G}(g,T) \cdot {}_{G} \varphi(g).
    \]
Then for every $c>0,N>0$, there exists a continuous seminorm $\| \cdot
\|_{c,N}$ on $\cT_{\cF}(G)$ such that
    \[
    \| \Lambda^{T} \underline{\varphi} - \Pi^T \underline{\varphi}
    \|_{\infty,N} \le e^{-c\|T\|} \| \underline{\varphi} \|_{c,N},
    \]
for $\varphi \in \cT_{\cF}(G)$ and $T \in \fa_0$ sufficiently positive, where the norm on left hand side is as defined in ~\eqref{eq:norm_Schwartz_GP1}.
\end{prop}

\begin{remark}
When $\varphi \in \cT([G])$ and $\prescript{}{P}{\varphi} = \varphi_P$ for
all $P \in \cF$, we have $(\varphi_Q)_P = \varphi_P$, and hence the family
$\underline{\varphi} = (\varphi_P)_{P \in \cF}$ is in $\cT_{\cF}(G)$ by the approximation by constant terms for $G$.
In this case $\Lambda^T \underline{\varphi}$ is Arthur's truncation operator
defined in~\cite{Arthur2}.
\end{remark}

Following Arthur ~\cite{Arthur3}*{Section 6}, for standard parabolic subgroups $P \subset Q$, we let $\sigma_P^Q$ be the characteristic functions of $H \in \fa_0$ that satisfies the following properties.
\begin{itemize}
    \item $\langle \alpha,H \rangle > 0$ for all $\alpha \in \Delta_P^Q$.
    \item $\langle \alpha,H \rangle \le 0$ for all $\alpha \in \Delta_P \setminus \Delta_P^Q$.
    \item $\langle \varpi, H \rangle > 0$ for all $\varpi \in \widehat{\Delta}_Q$.
\end{itemize}

We will need the following lemma in the proof of
Proposition~\ref{prop:relative_truncation}.

\begin{lemma} \label{lem:F_sigma_non_zero}
Let $G$ be a reductive group over $F$, then for every sufficiently positive
$T \in \fa_0$ and every $g \in G(\bA)^1$ with
    \[
    F^P(g,T) \sigma_P^Q(H_P(g)-T_P) \ne 0
    \]
there exists $r>0$ such that
    \begin{align*}
    e^{\| T \|} &\ll \left( \min_{\alpha \in \Delta_0^Q \setminus \Delta_0^P}
    d_{P,\alpha}(g) \right)^r \text{ and } \\
    \| g \|_P &\ll \left( \max_{\alpha \in \Delta_0^Q \setminus \Delta_0^P}
    d_{P,\alpha}(g)  \right)^r.
    \end{align*}
\end{lemma}

\begin{proof}
The proof is similar to ~\cite{BPCZ}*{Lemma 3.5.1.2}. For the convenience of
readers, we reproduce it here.

We can assume $g \in \fs^P$. By \cite{BPC22}*{Lemma 2.3.3.1}, for any $\alpha
\in \Delta_0^Q \backslash \Delta_0^P$, we have $\langle \alpha,H_0(g) \rangle
> \langle \alpha,T \rangle \ge \varepsilon \| T \| $. Therefore
    \[
    d_{P,\alpha}(g) \sim e^{\langle \alpha, H_0(g) \rangle}
    > e^{\varepsilon \|T\|}
    \]
for any $\alpha \in \Delta_0^P \setminus \Delta_0^Q$, the first inequality is
thus proved.

By~\cite{LW}*{Lemme 2.10.6}, when $g \in G(\bA)^1$ we have
    \begin{equation} \label{eq:mixed_truncation_lemma_1}
    \| H_P(g)-T_P \| \ll  \| H_P^Q(g)-T_P^Q \|
    \sim 1 + \max_{\alpha \in \Delta_P^Q} \langle \alpha,H_0(g)-T \rangle,
    \end{equation}
where $H_P^Q(g)$ and $T_P^Q$ stands for projection of $H_P(g)$ and $T_P$ into
$\fa_P^Q$ respectively. The condition $F^P(H_0(g)-T) \ne 0$ implies
    \begin{equation} \label{eq:mixed_truncation_lemma_2}
    \| H^P(g) \| \ll \|T\|,
    \end{equation}
where $H^P(g)$ stands for the projection of $H_0(g)$ to $\fa^P$. For $\alpha
\in \Delta_0^Q \setminus \Delta_0^P$, let $\alpha = \alpha_P + \alpha^P$ be
the decomposition of $\alpha$ according to $\fa_0^*= \fa_P^* \oplus \fa_0^{P,*} $, then $\alpha_P$ runs through $\Delta_P^Q$ as $\alpha$ runs through
$\Delta_0^Q \setminus \Delta_0^P$. Since for any $\beta \in \Delta_0^P$,
$\langle \alpha, \beta^\vee \rangle = \langle \alpha,\beta \rangle \le 0$.
Thus $\alpha^P$ is a non-positive linear combination of $\widehat{\Delta}_0^P$,
hence
    \begin{equation} \label{eq:mixed_truncation_lemma_3}
    \langle \alpha , H_0(g)-T \rangle =
    \langle \alpha_P,H_0(g)-T \rangle +
    \langle \alpha^P,H_0(g)-T \rangle \ge
    \langle \alpha_P,H_0(g)-T \rangle.
    \end{equation}
Combining \eqref{eq:mixed_truncation_lemma_1},
~\eqref{eq:mixed_truncation_lemma_2} and
~\eqref{eq:mixed_truncation_lemma_3}, we obtain
    \begin{equation} \label{eq:mixed_truncation_lemma_4}
    \| H_0(g) \| \ll 1 + \|T\| +
    \max_{\alpha \in \Delta_0^Q \setminus \Delta_0^P}
    \langle \alpha, H_0(g) \rangle.
    \end{equation}

Combining with the first assertion, we have
    \[
    \| H_0(g) \| \ll 1 + \max_{\alpha \in \Delta_0^Q \setminus \Delta_0^P}
    \langle \alpha, H_0(g) \rangle.
    \]
Finally note that $g \in \fs^P$, thus for some $r>0$,
    \[
    \| g \|_P \sim e^{\| H_0(g) \|} \ll
    \max_{\alpha \in \Delta_0^Q \setminus \Delta_0^P}
    e^{r\langle \alpha, H_0(g) \rangle} \sim
    \left( \max_{\alpha \in \Delta_0^Q \setminus \Delta_0^P} d_{P,\alpha}
    \right)^r.
    \]
This proves the lemma.
\end{proof}

\begin{proof}[Proof of Proposition~\ref{prop:relative_truncation}]
Our proof follows the same line as~\cite{BPCZ}*{Section~3.5}. Using Langlands
partition formula~\eqref{eq:langlands_partition} and
    \[
    \tau_P^R \widehat{\tau}_R = \sum_{Q \supset R} \sigma_P^Q
    \]
we have
    \begin{align*}
    \Lambda^T \underline{\varphi}(g)
    &= \sum_{R \in \cF} \sum_{P \subset R} \epsilon_R \times \\
    \sum_{\gamma \in R(F) \backslash G(F)}
    &  \sum_{\delta \in P(F) \backslash R(F)}
    \widehat{\tau}_{R}(H_R(\gamma g)-T_R) {}_R \varphi(\gamma g)
    F^P(\delta \gamma g,T) \tau_P^R(H_P(\delta \gamma g)-T_P)  \\
    &=\sum_{P \subset Q} \sum_{\delta \in P(F) \backslash G(F)}
    F^P(\delta g,T) \sigma_P^Q(H_P(\delta g)-T) {}_{P,Q} \varphi(\delta g)
    \end{align*}
where
    \[
    {}_{P,Q} \varphi(g) = \sum_{P \subset R \subset Q}
    \epsilon_R \cdot {}_R \varphi(g)
    \]
Hence
    \[
    \Lambda^T \underline{\varphi}(g) - \Pi^T \underline{\varphi}(g) =
    \sum_{P \subsetneq Q} \sum_{\delta \in P(F) \backslash G(F)}
    F^P(\delta g,T) \sigma_P^Q(H_P(\delta g)-T_P) {}_{P,Q}\varphi(\delta g).
    \]

Since $E_P^G$ is a continuous map from $\cS^0(P(F)\backslash G(F))$ to $\cS^0([G])$, we only
need to show for every $c>0,N>0$, there exists a continuous seminorm on
$\cT_{\cF}(G)$ such that
    \begin{equation} \label{eq:relative_truncation_only_need}
    \left| {}_{P,Q} \varphi(g) \right| \ll e^{-c\|T\|} \|g\|_P^{-N} \|
    \underline{\varphi} \|
    \end{equation}
holds every $g$ with $F^P(g,T)\sigma_P^Q(H_P(g)-T_P)\ne 0$ and every
$\underline{\varphi} \in \cT_{\cF}(G)$.

Fix $\alpha \in \Delta_0^Q \setminus \Delta_0^R$. For a parabolic subgroup $R$ with $P
\subset R \subset Q$ and $\alpha \in \Delta_0^R$, define $R^\alpha$ such that
$\Delta_0^{R^\alpha}=\Delta_0^R \setminus \{ \alpha \}$. Then there exists
$N_0>0$ such that for any $r>0$, there exists a seminorm $\| \cdot \|$ on
$\cT_{\cF}(G)$ such that
    \[
    \left|  {}_{P,Q}\varphi(g) \right|
    \le \sum_{ \substack{ P \subset R \subset Q \\ \alpha \in \Delta_0^R }}
    \left| {}_R \varphi(g)-{}_{R^\alpha} \varphi(g) \right|
    \le \|g\|_P^{N_0}
    \sum_{ \substack{P \subset R \subset Q \\ \alpha \in \Delta_0^R}}
    d_{R^\alpha,\alpha}(g)^{-r} \| \underline{\varphi} \|.
    \]

By first assertions of Lemma ~\ref{lem:F_sigma_non_zero}, $g \in
\omega_P^Q[>C] \subset \omega_P^{R^\alpha}[>C']$, hence
$d_{R^\alpha,\alpha}(g) \sim d_{P,\alpha}(g)$. Hence
    \[
    \left|  {}_{P,Q}\varphi(g) \right|
    \ll \|g\|_P^{N_0}
    \sum_{ \substack{P \subset R \subset Q \\ \alpha \in \Delta_0^R}}
    d_{P,\alpha}(g)^{-r} \| \underline{\varphi} \|.
    \]
Finally~\eqref{eq:relative_truncation_only_need} follows if we let $\alpha$
vary and use the second assertion of Lemma~\ref{lem:F_sigma_non_zero}, .
\end{proof}

\subsection{Decomposition of Schwartz functions}

\label{subsec:decomposition_of_Schwartz_functions}
We begin by some general considerations. Let $G$ be an algebraic group over $F$
with a morphism $p:G \to \mathbf{A}^m_F$. Here $\mathbf{A}^m_F$
stands for the $m$-dimensional affine space over
$F$, whose $F$-points will also be noted by $F^m$. Let $\Lambda \subset
F_\infty^m$ be a full lattice (i.e. discrete finitely generated abelian group
which generate $F_\infty^m$ as an $\R$-vector space).

Let $U \subset F^m_\infty$ be a neighbourhood of $0$, such that $U \cap
\Lambda = \{ 0 \}$, choose any $u \in C_c^\infty(F^m_\infty)$ such that
$\supp u \subset U$ and $u(0)=1$. For any $\alpha \in \Lambda$, define a
function $u_\alpha$ on $G(\bA)$ by $u_\alpha(g)=u(p(g_\infty)-\alpha)$, and
$f_\alpha:=f \cdot u_\alpha$. For $\beta \in \Lambda $, and $g \in G(\bA)$
with $p(g)=\beta$, we have
    \begin{equation} \label{eq:f_alpha(g)}
    f_\alpha(g) =
    \begin{cases} f(g) & \beta = \alpha \\
    0 & \beta \ne \alpha  \end{cases}
    \end{equation}

\begin{prop} \label{prop:decomposition_of_Schwartz}
For any $f \in \cS(G(\bA))$, the sequence of functions $(f_\alpha)_{\alpha \in
\Lambda}$ is absolutely summable in $\cS(G(\bA))$.
\end{prop}

\begin{proof}

The space $\cS(G(\bA))$ is a union of Fr\'echet spaces $\cS(G(\bA),C,K)$ (cf.~Subsection ~\ref{subsec:spaces_of_function} with $Z$ being trivial), and if $f \in \cS(G(\bA),C,K)$ ,
then $f_\alpha \in \cS(G(\bA),C,K)$ for all $\alpha$. It suffices to prove that for all $N>0$ and $X \in \cU(\fg_\infty)$
    \[
    \sum_{\alpha \in \Lambda} \| f_\alpha \|_{X,N} < \infty.
    \]
(c.f. Remark ~\ref{rmk:Schwartz_space_equiv_norm}).

By the Leibniz rule,
    \[
    X f_\alpha = \sum_{Y,Y'} c_{Y,Y'} \cdot  Y u_\alpha \cdot Y' f
    \]
for some universal constants $c_{Y,Y'}$, thus it suffices to show for any
$Y,Y' \in \cU(\fg_\infty)$
    \[
    \sum_{\alpha \in \Lambda} \sup_{g \in G(\bA)} \| g \|^N
    \left| Y u_\alpha(g) \right| \left|Y'f(g) \right| < \infty.
    \]
By the chain rule,
    \[
    Y u_\alpha (g) = \sum_Z c_Z \cdot
    (Zu)(p(g_\infty)-\alpha) \cdot p_Z(g_\infty)
    \]
for some $Z \in \cU(g)$, where $c_Z$ are constants and $p_Z$ are polynomials.
Choose $N_1$ such that $\left| p_Z(g_\infty) \right| \ll \| g_\infty \|^{N_1}
\ll \| g \|^{N_1}$ holds for all $Z$. Then it suffices to show for any $Y',Z
\in \cU(\fg_\infty)$
    \[
    \sum_{\alpha \in \Lambda} \sup_{g \in G(\bA)} \| g\|^{N+N_1}
    \left|  Zu(p(g_\infty)-\alpha) \right| \left| Y'f(g) \right| <\infty.
    \]
Choose an norm $\| \cdot \|$ on $F_\infty^n$. There exists $C$ such that
$u(x-\alpha) \ne 0$ implies $\| x \| \ge C\| \alpha \|$,
by~\cite{Kottwitz05}*{Proposition~18.1~(1)} there exists $N_2$ such that
$Zu(p(g_\infty)-\alpha) \ne 0$ implies $\| g_\infty \|^{N_2} \gg \| \alpha
\|$. Since $f$ is Schwartz, for any $M>0$
    \[
    \sup_{g \in G(\bA)} \left| Y'f(g) \right|
    \| g \|^{N_1+N} \ll \| g \|^{-M} \ll \| g_\infty \|^{-M}.
    \]
Combining these, for any $N'>0$
    \[
    \sum_{\alpha \in \Lambda} \sup_{g \in G(\bA)}
    \| g\|^{N+N_1} \left|  Zu(p(g_\infty)-\alpha) \right|
    \left| Y'f(g) \right| \ll \sum_{\alpha \in \Lambda} \| \alpha \|^{-N'},
    \]
which is finite for $N'$ large enough.
\end{proof}

\section{Jacobi groups} \label{sec:Jacobi_groups}

\subsection{Jacobi groups: general linear groups} \label{subsec:Jacobi groups}

Put $L = E^n$ (column vectors) and $L^\vee = E_n$ (row vectors). The usual dot product gives a pairing between $L$ and $L^\vee$. If $W$ is a linear
subspace of $L$ we will write $W^\vee = \{\tp{x} \in L^\vee \mid x \in W\}$ and $W^\perp = \{ y \in L^\vee \mid x y = 0 \text{ for all } x \in W\}$. Denote by
$(e_1, \hdots, e_n)$ the standard basis of $E^n$, and $(\tp{e_1},\hdots,\tp{e_n})$
its dual basis. We always consider $G_n$ as the subgroup of $G_{n+1}$
acting trivially on $e_{n+1}$ the last vector of the canonical basis of $E^{n+1}$.

Let $\widetilde{S}= L^\vee  \times L \times E$ be the Heisenberg group over $E$ whose
multiplication is given by
    \[
    (u,v,t) \cdot (u',v',t') =
    \left( u+u', v+v', t+t'+\frac{uv'-u'v}{2} \right),
    \]
The center of $\widetilde{S}$ consists of elements of the form $(0, 0, t)$, $t \in E$ and
is isomorphic to $\G_{a, E}$. Put $S = \Res_{E/F} \widetilde{S}$.

The group $G_n$ acts on the left on $S$ as group automorphism by $g \cdot
(u,v,t)=(ug^{-1},gv,t)$. We define the Jacobi group to be the semi-direct
product
    \[
    J_n =  S \rtimes G_n
    \]
The center of $J_n$ is $Z(J_n)=Z(S) \rtimes \{1\} \subset J_n$ which is
isomorphic to $\Res_{E/F}\G_a$. We denote by $Z$ for both centers of $S$ and $J_n$.

\begin{remark}
The group $J_n$ is not exactly what is usually called a Jacobi group in the literature, but rather the restriction of scalars of them.
\end{remark}

The group $J_n$ satisfies the condition~\eqref{eq:(SR)}, therefore we can
speak of its standard D-parabolic subgroups as in
Subsection~\ref{subsubsec:D-parabolic}. Here we recall that we have the upper
triangular minimal parabolic subgroup $B_n$ of $G_n$ and standard D-parabolic subgroups
of $J_n$ are those that contains $B_n$. We denote by $\cF$ the set of standard
D-parabolic subgroups of $J_n$.

We now give an explicit description of $\cF$. For $P \in \cF$, we put $P_n = P
\cap G_n$ and assume that $P_n$ is the stabilizer of the flag
    \begin{equation}    \label{eq:increasing_sequence}
     0 = L_0 \subset L_1 \subset \cdots \subset L_r=L.
    \end{equation}
Then  $P$ is either of the form
    \begin{equation}    \label{eq:Type 1 D parabolic}
    P = (L_k^\perp \times L_k \times E) \rtimes P_n
    \end{equation}
for some $0 \leq k \leq r$, in which case we call $P$ of type I,  or of the form
    \begin{equation}    \label{eq:Type 2 D parabolic}
    P  = (L_k^\perp \times L_{k+1} \times E) \rtimes P_n
    \end{equation}
for $0 \le k \le r-1$, in which case we call $P$ of type II. If $P$ is of type I,
the D-Levi decomposition for $P$ is
    \[
     M_{P}  = (0 \times 0 \times E) \rtimes M_{P_n}, \quad
    N_{P}  = (L_k^\perp \times L_k \times 0) \rtimes N_{P_n}.
    \]
If $P$ is of type II, the D-Levi decomposition for $P$ is
    \[
        M_{P}
        = (L_k^\perp/L_{k+1}^{\perp} \times L_{k+1}/L_k \times E)
        \rtimes M_{P_n}, \quad
        N_{P}  = (L_{k+1}^\perp \times L_k \times 0) \rtimes N_{P_n}.
    \]
It is readily checked that $P$ is of type II if and only if $P=P(\lambda)$ with
some component of $\lambda:\bG_m \to T$ being the trivial character, otherwise
it is of type I.

We also define the Heisenberg part of these groups. For $X \in \{P,M_P,N_P\}$
put
    \[
    X_S = X \cap S, \quad
    X_{L^\vee}=X \cap L^\vee, \quad X_L= X \cap L.
    \]

Let $\cF_{\mathrm{RS}}$ be the set of Rankin-Selberg parabolic subgroups of
$G_{n+1} \times G_n$ introduced in ~\cite{BPCZ}*{Section 3.1}. The set
$\cF_\mathrm{RS}$ consists of semistandard parabolic subgroups $P_{n+1}
\times P_n$ of $G_{n+1} \times G_n$, such that $P_n = P_{n+1} \cap G_n$ and
$P_n$ is standard. Let $P \in \cF$ be a standard parabolic subgroup of $J_n$. We construct
a Rankin--Selberg parabolic subgroup $P_{n+1} \times P_n \in \cF_{\mathrm{RS}}$
as follows. First let $P_n = P \cap G_n$. Assume that $P_n$ stabilizes the
flag~\eqref{eq:increasing_sequence}. If $P$ is of type I and is of the
form~\eqref{eq:Type 1 D parabolic}, then let $P_{n+1}$ be
the parabolic subgroup of $G_{n+1}$ stabilizing the flag
    \begin{equation} \label{eq:type_1_flag}
    0=L_0 \subset \cdots \subset L_k \subset L_k
    \oplus \mathrm{span}_E(e_{n+1}) \subset \cdots \subset L_r
    \oplus \mathrm{span}_E(e_{n+1}).
    \end{equation}
If $P$ is of type II and is of the form~\eqref{eq:Type 2 D parabolic},
then let $P_{n+1}$ be the parabolic subgroup of $G_{n+1}$
stabilizing the flag
    \begin{equation} \label{eq:type_2_flag}
    0=L_0 \subset \cdots  \subset
    L_k \subset L_{k+1}\oplus \mathrm{span}_E(e_{n+1}) \subset \cdots \subset
    L_r \oplus \mathrm{span}_E(e_{n+1}).
    \end{equation}

\begin{lemma}   \label{lemma:RS_parabolic_bijection}
Let the notation be as above. The map $P \mapsto P_{n+1} \times P_n$ is a bijection
from $\cF$ to $\cF_{\mathrm{RS}}$. Moreover, for any $P  \in \cF$, the map
    \[
    \Psi_{P_{n+1}} \to \Psi_{P} , \quad \alpha \mapsto \alpha|_{A_n}
    \]
is a bijection. Here we recall that $\Psi_P$ (resp. $\Psi_{P_{n+1}}$)
stands for the roots of $A_n$ (resp. $A_{n+1}$) on $\fn_{P}$ (resp. $\fn_{P_{n+1}}$).
\end{lemma}

\begin{proof}

This is clear from the above construction.
\end{proof}

Let $P = MN \in \cF$. The subgroup $P(F)$
or $M(F)N(\bA)$ of $J_n(\bA)$ satisfies the condition (SL) in
Subsection~\ref{subsec:spaces_of_function}, and hence it makes sense to speak of the various spaces of functions on $P(F) \bs J_n(\bA)$ or $[J_n]_P$ (with the character $\psi$) defined there. As mentioned in Subsection \ref{subsubsec:gamma}, we need to verify that in this case, for any compact subgroup $U \subset J_n(\bA_f)$, there exists an open compact subset $C \subset J_n(\bA_f)$ such
that the support of any $\phi \in C^\infty(P(F) \bs J_n(\bA), \psi_E)^U$ is contained in
$\Gamma (C \times J_n(F_\infty))$. We will see in the argument
that this crucially relies on the fact that $\psi_E$ is nontrivial.

Since $N(F) \bs N(\bA)$ is
compact, it is enough to assume that $\phi \in C^\infty(M(F)N(\bA) \bs J_n(\bA), \psi)^U$. For simplicity of notation, we assume that $P$ is of
type I. The type II case can be treated in the same way. Assume that
    \[
    L_k = \mathrm{span}_E(e_1, \hdots, e_a), \quad
    L_{k}^\vee = \mathrm{span}_E(\tp{e_1}, \hdots, \tp{e_a}).
    \]
We view them as subgroups of $S$.  We write an element in $J_n(\bA)$ as
    \[
     (u+v, u^\vee+v^\vee, t) g_f \widetilde{g}_\infty
    \]
where $g_f \in G_n(\bA_f)$, $\widetilde{g}_{\infty} \in J_n(F_\infty)$,
$u \in L_k(\bA_f)$, $v \in L_k^{\vee, \perp}(\bA_f)$,
$u^\vee \in L_k^\vee(\bA_f)$, $v^\vee
\in L_k^\vee(\bA_f)$. First we know that there is a compact subset $C_1$ of
$G_n(\bA)$ such that
    \[
    P_n(F) (C_1 \times G_n(F_\infty)) = G_n(\bA).
    \]
Thus we need to explain that for any open compact subgroup $U$ of $J_n(\bA_f)$,
there are compact subsets $C_2 \subset
L_k^\vee(\bA_f)$ and $C_3 \subset L_{k}^{\vee, \perp}(\bA_f)$ such that if
$\phi \in C^\infty(M(F)N(\bA) \bs J_n(\bA), \psi_E)^U$ and
    \[
    \phi(  (u+v, u^\vee+v^\vee, t) g_f \widetilde{g}_\infty) \not=0,
    \]
then $u^\vee \in C_2$ and $v \in C_3$. This can be seen as follows. First we may assume
that $g_f \in C_1$. Then consider
    \[
    \bigcap_{g_f \in C_1} g_f U g_f^{-1}.
    \]
Since $U$ is an open compact subgroup of $G(\bA_f)$ and $C_1$ is compact, this is
essentially a finite intersection and hence an open subgroup of $G(\bA_f)$. It follows that
    \[
    U' = \bigcap_{g_f \in C_1} g_f U g_f^{-1} \cap N_{L^\vee}(\bA)
    \]
is a nontrivial open compact subgroup of  $N_{L^\vee}(\bA_f)$.
We pick $(x, y, 0)$ in $U'$. Then we have
    \[
    \begin{aligned}
    \phi((u+v, u^\vee+v^\vee, t) g_f \widetilde{g}_\infty ) &=
    \phi((x, y, 0) (u+v, u^\vee+v^\vee, t)  g_f \widetilde{g}_\infty) \\
    &=
    \psi(u^\vee x + y v) \phi( (u+v, u^\vee+v^\vee, t) (x, y, 0) g_f \widetilde{g}_\infty)\\
    & =
    \phi( (u+v, u^\vee+v^\vee, t) g_f \widetilde{g}_\infty).
    \end{aligned}
    \]
Here the first equality is because $\phi$ is left $N_{L^\vee}(\bA)$-invariant, the second is
because $\phi$ is $\psi$-invariant by the central element, and the third is because
$g_f^{-1} (x, y, 0) g_f \in U$ be our choices. It follows that if
$\phi( (u+v, u^\vee+v^\vee, t) g_f \widetilde{g}_\infty) \not=0$ then $\psi(u^\vee x + y v)  = 1$.
As $(x, y, 0)$ varies in  the group $U'$, we conclude that $u^\vee$ and $v$ should lie
in an open compact subgroup of $L_k^\vee(\bA_f)$ and $L_k^{\vee, \perp}(\bA_f)$ which only depends
on $U'$. Note that this is where the fact that $\psi$ is nontrivial is used.
This proves that $M(F)N(\bA)$, and hence $P(F)$ satisfy the condition (SL).

For $P \in \cF$, the embedding $G_n \hookrightarrow G_{n+1}$
induces an embedding
    \[
    [G_n]_{P_n} \hookrightarrow [G_{n+1}]_{P_{n+1}}.
    \]
We can then restrict any weight function on $[G_{n+1}]_{P_{n+1}}$ to obtain a
weight function on $[G_n]_{P_n}$.
The embedding also induces an embedding $A_n^\infty \to A_{n+1}^\infty$.

\begin{lemma}   \label{lem:weight_J_GL}
For $P, Q \in \cF$, we
have $d_{P}^{Q} \sim d_{P_{n+1}}^{Q_{n+1}}|_{[G_n]_{P_n}}$. Here $d_{P}^Q$ is the weight
on $[G_n]_{P_n}$ defined in~\eqref{eq:weight_d_P_Q}.
\end{lemma}

\begin{proof}
It is enough to prove that
    \[
    d_{P}^Q(a) \sim d_{P_{n+1}}^{Q_{n+1}}(a), \quad
    a \in A_n^\infty.
    \]
By the definitions of $d_{P_{n+1}}^{Q_{n+1}}$ and $d_P^Q$ (see also Remark~\ref{rk:after_the_proof} after this proof),
for all $a \in A_n^\infty$ we have
    \[
    d_{P_{n+1}}^{Q_{n+1}}(a) \sim
    \min_{\alpha \in \Psi_{P_{n+1}}^{Q_{n+1}}} \alpha(a), \quad
    d_{P}^{Q}(a) \sim \min_{\alpha \in \Psi_{P}^{Q}} \alpha(a).
    \]
The lemma then follows from Lemma ~\ref{lemma:RS_parabolic_bijection}.
\end{proof}

\begin{remark}
\label{rk:after_the_proof}
We temporarily denote by $G$ a reductive group with a fixed minimal parabolic subgroup
$P_0 = M_0N_0$ and a maximal split torus $A_0 \subset M_0$.
Let $P, Q$ be two standard parabolic subgroups. In general, by definition we have
    \[
    d_{P, \alpha}(x) \sim e^{\langle \alpha, H_0(x) \rangle},
    \quad x \in \fs^P,
    \]
and
    \[
    d_P^Q(x) = \min_{\alpha \in \Psi_{P}^Q} d_{P, \alpha}(x), \quad
    x \in [G]_P.
    \]
Put $W^P = \mathrm{Norm}_{M_P(F)}(A_0)/M_0(F)$ and let $\Lambda \subset \fa_0^*$
be a $W^P$-invariant subset. Then for all $a \in A_0^\infty$, we have
    \begin{equation} \label{eq:min_d_Plambda_on_A}
    \min_{\lambda \in \Lambda} d_{P,\lambda}(a)
    \sim \min_{\lambda \in \Lambda} e^{\langle \lambda, H_0(a) \rangle}.
     \end{equation}
So in particular, for all $a \in A_0^\infty$, we have
    \begin{equation} \label{eq:d_P^Q_on_A}
    d_P^Q(a) \sim \min_{\alpha \in \Psi_P^Q} e^{\langle \alpha, H_0(a) \rangle}
    \end{equation}
As a consequence,~\eqref{eq:d_P^Q_on_A} holds for all semi-standard $P,Q$.
\end{remark}

\subsection{Theta functions: general linear groups}    \label{subsec:theta_series}
We now introduce theta functions on $J_n(\bA)$. Recall that we have
fixed a character $\mu:E^\times \bs \bA_E^\times \to \C^\times$ whose
restriction to $\bA^\times$ equals $\eta$. Define an action of $J_n(\bA)$ on
$\cS(\bA_{E, n})$ by
    \begin{equation}    \label{eq:weil_GL}
    \mathrm{R}_{\mu^{-1}}((u,v,t)g)\Phi(x) = \mu(\det g)^{-1}\abs{\det g}_E^{1/2}
    \Phi((x+u)g) \psi_E \left(xv+\frac{uv}2 + t \right)
    \end{equation}
This is an SLF-representation of $J_n(\bA)$.

For $P = MN$ a standard D-parabolic subgroup of $J_n$, we define the theta
series $\prescript{}{P}{\Theta}(\cdot, \Phi)$ on $[J_n]_{P}$ as
\begin{equation}
\label{eq:theta_gln_P}
     \prescript{}{P}{\Theta}(j, \Phi) =
    \int_{N_{L^\vee}(\bA)} \sum_{m \in M_{L^\vee}(F)} \mathrm{R}_{\mu^{-1}}(j)\Phi(m+n) dn, \quad
    j \in [J_n]_{P}.
\end{equation}
When $P = J_n$, we omit the left subscript and write simply $\Theta$. If $g \in
G_n(\bA)$, we have
    \[
    \Theta(g, \Phi) = \mu(\det g)^{-1} \abs{\det g}^{\frac{1}{2}}
    \sum_{x \in E_n} \Phi(xg),
    \]
This is closely related to the mirabolic Eisenstein series, cf.~\cite{JS}*{Section~4.1}, which
differs essentially by a term corresponding to $x= 0$ and an integral along the center.

\begin{lemma}   \label{lemma:theta_moderate_growth}
There is an $N>0$ such that
$\prescript{}{P}{\Theta}(\cdot, \Phi) \in
\cT_N([J_n]_{P},\psi_E)$ for $\Phi \in \cS(\bA_{n})$.
\end{lemma}

\begin{proof}
The fact that $\prescript{}{P}{\Theta}(\cdot, \Phi)$ is invariant under
$M_{P}(F) N_{P}(\bA)$ and transform by $\psi_E$ under multiplication by $Z$ is
a direct calculation.

We next show that $\prescript{}{P}{\Theta}$ is of uniform moderate growth.
Recall that a height function $\aabs{\cdot}$ is fixed on $\bA_{E, n}$
by~\eqref{eq:height_vector_spaces}. By~\cite{BP}*{Proposition~A.1.1~(v)}
we can pick a large $N_1$ such that
    \[
    \int_{N_{L^\vee}(\bA)} \sum_{m \in M_{L^\vee}(F)}
    \aabs{mn}^{- N_1} \rd n
    \]
is convergent. For any $\Phi \in \cS(\bA_{E, n})$ and $N_1>0$, we put
    \[
    \aabs{\Phi}_{N_1} =
    \sup_{x \in \bA_{E, n}} \aabs{x}^{N_1} \abs{\Phi(x)}
    \]
Then we are reduced to showing that we can find an $N_2$ such that
for all $X \in \cU(\fj_{\infty})$ (where $\fj_{\infty}=\mathrm{Lie}(J_n(F_{\infty})))$ and $\Phi \in \cS(\bA_{E, n})$ we have
    \[
    \sup_{j \in J_n(\bA)} \aabs{j}^{-N_2}
    \aabs{\mathrm{R}_{\mu^{-1}}(j)(\mathrm{R}(X) \Phi)}_{N_1} < \infty.
    \]
This can be checked directly from the definition~\eqref{eq:weil_GL}
the action $\mathrm{R}_{\mu^{-1}}$.
\end{proof}

\begin{lemma} \label{lem:property_Theta}
For $P \subset Q \in \cF$, we have
    \[
    \int_{[N_{P_S}]} \prescript{}{Q}{\Theta}(nj,\Phi) \rd n
    = \prescript{}{P}{\Theta}(j,\Phi), \quad
    \sum_{\gamma \in N_{P_S}(F) \bs N_{Q_S}(F)}
    \prescript{}{P}{\Theta}(\gamma j, \Phi) = \prescript{}{Q}{\Theta(j, \Phi)}.
    \]
\end{lemma}

\begin{proof}
This is a direct calculation using~\eqref{eq:weil_GL}.
\end{proof}

\subsection{Jacobi groups: unitary groups}  \label{subsec:u_Jacobi}
Let $(V, q_V)$ be a nondegenerate $n$-dimensional skew-Hermitian vector
space over $E$ and $\U(V)$ the corresponding unitary group.

Let $\Res V$ be the symplectic space over $F$ whose underline vector space is
$V$, and the symplectic form $\mathrm{Tr}_{E/F} \circ q_V$.
Let $S(V) = \Res V \times F$ be the Heisenberg group, where the
multiplication is given by
    \[
    (v_1, t_1) \cdot (v_2, t_2) =
    \left(v_1+v_2, t_1+t_2+ \frac{1}{2} \Tr_{E/F} q_V(v_1,v_2)\right).
    \]
The group $\U(V)$ acts on $S(V)$ by $x \cdot (v,t)=(xv,t)$ for $x \in \U(V)$
and $(v,t) \in S(V)$. Define the Jacobi group to be the semi-direct product
    \[
    \label{eq:Jacobi_group}
    J(V) =  S(V) \rtimes \U(V).
    \]
The center of $S(V)$ and $J(V)$ are both isomorphic to $\G_a$, we use $Z$ to
denote either of them.

Let $m$ be the Witt index of $V$, and $V_{\mathrm{an}}$ be an anisotropic
kernel of $V$. Choose a basis
    \[
    e_1, \hdots, e_m, e_1^\vee, \hdots, e_m^\vee
    \]
of the orthogonal complement of $V_{\mathrm{an}}$, such that
    \[
    q_V(e_i, e_j) = q_V(e_i^\vee, e_j^\vee) = 0, \quad
    q_{V}(e_i, e_j^\vee) = \delta_{ij}, \quad
    1 \leq i, j \leq m.
    \]

Let $P_{0}$ be the minimal parabolic subgroup of $\U(V)$ stabilizing the
flag
    \begin{equation}    \label{eq:maximal_isotropic_flag}
    0 \subset \mathrm{span}_E(e_1)  \subset \mathrm{span}_E(e_1,e_2)
    \subset \cdots \subset \mathrm{span}_E(e_1,\cdots,e_m) .
    \end{equation}
Let $A_{0}$ be the maximal split torus contained in $P_{0}$. As in the
general linear group case, the Jacobi group $J(V)$ satisfies the
condition~\eqref{eq:(SR)}, hence we can speak of the standard D-parabolic
subgroups of $J(V)$. Let $\cF_V$ be the set of standard D-parabolic subgroups
of $J(V)$ and $\cF'_{V}$ be the set of standard parabolic subgroups of $\U(V)$.
If $P \in \cF_{V}$, then we put $P' = P \cap \U(V) \in \cF_V'$.

\begin{lemma}   \label{lemma:Jacobi_parabolic_subgroup_bijection}
The map
    \[
    \cF_V \to \cF'_{V}, \quad P \mapsto P'
    \]
is a bijection.
\end{lemma}

\begin{proof}
By definition, $P' \in \cF'_{V}$ is standard parabolic subgroup. Assume that
it stabilizes an isotropic flag of the form
    \begin{equation} \label{eq:isotropic flag}
    0=X_0 \subset X_1 \subset \cdots \subset X_r
    \end{equation}
in $V$. Then there is a unique $P \in \cF_V$ such that $P \mapsto P'$
given by $P = (X_r^\perp \times F) \rtimes P'$.
\end{proof}

Assume that $P \in \cF_V$ and $P' = P \cap \U(V)$ stabilizes the isotropic
flag~\eqref{eq:isotropic flag}. Assume that $X_r = \mathrm{span}_E(e_1, \hdots, e_{a_r})$. We put
    \[
    X_r^\vee = \mathrm{span}_E(e_1^\vee, \hdots, e_{a_r}^\vee), \quad
    W_r = (X_r + X_r^\vee)^\perp.
    \]
Then the D-Levi decomposition for $P$ is given by
    \[
     M_P = (W_r \times F) \rtimes M_{P'}, \quad N_P = V_r \rtimes N_{P'}
    \]
In particular, $M_P$ is a product of general linear groups and of the Jacobi group
attached to the skew-hermitian space $W_r$.

We also define the Heisenberg part of these groups. For $X \in \{P,M_P,N_P\}$
put
    \[
    X_S = X \cap S(V), \quad    X_{V}=X \cap V.
    \]
Then we have
    \[
    P_V = X_r + W_r , \quad  M_{P_V} = W_r , \quad N_{P_V} = X_r.
    \]

Recall that we have defined weights $d_P^Q$ and $d_{P'}^{Q'}$
at the end of Subsection~\ref{subsubsec:D-parabolic}.

\begin{lemma} \label{lem:d_P^Q_unitary_equivalent}
For $P,Q \in \cF_V$ with $P \subset Q$, then as weights on
$[\U(V)]_{P'}$, we have
        \[
        \min \{d_{P'}^{Q'},\, (d_{P'}^{Q'})^{\frac 12} \}
        \ll d_{P}^{Q} \ll d_{P'}^{Q'}.
        \]
\end{lemma}

\begin{proof}
By Lemma~\ref{lemma:weight_restriction} we have
$d_{P}^{Q} \ll d_{P'}^{Q'}$, which gives the second inequality.

Assume that $P'$ stabilizes the flag ~\eqref{eq:isotropic flag}, and
that $Q'$ stabilizes another flag
   \[
        0 = X_0 \subset X_{k_1} \cdots \subset X_{k_s},
    \]
    with $1 < k_1 < \cdots < k_s \le r$, obtained by deleting some of the terms in the
flag~\eqref{eq:isotropic flag}. If $X_{k_s} = X_r$,
then one checks that $\fn_P^Q = \fn_{P'}^{Q'}$, and hence the
lemma holds automatically. Assume that $X_{k_s} \not= X_r$, and write
    \[
    V_{k_s} = \mathrm{span}_E(e_1, \hdots, e_{a_s}), \quad
    V_r = \mathrm{span}_E(e_1, \hdots, e_{a_r}).
    \]
Then $\fn_P^Q = \fn_{P'}^{Q'} \oplus
\mathrm{span}_E(e_{a_{s}+1}, \hdots, e_{a_r})$.
For $i = a_{s}+1, \hdots, a_r$, let $\lambda_i \in \Psi_P^Q$
be the root that appears in  $\mathrm{span}_E(e_i) \subset
\mathrm{span}_E(e_{a_{s}+1}, \hdots, e_{a_r})$. Then
we observe that $2 \lambda_i \in \Psi_{P'}^{Q'}$ for all $i$.
It follows that we have
        \[
        e^{\langle \lambda_i, H_0(x) \rangle} =
        \left( e^{\langle 2 \lambda_i,H_0(x) \rangle} \right)^{\frac 12}
        \gg d_{P'}^{Q'}(x)^{\frac 12}, \quad
        x \in \fs^{P'}.
        \]
Therefore for $x \in \fs^{P'}$ we have
        \[
        d_{P}^{Q}(x) \sim
        \min \{ d_{P'}^{Q'}(x), \min_{a_{s}+1 \leq i \le a_r}
        e^{\langle \lambda_i, H_0(x) \rangle} \}
        \gg \min \{d_{P'}^{Q'}(x),d_{P'}^{Q'}(x)^{\frac 12} \}.
        \]
This gives the first inequality and concludes the proof.
\end{proof}

\subsection{Theta functions: unitary groups}  \label{subsec:theta_series_U}
Let $V$ be a nondegenerate $n$-dimensional skew-Hermitian space. Let $S(V)$ be the associated Heisenberg group. Fix a polarization $\Res V = L \oplus L^\vee$,
i.e. $L$ and $L^\vee$ are maximal isotropic subspaces of $\Res V$ such that the
pairing $\Tr_{E/F} q_V$ is nondegenerate when restricted to $L \times L^\vee$. We
denote by $\rho = \rho_{\psi}$ the oscillator representation of $S(V)(\bA)$
on $\cS(L^\vee(\bA))$. For $\phi \in \cS(L^\vee(\bA)),$ it is characterized by
    \begin{equation}    \label{eq:oscillating}
    \rho((l+l', z))\phi(x) =
    \psi\left( z + (x, l) + \frac{1}{2} \Tr_{E/F} q_V(l', l) \right)
    \phi(x+l'), \quad l \in L(\bA), \quad x, l' \in L^\vee(\bA).
    \end{equation}
A different choice of the polarization gives another model, and the isomorphism
between the two models are given by a partial Fourier transform,
cf.~\cite{MVW}*{Chapitre~2, I.~7}.

Let $\Mp(\Res V)(\bA)$ be the metaplectic group
attached to $\Res V$, which sits in a central extension
    \[
    1 \to \C^1 \to \Mp(\Res V)(\bA) \to \Sp(\Res V)(\bA) \to 1,
    \]
where $\C^1$ stands for the complex numbers of norm $1$.
Note that even though we use this notation, $\Mp(\Res V)$ is not an algebraic
group, and indeed does not make sense on its own. A theorem of Weil implies
that the oscillating representation canonically extends to a representation
of the group
    \[
    S(V)(\bA) \rtimes \Mp(\Res V)(\bA).
    \]
By definition there is an embedding $\U(V) \to \Sp(\Res V)$. Recall that we
have fixed a character $\mu:E^\times \to \bA_E^\times$ extending the
quadratic character $\eta$. Given such a $\mu$, there is an explicit lift of
this embedding $\iota_{\mu}: \U(V)(\bA) \to \Mp(\Res V(\bA))$, cf.~\cite{Kudla}.
In this way we obtain a representation of $J(V)(\bA)$, realized on
$\cS(L^\vee(\bA))$. This representation depends on two characters $\psi$ and
$\mu$, and we denoted it by $\omega_{\psi, \mu}$, or simply $\omega$ when the
characters are clear from the context.

We define the theta function
\begin{equation}
    \label{eq:Theta_unitary_definition}
     \theta_{\psi, \mu}(j, \phi) = \sum_{x \in L^\vee(F)}
    \omega_{\psi, \mu}(j) \phi(x), \quad j \in J(V)(\bA), \quad
    \phi \in \cS(L^\vee(\bA)).
\end{equation}
As the different models are related by partial Fourier transforms, the Poisson summation formula ensures that the theta function is independent of the choice of polarization.

For our purposes, it is more convenient to choose a specific polarization $V=L\oplus L^\vee$ and interpret the model
$\cS(L^\vee(\bA))$ as a mixed model,
where the actions of parabolic subgroups are transparent. We follow the exposition
in~\cite{GI2}. Though only local non-Archimedean cases were considered in~\cite{GI2}, the
formulae listed there are also valid in the Archimedean case, and taking product
gives the formulae in the global situation. Recall
that we have fixed a minimal parabolic subgroup $P_{0}$ of $\U(V)$
stabilizing the maximal isotropic flag~\eqref{eq:maximal_isotropic_flag}. We fix a polarization $\Res V_{\mathrm{an}} = L_{\mathrm{an}} \oplus L_{\mathrm{an}}^\vee$. Put
\[
     L = L_{\mathrm{an}} \oplus \mathrm{span}_E(e_1,\hdots,e_m), \quad L^\vee = L_{\mathrm{an}}^\vee \oplus \mathrm{span}_E (e_1^\vee,\hdots,e_m^\vee).
\]
Then $\Res V
= L \oplus L^\vee$ is a polarization of $\Res V$.

Let $1\leq k \leq m$ be a fixed integer. We put
\begin{align*}
     &X = \mathrm{span}_E(e_1, \hdots, e_k), \quad
    X^\vee =\mathrm{span}_E(e_1^\vee, \hdots, e_k^\vee), \\
    & Y = \mathrm{span}_E(e_{k+1}, \hdots, e_m), \quad
    Y^\vee =\mathrm{span}_E(e_{k+1}^\vee, \hdots, e_m^\vee ).
\end{align*}
Set $V_0 =Y \oplus V_{\mathrm{an}} \oplus Y^\vee$. We have a decomposition $V =
X \oplus V_0 \oplus X^\vee$. We view $X, X^\vee, Y, Y^\vee$
as $F$-subspaces of $\Res V$. Put
    \[
    L_0 = L_{\mathrm{an}} \oplus Y, \quad
    L_0^\vee = L_{\mathrm{an}}^\vee \oplus Y^\vee .
    \]
Then $\Res V_0 = L_0 \oplus L_0^\vee$ is a polarization of $\Res V_0$.

We denote by $\omega_{0}$ the Weil representation of $J(V_0)$ , realized on
$\cS_{0} = \cS(L_0^\vee(\bA))$. Via the canonical isomorphism
$\cS(L^\vee(\bA)) \simeq \cS(X^\vee(\bA)) \otimeshat \cS_0$, we view elements in
$\cS(L^\vee(\bA))$ as Schwartz functions on $X^\vee(\bA)$
valued in $\cS_0$. We now describe the action. Let $P'= M N \subset \U(V)$
be the maximal parabolic subgroup stabilizing the isotropic subspace $X$ in
$V$. We write elements in $P$ as $m_P(a) g_0  n_P(b) n_P(c)$ where $g_0 \in
\U(V_0)$, $a \in \GL(X)$, $b \in \Hom(V_0, X)$, $c \in \herm(X^\vee, X)$, and
    \[
    m_P(a) = \begin{pmatrix} a \\ & \id_{V_0} \\ && (a^*)^{-1} \end{pmatrix},
    \quad
    n_P(b) = \begin{pmatrix} \id_X & b & -\frac{1}{2} b b^* \\
    & \id_{V_0} & -b^* \\ && \id_{X^\vee} \end{pmatrix},
    \quad
    n_P(c) = \begin{pmatrix} \id_X & 0 &c \\ & \id_{V_0} & 0 \\
    && \id_{X^\vee} \end{pmatrix},
    \]
and
    \[
    \herm(X^\vee, X) = \{ c \in \Hom(X^\vee, X) \mid c^* = -c \}.
    \]
Here $a^* \in \GL(X^\vee)$, $b^* \in \Hom(X^\vee, V_0)$ and $c^* \in
\Hom(X^\vee, X)$ are defined by
    \[
    q_V(ax, x^\vee) = q_V(x, a^* x^\vee), \quad
    q_V(b v, x^\vee) = q_V(v, b^* x^\vee), \quad
    q_V(c x^\vee, y^\vee) = q_V(x^\vee, c^* y^\vee),
    \]
for $v \in V_0$, $x \in X$ and $x^\vee, y^\vee \in X^\vee$. We list
the actions of various elements in $J(V)$ on $\cS(L^\vee(\bA))$. Take $\phi
\in \cS(L^\vee(\bA))$ and $x \in X^\vee(\bA)$. The action of element in $P$
are given as follows.

    \begin{align}
    \omega(g_0)\phi(x) &= \omega_0(g_0)(\phi(x)) \label{eq:weil_formula_P1},\\
    \omega(m_P(a)) \phi(x) & = \mu(\det a) \abs{\det a}^{\frac{1}{2}}
    \phi(a^* x) \label{eq:weil_formula_P2},\\
    \omega(n_P(b))\phi(x) &= \omega_0((b^* x, 0)) (\phi(x)) \label{eq:weil_formula_P3} ,\\
    \omega(n_P(c))\phi(x) & = \psi(\frac{1}{2} \Tr_{E/F} q_V(cx, x)) (\phi(x)) \label{eq:weil_formula_P4}.
    \end{align}
Define an isomorphism $I_X: X^\vee \to X$ by $I_X(e_i^\vee) = e_i$, and put
    \[
    w_P = \begin{pmatrix} && -I_X \\ & \id_{V_0} \\ I_X^{-1} \end{pmatrix}
    \in \U(V).
    \]
Then
    \begin{equation}    \label{eq:weil_formula_weyl}
    \omega(w_P)\phi (x) = \int_{X(\bA)} \phi(-I_X^{-1} y)
    \psi(\Tr_{E/F}q_V(y, x)) \rd y.
    \end{equation}
Finally elements in $S(V)$ acts as follows. For $u \in X$, $u^\vee \in
X^\vee$, $v_0 \in V_0$ and $\phi \in \cS(X^\vee(\bA))$ we have
    \begin{equation}\label{eq:weil_formula_Heisenberg}
    \omega((u + v_0 + u^\vee, 0)) \phi(x) =
    \psi(\Tr_{E/F} q_V(x, u) + \frac{1}{2} \Tr_{E/F} q_V(u^\vee, u))
    \omega_0 ( (v_0, 0)) (\phi(x+u^\vee)),
    \end{equation}

We write $\cS_{\mathrm{an}} =
\cS(L_{\mathrm{an}}^\vee(\bA))$ for a model of the representation
$\omega_{\mathrm{an}}$ of $J(V_{\mathrm{an}})(\bA)$.
In a similar fashion, via the canonical isomorphism
$\cS(L_0^\vee(\bA)) \simeq \cS(X_0^\vee(\bA)) \otimeshat \cS_{\mathrm{an}}$,
we may interpret elements in $\cS(L_0^\vee(\bA))$ as Schwartz functions on $X_0(\bA)$
valued in $\cS_{\mathrm{an}}$ and $\cS(L_0^\vee(\bA))$ as a mixed model.

We now define theta functions in general. Let $P \in \cF_V$ be a standard parabolic
subgroup of $J(V)$ and let $X \in \{ P,M_P,N_P \}$. We set
 \[
 X_{L} = X \cap L, \quad X_{L^\vee} = X \cap L^\vee.
 \]
 We define
    \[
    \prescript{}{P}{\theta(j, \phi)} =
    \sum_{x \in M_{P_{L^\vee}}(F)}
    \omega(j)\phi(x), \quad
    j \in J(V)(\bA).
    \]

\begin{prop} \label{prop:u_theta_property}
For any $\phi \in \cS(L^\vee(\bA))$, we have
$\prescript{}{P}{\theta}(\cdot,\phi) \in \cT([J(V)]_{P},\psi)$. Moreover
for any $j \in J(\bA)$, we have
        \[
        \int_{[N_{P_V}]} \theta(nj,\phi) \rd n = {}_P \theta(j,\phi).
        \]
In particular, for any $\varphi \in \cT([\U(V)])$, the constant
term of the function $\varphi(\cdot) \theta(\cdot, \phi)$ on $[J(V)]$
along $P$ equals $\varphi_{P'}(\cdot) \prescript{}{P}{\theta}(\cdot, \phi)$.
\end{prop}

\begin{proof}
The invariance of the theta function by $M(F)N(\bA)$, and the
constant term calculation can be checked directly using mixed models.
That it is of uniformly moderate growth can be proved in exactly the same way
as Lemma~\ref{lemma:theta_moderate_growth}.
\end{proof}

For $P \in \cF_V$, we also put

\begin{equation}
\label{eq:constant_term_theta}
      {}_P \theta^\vee(j,\phi) = \sum_{x \in M_{P_{L^\vee}}(F)} \omega^\vee(j) \phi(x).
\end{equation}
Then ${}_P \theta^\vee \in \cT([J]_P,\psi^{-1})$ and exhibits similar properties as in Proposition ~\ref{prop:u_theta_property}. We have
\[
    {}_P \theta^\vee(j,\phi) = \overline{{}_P \theta(j,\overline{\phi})}.
\]

\section{The coarse spectral expansion: general linear groups}

\subsection{Notation} \label{subsec:main_thm_spectral}
We first list notation that will be used throughout this section.
\begin{itemize}
\item Put $L = E^n$ and $L^\vee = E_n$.

\item Put $G = G_n \times G_n$, $G' = G'_n \times G'_n$, and
    $H = G_n$ which embeds in $G$ diagonally.
    If $g \in G$ or $G'$, we always write $g =
    (g_1, g_2)$ where $g_i \in G_n$ or $G_n'$. We view $G$ as a subgroup of
    $G_{n+1} \times G_n$ via the embedding $G_n \to G_{n+1}$.

\item We write (see Subsection~\ref{subsec:reduction_theory}) $K_G$, $K_H$ and $K_{G'}$ for the standard maximal compact subgroups of $G(\bA)$, $H(\bA)$ and $G'(\bA)$ respectively.

\item Define two characters of $G'(\bA)$ by
\begin{equation}
\label{eq:eta_defi}
      \eta_{n+1}(g_1, g_2) = \eta(\det g_1 g_2)^{n+1}, \quad
    \eta_{G'}(g_1, g_2) = \eta(\det g_1)^{n} \eta(\det g_2)^{n+1}.
\end{equation}

\item For $k = n, n+1$, we have the subgroups $B_k, T_k, A_k$ of $G_k$, and subgroups
    $B_k', T_k'$ of $G_k'$, cf.~Subsection~\ref{subsec:notation_intro_groups}.

\item For $k=n,n+1$, put $\fa_{k} = \fa_{B_{k}}$ and $\fa_{k}' = \fa_{B_k'}$.
    A truncation parameter $T$ is an element in $\fa_{n+1}'$.

\item Recall that we define the Jacobi group $J_n = S \rtimes G_n$ and we put
    $\widetilde{G} = J_n \times G_n$. There is a natural embedding
    $J_n \to \widetilde{G}$ which is identity on the first coordinate and the
    natural projection on the second coordinate. The image is denoted by
    $\widetilde{H}$.

\item Recall the convention that if $Q$ is a D-parabolic subgroup, without saying
    the contrary, we write $Q = M_Q N_Q$ for its D-Levi decomposition.

\item  Recall that $\cF$ is the set of standard D-parabolic subgroup of $J_n$.
    Let
        \[
        P \mapsto P_{n+1} \times P_n
        \]
    be the bijection defined in Lemma ~\ref{lemma:RS_parabolic_bijection}, where
    $P_{k}$ is a semistandard parabolic subgroup of $G_{k}$, $k= n, n+1$, and
    $P_n = P_{n+1} \cap G_n$ is standard.

\item If $P \in \cF$, we put
        \[
        \widetilde{P} = P \times P_n, \quad
        P_G = \widetilde{P} \cap G = P_n \times P_n, \quad
        P_H = \widetilde{P} \cap H.
    \]
    and
        \[
        P_{G'}=\widetilde{P} \cap G', \quad P_n' = P_n \cap G_n',
        \quad P'_{n+1}=P_{n+1} \cap
        G_{n+1}'.
        \]
    They are parabolic subgroups of $\widetilde{G},G,H,G',G_n',G_{n+1}'$
    respectively. We also put $M_G = M_n \times M_n$ and $N_G = N_n \times N_n$
    and hence $P_G = M_G N_G$ is a Levi decomposition.

\item Put $G_+ = G \times E_n$. This is viewed as a group over $F$, with the
    groups structure given by simply the product of the group $G$ and the
    additive group $E_n$. Note that this is not a subgroup of $\widetilde{G}$, but
    merely a subvariety of it. For $P \in \cF$, we
    define $M_{P, +} = M_{G} \times M_{L^\vee}$, $N_{P,+} = N_G \times N_{L^\vee}$,
    and $P_+ = P_G \times P_{L^\vee}$, where $M_{L^\vee}, N_{L^\vee}$ and
    $P_{L^\vee}$ are defined in
    Subsection~\ref{subsec:Jacobi groups}. An element in $P_+$ is often written as
    $(x, u)$ where $x \in P_G$ and $u \in P_{L^\vee}$, or as $m_+n_+$ where
    $m_+ \in M_{P, +}$ and $n_+ \in N_{P, +}$. Note that the product $m_+ n_+$ is
    taken in $G_+$, not in $\widetilde{G}$.

\item There is a right action of $H \times G$ on $G_+$ given by
        \[
        (x,u) \cdot (h,g) = (h^{-1}xg,uh),
        \]
    and it restricts to an action of $M_{P_H} \times M_{P_G}$
    on $P_+$ for any $P \in \cF$.
\end{itemize}

\subsection{Technical preparations}

Let $P \in \cF$ be a standard parabolic subgroup of $J_n$ and $w$ be a weight on $[G]_{P_G}$. We pull $w$ back to a
function on $[\widetilde{G}]_{\widetilde{P}}$, which we still denote by $w$.
Let $N$ be a nonnegative integer. Let $Z \simeq \Res_{E/F} \bG_{a, E}$ be the central unipotent subgroup
$Z \times 1 \subset \widetilde{G}$ and $\psi: [Z] \to \C^\times$ be a non trivial
character. We then have various spaces of functions as defined in
Subsection~\ref{subsec:spaces_of_function}.

For $P,Q \in \cF$ and $P \subset Q$, we define weights on
$[G]_{P_G}$ by
    \[
    \Delta_P(g)=
    \inf_{\gamma \in M_{P_n}(F)N_{P_n}(\bA)} \| g_1^{-1} \gamma g_2 \|, \quad
    d_P^{Q,\Delta}(g) =
    \min \left( d_{P}^{Q}(g_1),d_{P_{n}}^{Q_{n}}(g_2) \right)
    \]
Pulling back under the natural projection $[\widetilde{G}]_{\widetilde{P}}
\to [G]_{P_G}$, we get two weights on
$[\widetilde{G}]_{\widetilde{P}}$, which we still denote by $\Delta_P$ and
$d_P^{Q,\Delta}$.

The ``approximation by constant terms'' for the group $\widetilde{G}$
takes the following form.

\begin{prop} \label{prop:approximating_constant_term_product}
Let $N>0,r \ge 0$ and for $P, Q\in \cF$ with $P \subset Q$. There exists a
continuous semi-norm $\aabs{\cdot}_{N,X,r}$ on $\cT_N(\widetilde{Q}(F) \bs
\widetilde{G}(\bA), \psi)$ such that
    \[
    \abs{\mathrm{R}(X)\varphi(g)-\mathrm{R}(X) \varphi_{\widetilde{P}}(g)}
    \le \aabs{g}_{P_G}^N d_{P}^{Q, \Delta}(g)^{-r}
    \aabs{\varphi}_{N,X,r}
    \]
holds for all $\varphi \in \cT_N([\widetilde{G}]_{\widetilde{Q}},\psi)$
and $g \in G(\bA) \subset \widetilde{G}(\bA)$.
\end{prop}

\begin{proof}
For $g= (g_1, g_2) \in [G]_{P_G}$ we have
    \[
    d_{\widetilde{P}}^{\widetilde{Q}}(g_1,g_2)
    \sim  d_P^Q(g_1) d_{P_n}^{Q_n}(g_2)
    \gg d_P^{Q,\Delta}(g_1,g_2)^2.
    \]
Hence the result follows from
Theorem~\ref{thm:approximation_by_constant_term}.
\end{proof}

We recall some function spaces introduced in \cite{BPCZ}*{Section~3.4.1}.
Let $\cT_{\cF}(G_n)$(resp. $\cT_{\cF}(G'_n)$) be the space of tuples of functions
    \[
    ( {}_P \varphi )_{P \in \cF}
    \in \prod_{P \in \cF} \cT([G_n]_{P_n}), \quad
    \text{resp. } ( {}_P \varphi )_{P \in \cF}
    \in \prod_{P \in \cF} \cT([G'_n]_{P'_n})
    \]
such that
    \[
    {}_P\varphi - ({}_Q \varphi)_{P_n}
    \in \cS_{d_{P}^{Q}}([G_n]_{P_n}),\quad
    \text{resp. } {}_P\varphi - ({}_Q \varphi)_{P'_n}
    \in \cS_{d_P^Q}([G'_n]_{P'_n}),
    \]
for all $P \subset Q \in \cF$.
These spaces agree with the spaces  $\cT_{\cF_\mathrm{RS}}(G_n)$
(resp. $\cT_{\cF_\mathrm{RS}}(G'_n)$) defined in defined in~\cite{BPCZ}
by similar formulae,
where the product ranges over  $\cF_{\mathrm{RS}}$
and the weights are $d_{P_{n+1}}^{Q_{n+1}}$ and $d_{P_{n+1}'}^{Q_{n+1}'}$
respectively.
This is because there is a bijection between $\cF$ and $\cF_{\mathrm{RS}}$, and $d_P^Q
\sim d_{P_{n+1}}^{Q_{n+1}}|_{[G_n]_{P_n}}$ (cf. Lemma~\ref{lem:weight_J_GL}), and
$d_{P_{n+1}}^{Q_{n+1}}|_{[G_n']_{P_n'}} \sim d_{P'_{n+1}}^{Q'_{n+1}}|_{[G_n']_{P_n'}}$
(cf.~\cite{BPCZ}*{Lemma~2.4.4.2}).

By the approximation by constant terms for the
groups $G$ and $G'$ respectively,
cf.~Theorem~\ref{thm:approximation_by_constant_term}, a family of functions
$({}_P \varphi)_{P \in \cF}$ is in
$\cT_{\cF}(G_n)$ (resp.
$\cT_{\cF}(G'_n)$) if and only if there exists an
integer $N_0 \ge 0$, such that for all $P \subset Q \in \cF$ and for any $g \in G_n(\bA)$ (resp. $G_n'(\bA))$,
$X \in \cU(\fg_{n,\infty})$ (resp. $\cU(\fg'_{n,\infty}))$, and $r>0$ we have
    \begin{equation} \label{eq:equivalent_definition_T_space}
    \abs{\mathrm{R}(X) \prescript{}{P}{\varphi}(g) -
    \mathrm{R}(X) \prescript{}{Q}{\varphi}(g) }
    \ll \aabs{g}_{P_n}^{N_0} d_{P}^{Q}(g)^{-r} \quad
    \text{resp. }\aabs{g}_{P'_n}^{N_0} d_{P}^{Q}(g)^{-r}.
    \end{equation}
As explained in~\cite{BPCZ}*{Section~3.4.1}, $\cT_{\cF}(G_n)$ and
$\cT_{\cF}(G'_n)$ are LF spaces.

Let $\cT_{\cF}^{\Delta}(\widetilde{G})$ be the space of tuples of
functions
    \[
    (\prescript{}{P}{\varphi})_{P \in \cF}
    \in \prod_{P \in \cF}
    \cS_{\Delta_P}([\widetilde{G}]_{\widetilde{P}},\psi)
    \]
such that
$\prescript{}{P}{\varphi}-(\prescript{}{Q}{\varphi})_{\widetilde{P}} \in
\cS_{d_P^{Q,\Delta}} ([\widetilde{G}]_{\widetilde{P}},\psi)$ for all $P
\subset Q \in \cF$. By Proposition
\ref{prop:approximating_constant_term_product}, this condition is equivalent
to the existence of $N_0$ such that for all $\widetilde{g} \in \widetilde{G}(\bA)$, all
$X \in \cU(\widetilde{\fg}_\infty)$ and all $r \geq 0$ we have
    \begin{equation}    \label{eq:T_Delta_jacobi}
    \valP{\mathrm{R}(X) \prescript{}{P}{\varphi}(\widetilde{g}) -
    \mathrm{R}(X) \prescript{}{Q}{\varphi}(\widetilde{g})}
    \ll \aabs{\widetilde{g}}_{\widetilde{P}}^{N_0}
    d_{P}^{Q, \Delta}(g)^{-r}.
    \end{equation}

In the same vein, we define $\cT_{\cF}^{\Delta}(G)$ to be the
space of tuples of functions
    \[
    (\prescript{}{P}{\varphi})_{P \in \cF}
    \in \prod_{P \in \cF}
    \cS_{\Delta_P}([G]_{P_{G}})
    \]
such that $\prescript{}{P}{\varphi}-(\prescript{}{Q}{\varphi})_{P_{G}} \in
\cS_{d_P^{Q,\Delta}} ([G]_{P_{G}})$ if $P \subset Q \in \cF$. By
Theorem~\ref{thm:approximation_by_constant_term}, this condition is
equivalent to the existence of an $N_0$ such that for all $g \in G(\bA)$, all $X \in \cU(\fg_\infty)$ and all
$r \geq 0$ we have
    \begin{equation}    \label{eq:T_Delta_G}
    \valP{\mathrm{R}(X) \prescript{}{P}{\varphi}(g) -
    \mathrm{R}(X) \prescript{}{Q}{\varphi}(g)} \ll \aabs{g}_{P_G}^{N_0}
    d_{P}^{Q, \Delta}(g)^{-r}
    \end{equation}

The space $\cT_{\cF}^{\Delta}(\widetilde{G})$
embeds in
    \[
    \prod_{P \in \cF}
    \cS_{\Delta_P}([\widetilde{G}]_{\widetilde{P}},\psi)
    \times \prod_{P \subset Q \in \cF}
    \cS_{d_P^{Q,\Delta}} ([\widetilde{G}]_{\widetilde{P}},\psi)
    \]
as a closed subspace, and hence inherits an LF topology. Similarly $\cT_{\cF}^{\Delta}(G)$
is an LF space.

\begin{lemma}   \label{lemma:restriction_space_T}
If $(\prescript{}{P}{\varphi})_{P \in \cF} \in
\cT_{\cF}^{\Delta}(\widetilde{G})$, then
    \[
    \left(\prescript{}{P}{\varphi}|_{[G]_{P_G}}\right)_{P \in \cF}
    \in \cT_{\cF}^{\Delta}(G).
    \]
\end{lemma}

\begin{proof}
This follows directly from the characterizations~\eqref{eq:T_Delta_jacobi}
and~\eqref{eq:T_Delta_G}.
\end{proof}

Let $w$ be a weight on $[G]_{P_G}$ and $\alpha \in \cT_{w}^0([G]_{P_G})$ be a Radon measure. We define a Radon measure $\alpha \cdot \prescript{}{P}{\Theta}(\cdot, \Phi_0)$ on $[\widetilde{G}]_{\widetilde{P}}$  as follows. For $f \in C_c([\widetilde{G}]_{\widetilde{P}})$, we put
\begin{equation}
\label{eq:Radon_defi}
     \langle \alpha\cdot \prescript{}{P}{\Theta}(\cdot, \Phi_0), f \rangle =
    \int_{[G]_{P_G}} \int_{[S]_{P_S}}
    f(sg) \prescript{}{P}{\Theta(sg_1, \Phi_0)} \rd s
    \alpha(g),
\end{equation}
where $[S]_{P_S} = N_{P_S}(\bA) M_{P_S}(F) \backslash S(\bA)$, and where we recall that $g = (g_1, g_2)$, $g_1, g_2 \in G_n(\bA)$.

\begin{lemma}   \label{lemma:distribution_extension_by_theta}
There is an $N_0>0$ such that for all $\Phi_0 \in \cS(\bA_{E, n})$ we have
    \[
    \alpha
    \cdot \prescript{}{P}{\Theta(\cdot, \Phi_0)} \in \cT_{w,
    N_0}^0([\widetilde{G}]_{\widetilde{P}}, \psi).
    \]
Moreover the map
    \[
    \cT^0_{w}([G]_{P_G}) \to \cT_{w,
    N_0}^0([\widetilde{G}]_{\widetilde{P}}, \psi), \quad
    \alpha \mapsto \alpha
    \cdot \prescript{}{P}{\Theta(\cdot, \Phi_0)}
    \]
is continuous.
\end{lemma}

\begin{proof}
We may pick an $N_0$ such that for all $\Phi_0$ we have
    \[
    \sup_{g_1 \in [G_n]_{P_{n}}} \aabs{g_1}_{P_n}^{-N_0}
    \int_{[S]_{P_S}} \abs{\prescript{}{P}{\Theta}(sg_1, \Phi_0)} \rd s
    < \infty.
    \]
It follows that
    \[
    \begin{aligned}
    &\int_{[\widetilde{G}]_{\widetilde{P}}}
    w(g)^{-1} \aabs{g}_{P_G}^{-N_0}
    \abs{\alpha(g)}
    \abs{\prescript{}{P}{\Theta}(sg_1, \Phi_0)} \rd s\\
     \leq
    &\sup_{g_1 \in [G]_{P_G}} \aabs{g_1}_{P_{G_n}}^{-N_0}
    \int_{[S]_{P_S}} \abs{\prescript{}{P}{\Theta}(s g_1, \Phi_0)} \rd s
    \int_{[G]_{P_G}} w(g)^{-1} \abs{\alpha(g)} < \infty.
    \end{aligned}
    \]
The assertion on continuity follows from the same estimate.
\end{proof}

\begin{lemma}   \label{lemma:constant_term_measure_product_group}
If $P \subset Q \in \cF$ and $\alpha \in \cT^0([G]_{P_G})$,
then
    \[
    (\alpha \cdot \prescript{}{Q}{\Theta}(\cdot, \Phi_0))_{\widetilde{P}}
    = \alpha_{P_G}\cdot \prescript{}{P}{\Theta}(\cdot, \Phi_0).
    \]
\end{lemma}

\begin{proof}
This follows from a direct computation using Lemma ~\ref{lem:property_Theta}.
\end{proof}

\begin{lemma} \label{lem:smoothed_constant_term}
If $\varphi \in \cT^0([H])$, $\Phi_0 \in \cS(\bA_{E, n})$, and $\widetilde{f}
\in \cS(\widetilde{G}(\bA),\psi^{-1})$, then the family
    \[
    P \mapsto \mathrm{R}(\widetilde{f})(\varphi_{P_{H}} \cdot
    \prescript{}{P}{\Theta}(\cdot, \Phi_0))
    \]
belongs to $\cT^{\Delta}_{\cF}(\widetilde{G})$.
\end{lemma}

\begin{proof}
We first note that the support of $\varphi_{P_H} \cdot
\prescript{}{P}{\Theta}(\cdot, \Phi_0)$ is contained in
$[\widetilde{H}]_{P_{\widetilde{H}}}$, and $\Delta_P$ is bounded on it by
definition. Thus by Lemma~\ref{lem:translation_weighted_Schwartz} we have
    \[
    \mathrm{R}(\widetilde{f}) (\varphi_{P_H} \cdot
    \prescript{}{P}{\Theta}(\cdot, \Phi_0)) \in
    \cS_{\Delta_P}([\widetilde{G}]_{\widetilde{P}}, \psi).
    \]
We need to show that
    \[
    \mathrm{R}(\widetilde{f})\left(\varphi_{P_{H}} \cdot
    \prescript{}{P}{\Theta}(\cdot, \Phi_0) -
    (\varphi_{Q_{H}} \cdot
    \prescript{}{Q}{\Theta}(\cdot, \Phi_0))_{\widetilde{P}} \right)
    \in \cS_{d_P^{Q,\Delta}}
    ([\widetilde{G}]_{\widetilde{P}},\psi).
    \]
Indeed we have a slightly stronger result. Define on $[G]_{P_G}$ a weight
    \[
    d_P^{Q,\Delta_n}(g) = \min\{ d_{P_n}^{Q_n}(g_1),d_{P_n}^{Q_n}(g_2)\},
    \]
which pulls back to a weight on $[\widetilde{G}]_{\widetilde{P}}$. We have
$d_P^{Q, \Delta} \ll d_P^{Q, \Delta_n}$ by~\cite{BPCZ}*{Lemma~2.4.4.2}. We
will show that
    \[
    \mathrm{R}(\widetilde{f})\left(\varphi_{P_{H}} \cdot
    \prescript{}{P}{\Theta}(\cdot, \Phi_0) -
    (\varphi_{Q_{H}} \cdot
    \prescript{}{Q}{\Theta}(\cdot, \Phi_0))_{\widetilde{P}} \right)
    \in \cS_{d_P^{Q,\Delta_n}}
    ([\widetilde{G}]_{\widetilde{P}},\psi).
    \]
By Lemma~\ref{lem:translation_weighted_Schwartz}, it suffices to show there
is a positive real number $C$ such that $\varphi_{P_{H}} \cdot
\prescript{}{P}{\Theta}(\cdot, \Phi_0)$ and $(\varphi_{Q_{H}} \cdot
\prescript{}{Q}{\Theta}(\cdot, \Phi_0))_{\widetilde{P}}$ coincide on the set
$\{ \widetilde{g} \in [\widetilde{G}]_{\widetilde{P}} \mid
d_{P}^{Q,\Delta_n}(\widetilde{g})>C\}$.

By Lemma~\ref{lemma:constant_term_measure_product_group} we have
    \[
    (\varphi_{Q_{H}} \cdot
    \prescript{}{Q}{\Theta}(\cdot, \Phi_0))_{\widetilde{P}} =
    (\varphi_{Q_{H}})_{P_G} \cdot
    \prescript{}{P}{\Theta}(\cdot, \Phi_0).
    \]
We are thus reduced to show that there is a constant $C$ such that
$\varphi_{P_{H}}$ and $(\varphi_{Q_{H}})_{P_G}$ coincide on the set $\{ g \in
[G]_{P_{G}} \mid d_{P}^{Q,\Delta_n}(g)>C\}$.

This is proved in the same way as~\cite{BPCZ}*{Proposition~3.4.2.1(1)}. Using
the adjunction relation~\eqref{eq:adjuction_pseudo_constant}, we are reduced
to show that if $\lambda \in \cS^0([G]_{P_G})$ is supported in $\{ g \in
[G]_{P_{G}} \mid d_{P}^{Q,\Delta_n}(g)>C\}$, then we have
    \[
     E_{P_H}^{Q_H} (\lambda|_H) =
     \left(E_{P_G}^{Q_G} \lambda\right)|_{[H]_{Q_H}}.
    \]
This in turn is equivalent to the fact that there is a $C>0$ such that for $g \in
P_G(F)N_{Q_G}(\bA) \bs G(\bA)$ with $d_P^{Q, \Delta_n}(g)>C$, $\pi_{Q_G}^{P_G}(g) \in [H]_{Q_H}$ implies $g \in P_{H}(F)Q_H(\bA) \bs H(\bA)$ (recall $\pi_{Q_G}^{P_G}$ is the map defined in~\eqref{eq:projections_P_Q}).
By definition, $d_{P}^{Q, \Delta_n}(g)>C$ means $d_{P_n}^{Q_n}(g_i) > C$, $i
= 1, 2$, and $\pi_{Q_G}^{P_G}(g) \in [H]_{Q_H}$ means $\pi_{Q_n}^{P_n}(g_1) =
\pi_{Q_n}^{P_n}(g_2)$. By Lemma~\ref{lemma:classical_reduction_theory}(4)
(applied to $G_n$), if $C$ is sufficiently large, this implies $g_1 = g_2$,
or equivalently $g \in P_{H}(F)Q_H(\bA) \bs H(\bA)$.
\end{proof}

Let $({}_P{\varphi})_{P \in \cF} \in
\cT_{\cF}^{\Delta}(G)$ and let $\varphi' \in \cT^0([G_n'])$.
For $g_1 \in G_n(\bA)$, we define a pairing
    \begin{equation}    \label{eq:pairing_tuples}
    \langle \prescript{}{P}{\varphi}, \varphi'_{P_n'} \rangle (g_1) =
    \int_{[G_n']_{P_n'}} \prescript{}{P}{\varphi}(g_1, g_2')
    \varphi'_{P_n'}(g_2') dg_2'
    \end{equation}

We note that by definition for any weight $w$ on
$[G_n]_{P_n}$, there exists $N_0 > 0$ such that
    \begin{equation} \label{eq:bound_on_DeltaP}
    w(g_2) \ll w(g_1) \Delta_P(g)^{N_0}.
    \end{equation}
By taking $w=\| \cdot \|_{P_n}$ in (\ref{eq:bound_on_DeltaP}), we see that
$\Delta_P(g)^{-N_0}  \ll \| g_1 \|_{P_n} \| g_2 \|^{-1}_{P_n} $. Therefore, there
exists $N_1>0$ such that for any $N>0$
\[ {}_P \varphi(g_1,g_2) \ll \| g_1 \|^{N_1+N}_{P_n} \| g_2\|^{-N}_{P_n}.  \]
Since $\|g_2\|_{P_n} \sim \|g_2'\|_{P_n'}$, we see that the integral in
(\ref{eq:pairing_tuples}) is convergent.

\begin{lemma}    \label{lemma:space_T_pairing}
Let the notation be as above. Then the family
    \[
    P \mapsto \langle \prescript{}{P}{\varphi}, \varphi'_{P_n'} \rangle
    \]
belongs to $\cT_{\cF}([G_n])$.
\end{lemma}

\begin{proof}
Using Dixmier--Malliavin theorem for the smooth Fr\'{e}chet representation $\cT^\Delta_\cF(G)^J$ for sufficiently small open compact subgroup $J \subset G_n(\bA_f)$, we see that any $({}_P \varphi) \in \cT^\Delta_{\cF}(G)$
is of the form $(\mathrm{R}(f) {}_P \varphi_1)$ for $f \in
C_c^\infty(G_n(\bA))$ and $\prescript{}{P}{\varphi_1} \in \cT^\Delta_{\cF}(G)$. Thus we are reduced to prove the same statement, but for
    \[
    \langle \prescript{}{P}{\varphi_1},
    \mathrm{R}(f^\vee)\varphi'_{P_n'} \rangle
    \]
where $f \in C_c^\infty(G_n(\bA))$. By~\cite{BPCZ}*{Proposition~3.4.2.1
(2)}, the family of functions
    \[
    P \in
    \cF \mapsto \mathrm{R}(f^\vee) \varphi'_{P'_n}
    \]
satisfies
    \begin{equation} \label{eq:unnamed_auxi_space_1}
    \mathrm{R}(f^\vee) \varphi'_{P'_n} \in \cT([G_n]_{P_n})
    \end{equation}
with
    \begin{equation} \label{eq:unnamed_auxi_space_2}
     \mathrm{R}(f^\vee) \varphi'_{P'_n}-
     (\mathrm{R}(f^\vee) \varphi'_{Q'_n})_{P_n}
     \in \cS_{d_{P_n}^{Q_n}}([G_n]_{P_n}), \quad
     \text{for $P \subset Q \in \cF$}.
    \end{equation}
Note that there is a slight inconsistency of notation in~\cite{BPCZ} here. In
the notation of~\cite{BPCZ}, we indeed apply~\cite{BPCZ}*{Proposition~3.4.2.1
(2)} to the case $\mathbf{G} = G_{n+1}$ and replace $n+1$ by $n$ at all
places.

For each $P \in \cF$, we define a function
$\prescript{}{P}{\beta}$ on $[G]_{P_G}$ by
    \[
    \prescript{}{P}{\beta}(g_1, g_2) = \prescript{}{P}{\varphi}(g_1,g_2)
    \mathrm{R}(f^\vee)
    \varphi'_{P_n'}(g_2)
    \]
By the characterization~\eqref{eq:T_Delta_G}
and~\eqref{eq:unnamed_auxi_space_1},~\eqref{eq:unnamed_auxi_space_2}, the family of functions $\beta = (\prescript{}{P}{\beta})_{P
\in \cF}$ belongs to $\cT_{\cF}^{\Delta}(G)$. Thus we are reduced to show that the family of functions
    \[
    P \mapsto \left( g_1 \mapsto \int_{[G_n]_{P_n}}
    \prescript{}{P}{\beta}
    (g_1, g_2) \rd g_2\right)
    \]
belongs to $\cT_{\cF}(G_n)$.

Again by the characterization~\eqref{eq:T_Delta_G}, we will need to show that there is an integer $N_0>0$ such that for all $({}_P \beta) \in \cT_\cF^\Delta(G)$, $P
\subset Q$, all $X \in \cU(\fg_\infty)$, and $r\geq 0$ we have
    \begin{equation}    \label{eq:wanted_estimate}
    \left| \int_{[G_n]_{P_n}}
    \mathrm{R}_1(X) {}_P \beta(g_1,g_2) \rd g_2 -
    \int_{[G_n]_{Q_n}}  \mathrm{R}_1(X)
    {}_Q{\beta}(g_1, g_2)  \rd g_2 \right|
    \ll_{r, X} \|g_1\|_{P_n}^{N_0} d_{P}^{Q}(g_1)^{-r}.
    \end{equation}
Here $\mathrm{R}_1$ stands for the action of the universal enveloping algebra
action on the first variable $g_1$. Up to replacing $\beta$ by $\mathrm{R}_1(X) \beta$, we can assume $X=1$.

The rest of the argument is similar to \cite{BPCZ}*{Proposition 3.4.3.1}.
By~\eqref{eq:bound_on_DeltaP} applied to $w=\aabs{\cdot}_{P_n}$ and
$w=d_{P_{n}}^{Q_{n}}$ we see that there exists $N_1>0$ such that for any
$N>0$ and $r \ge 0$,
    \begin{equation} \label{eq:bound_product_1}
    \abs{\prescript{}{P}{\beta}(g)} \ll
    \aabs{ g_1 }_{P_n}^{N_1+N} \aabs{g_2}_{P_n}^{-N}
    d_{P_{n}}^{Q_{n}}(g_2)^r d_{P_{n}}^{Q_{n}}(g_1)^{-r}
    \end{equation}
and
    \begin{equation} \label{eq:bound_product_2}
    \abs{ {}_Q \beta(g)} \ll \aabs{g_1}_{Q_n}^{N_1+N}
    \aabs{ g_2 }_{Q_n}^{-N} d_{Q_{n}}^{P_{n}}(g_2)^r
    d_{Q_{n}}^{P_{n}}(g_1)^{-r} \ll \aabs{g_1}_{Q_n}^{N_1+N}
    \aabs{ g_2 }_{Q_n}^{-N} d_{Q_{n}}^{P_{n}}(g_2)^r
    d_{P_{n}}^{Q_{n}}(g_1)^{-r}.
    \end{equation}
Here in the second inequality of~\eqref{eq:bound_product_2} we have made use of the fact that $d_{Q_n}^{P_n} \ll d_{P_n}^{Q_n}$, cf.~\cite{BPCZ}*{(2.4.4.19)}.

For $C>0$ , let $\omega = \omega_C = \{g \in P_n(F) N_{Q_n}(\bA) \backslash
G_n(\bA) \mid d_{P_{n}}^{Q_{n}}(g)>C \}$. Let $\omega_P$ and $\omega_Q$
be the image of $\omega$ under the projections $\pi_{P_n}^{Q_n}$ and $\pi_{Q_n}^{P_n}$ respectively. By Lemma~\ref{lemma:classical_reduction_theory}(1)(2),
$d_{Q_{n}}^{P_{n}}$ and
$d_{P_{n}}^{Q_{n}}$ are bounded above on $[G_n]_{P_n} \bs \omega_P$ and
$[G_n]_{Q_n} \bs \omega_Q$ respectively. Thus by (\ref{eq:bound_product_1}),
for any $r>0$
    \[
    \int_{[G_n]_{P_n}\bs \omega_P}
    {}_P \beta(g_1,g_2) \rd g_2 \ll
     \|g_1\|_{P_n}^{N_1+N} d_{P_{n}}^{Q_{n}}(g_1)^{-r} \ll  \|g_1\|_{P_n}^{N_1+N} d_{P}^{Q}(g_1)^{-r}
    \]
for $N$ large enough. Here the second inequality follows from the fact that $d_P^Q \sim d_{P_{n+1}}^{Q_{n+1}} \ll d_{P_n}^{Q_n}$, cf.~Lemma~\ref{lem:weight_J_GL} and~\cite{BPCZ}*{Lemma~2.4.4.2}.
Similarly by (\ref{eq:bound_product_2}), for any $r>0$, we have
    \[
    \int_{[G_n]_{Q_n}\setminus \omega_Q} {}_Q \beta(g_1,g_2) \rd g_2
    \ll \|g_1\|_{P_n}^{N_1+N} d_{P}^{Q}(g_1)^{-r}.
    \]

It remains to estimate
    \[
    \int_{\omega_P} {}_P \beta(g_1,g_2) \rd g_2 -
    \int_{\omega_Q} {}_Q \beta(g_1,g_2) \rd g_2.
    \]
We choose $C$ large enough so that $\omega \to \omega_Q$ is injective.
This is possible by Lemma~\ref{lemma:classical_reduction_theory}(4).
Thus the above difference equals to
    \begin{equation} \label{eq:integrate_over_omega}
    \int_{\omega} {}_P \beta(g_1,g_2)  - {}_Q \beta(g_1,g_2) \rd g_2.
    \end{equation}

We assume $g_2 \in \omega$ for the rest of the proof.
Since $d_{P_{n}}^{Q_{n}} \sim d_{Q_{n}}^{P_{n}}$ on $\omega$ by
Lemma~\ref{lemma:classical_reduction_theory}(2), thus by~\eqref{eq:bound_product_1}
and~\eqref{eq:bound_product_2}, we have for any $r \ge 0$ and $N > 0$
    \[
    \left| {}_P \beta (g) - {}_Q \beta (g) \right|
    \ll \| g_1 \|_{P_n}^{N_1+N} \| g_2 \|_{P_n}^{-N}
    d_{P_{n}}^{Q_{n}}(g_2)^r d_{P_{n}}^{Q_{n}}(g_1)^{-r}
    \ll
    \| g_1 \|_{P_n}^{N_1+N} \| g_2 \|_{P_n}^{-N}
    d_{P_{n}}^{Q_{n}}(g_2)^r d_{P}^{Q}(g_1)^{-r}.
    \]
In particular, taking $r=0$ in the first inequality, we have
    \begin{equation*}
    \left| {}_P \beta (g) - {}_Q \beta (g) \right|
    \ll \| g_1 \|_{P_n}^{N_1+N} \| g_2 \|_{P_n}^{-N}.
    \end{equation*}
Thus we conclude that
    \[
    \left| {}_P \beta (g) - {}_Q \beta (g) \right|
    \ll \| g_1 \|_{P_n}^{N_1+N} \| g_2 \|_{P_n}^{-N}
    \max \{1,  d_{P}^{Q}(g_1) d_{P_{n}}^{Q_{n}}(g_2)^{-1} \}^{-r}.
    \]

Since the family $({}_P \beta) \in \cT_\cF^\Delta(G)$, there is an $N_0>0$ such that for any $r \geq 0$ we have
    \[
    \abs{\prescript{}{P}{\beta}(g) - \prescript{}{Q}{\beta}(g)}
    \ll \aabs{g}_{P_G}^{N_0} d_{P}^{Q, \Delta}(g)^{-r}.
    \]
Note that for any real numbers $a,b>0$, we have
$\max\{1,b/a\} \min\{a,b\}=b$.
Thus there is an $N_2$ such that for any $ r \ge 0$ and $N>0$, we have
    \begin{align*}
    \left| {}_P \beta (g) - {}_Q \beta (g) \right|
    &\ll \| g_1 \|_{P_n}^{N_2+N} \| g_2 \|_{P_n}^{-N}
    \left(\max\{1,d_{P}^{Q}(g_1)d_{P_{n}}^{Q_{n}}(g_2)^{-1}\}
    d_P^{Q,\Delta}(g)\right)^{-r} \\
    &= \| g_1 \|_{P_n}^{N_2+N}
    \| g_2 \|_{P_n}^{-N}  d_{P}^{Q}(g_1)^{-r}.
    \end{align*}
Fix a large $N$. It follows that there is an $N_3$ such that
    \[
    \eqref{eq:integrate_over_omega}\ll \aabs{g_1}_{P_n}^{N_3} d_{P}^Q(g_1)^{-r}
    \]
for any $r >0$. This proves the estimate~\eqref{eq:wanted_estimate}.
\end{proof}

\subsection{A modified kernel}  \label{subsec:modified_kernel_1}
For $f \in \cS(G(\bA))$, $\chi \in \fX(G)$, and $P \in \cF$,
we let $K_{f, P_G}$ and $K_{f, P_G, \chi}$ be the kernel functions defined in
Subsection~\ref{subsec:langlands_decomposition}. If $\Phi \in \cS(\bA_{E,
n})$, we have the theta series $\prescript{}{P}{\Theta}(\cdot, \Phi)$ on
$[J]_{P}$ defined in \eqref{eq:theta_gln_P}. For $(h, g) \in [H]_{P_H} \times [G]_{P_G}$ we define kernel functions for the test function $f \otimes \Phi$ by
\begin{equation}
\label{eq:kernel_gln+}
     K_{f \otimes \Phi, P}(h, g) = K_{f, P_G}(h, g) \cdot
    \prescript{}{P}{\Theta}(h, \Phi), \quad
    K_{f \otimes \Phi, P, \chi}(h, g) = K_{f, P_G, \chi}(h, g) \cdot
    \prescript{}{P}{\Theta}(h, \Phi).
\end{equation}

We will need to extend these definitions to all test functions $f_+ \in
\cS(G_+(\bA))$, not necessarily pure tensors. The case of $K_{f_+, P}$ is
straightforward. Put
    \[
    K_{f_+ ,P}(h, g) = \mu(\det h)^{-1} \abs{\det h}^{\frac{1}{2}}
    \sum_{m_+ \in M_{P, +}(F)} \int_{N_{P, +}(\bA)}
    f_+(m_+ n_+\cdot (h, g)) \rd n_+.
    \]
Then it is clear that if $f_+ = f \otimes \Phi$ then $K_{f_+ ,P}(h, g) =
K_{f, P_G}(h, g) \prescript{}{P}{\Theta}(h, \Phi)$. Moreover, for
fixed $h, g$ the linear form on $\cS(G_+(\bA))$ given by $f_+ \mapsto
K_{f_+,P}(h, g)$ is continuous. The case of $K_{f_+, P, \chi}$ needs some work.
First we note that for all $N_1>0$ there is an $N_2$ and a semi-norm
$\aabs{\cdot}_{N_1, N_2}$ on $\cS(G(\bA))$ such that for all $f \in
\cS(G(\bA))$ and $h \in [H]_{P_H}$, $g \in [G]_{P_G}$ we have
    \begin{equation}    \label{eq:estimate_kernel_GL_strong}
    \sum_{\chi \in \fX(G)} \Abs{K_{f, P, \chi}(h, g)} \leq
    \aabs{g_1}^{N_2}_{P_n} \aabs{g_2}_{P_n}^{-N_1}
    \aabs{h}_{P_H}^{-N_1} \aabs{f}_{N_1, N_2},
    \end{equation}
and a similar estimate with $g_1$ and $g_2$ swapped. This is a consequence of
Lemma ~\ref{lemma:estimate_kernel} applied to the weight $w =
\Delta_P^{2N_1} \aabs{\cdot}^{-N_1}$. As $\prescript{}{P}{\Theta} \in
\cT([H]_{P_H})$, we arrive at the following estimates: for all $N_1>0$ there
is an $N_2$ and a semi-norm $\aabs{\cdot}_{N_1, N_2}$ on $\cS(G_+(\bA))$ such
that for all $f_+ \in \cS(G(\bA)) \otimes \cS(\bA_{E, n})$ (algebraic tensor)
we have
    \begin{equation}    \label{eq:extension_kernel_estimate1}
    \sum_{\chi \in \fX(G)} \Abs{K_{f_+, P, \chi}(h, g)} \leq
    \aabs{g_1}_{P_n}^{N_2} \aabs{g_2}_{P_n}^{-N_1}
    \aabs{h}_{P_H}^{-N_1} \aabs{f_+}_{N_1, N_2},
    \end{equation}
and
    \begin{equation}    \label{eq:extension_kernel_estimate2}
    \sum_{\chi \in \fX(G)} \Abs{K_{f_+, P, \chi}(h, g)} \leq
    \aabs{g_1}_{P_n}^{-N_2} \aabs{g_2}_{P_n}^{-N_2}
    \aabs{h}_{P_H}^{N_1} \aabs{f_+}_{N_1, N_2}.
    \end{equation}

If $f_+ \in \cS(G_+(\bA))$ is approximated by a sequence of functions $f_{+, i}
\in \cS(G(\bA)) \otimes \cS(\bA_{E, n})$, by the
estimate~\eqref{eq:extension_kernel_estimate1}, the sequence $K_{f_{+, i}, P,
\chi}$ is convergent to a function on $[H]_{P_H} \times [G]_{P_G}$, and this
convergence along with all the derivative is locally uniformly for $(h, g) \in [H]_{P_H} \times [G]_{P_G}$.
We denote this function by $K_{f_+,P, \chi}$. It is clearly independent of
the choice of the sequence approximating $f_+$. Because the convergence of all the derivative is
locally uniform, $K_{f_+, P, \chi}$ is a smooth function. Moreover the
estimates~\eqref{eq:extension_kernel_estimate1}
and~\eqref{eq:extension_kernel_estimate2} continue to hold for $K_{f_+, P,
\chi}$.

We now give another interpretation of $K_{f_+, P, \chi}$. Fix a Schwartz
function $\Phi_0 \in \cS(\bA_{E, n})$ with $\aabs{\Phi_0}_{L^2} = 1$. For any
$\varphi \in \cT^0([H]_{P_H})$, we push forward $\varphi$ to
a measure on $[G]_{P_G}$ and we let $\varphi \cdot
\prescript{}{P}{\Theta}(\cdot, \Phi_0)$ be the measure on
$[\widetilde{G}]_{\widetilde{P}}$ defined in \eqref{eq:Radon_defi}. If $f_+ \in
\cS(G_+(\bA))$ we define a function $\widetilde{f_+} \in
\cS(\widetilde{G}(\bA))$ by
    \[
    \widetilde{f_+}(gs) =
    \langle f_+(g^{-1}, \cdot),\  \mathrm{R}_{\mu^{-1}}(s) \Phi_0(\cdot) \rangle_{L^2}, \quad
    g \in G(\bA), \ s \in S(\bA).
    \]
This notation means that we evaluate $f_+$ at $g^{-1}$ to obtain a Schwartz function
on $\bA_{E, n}$ and then take the $L^2$-inner product with
$\mathrm{R}_{\mu^{-1}}(s)\Phi_0$.
If $f_+ = f \otimes \Phi$, then $\widetilde{f_+}(gs) = f^\vee(g) \langle
\Phi, \mathrm{R}_{\mu^{-1}}(s) \Phi_0 \rangle$, where we recall that $f^\vee(g) = f(g^{-1})$.

We first note that the composition of the right translation
    \[
    \varphi \mapsto
    \mathrm{R}(\widetilde{f_+})(\varphi \cdot
    \prescript{}{P}{\Theta(\cdot, \Phi_0)})
    \]
and restriction to $[G]_{P_G}$ induces a continuous linear map
    \begin{equation}    \label{eq:interpretation_of_the_kernel}
    \mathrm{R}^G(\widetilde{f_+}):
    \cT^0([H]_{P_H}) \to \cT([G]_{P_G}), \quad
    \end{equation}
Let $\chi \in \fX(G)$ and define $\mathrm{R}_{\chi}^G(\widetilde{f_+})$ to be the
composition of $\mathrm{R}^G(\widetilde{f_+})$ followed by projection to the
$\chi$-component $\cT_{\chi}([G]_{P_G})$.

\begin{lemma}   \label{lemma:interpretation_of_the_kernel}
The maps $\mathrm{R}^G(\widetilde{f_+})$ and
$\mathrm{R}_{\chi}^G(\widetilde{f_+})$ are represented by the kernel functions
$K_{f_+, P}$ and $K_{f_+, P, \chi}$ respectively, i.e. for all $\varphi \in
\cT^0([H]_{P_H})$ and $g \in [G]_{P_G}$ we have
    \begin{equation}    \label{eq:GL_kernel_interpretation}
    \mathrm{R}^G(\widetilde{f_+})(\varphi)(g) =
    \int_{[H]_{P_H}} K_{f_+, P}(h, g)\varphi(h), \quad
    \mathrm{R}^G_{\chi}(\widetilde{f_+})(\varphi)(g) =
    \int_{[H]_{P_H}} K_{f_+, P, \chi}(h, g)\varphi(h).
    \end{equation}
In particular the map $\mathrm{R}_{\chi}^G(\widetilde{f_+})$ is independent
of the choice of $\Phi_0$.
\end{lemma}

Note that the integrations appear in the first variable of the kernel function because we used $g^{-1}$
in the definition of $\widetilde{f_+}$.

\begin{proof}
Since both sides of~\eqref{eq:GL_kernel_interpretation} are continuous linear
forms in $f_+$, we may assume that $f_+ = f \otimes \Phi$ where $f \in
\cS(G(\bA))$ and $\Phi \in \cS(\bA_{E, n})$.

Let us prove~\eqref{eq:GL_kernel_interpretation} for
$\mathrm{R}^G(\widetilde{f_+})$. The equality for
$\mathrm{R}_{\chi}^G(\widetilde{f_+})$ follows from it. For
$\varphi \in \cT^0([H]_{P_H})$ and $g \in [G]_{P_G}$, we have
    \[
    \mathrm{R}_{\chi}^G(\widetilde{f_+})(\varphi)(g) =
    \int_{H(\bA)}\int_{Z(\bA) \bs S(\bA)} \widetilde{f_+}(g^{-1}uh)
     \prescript{}{P}{\Theta}(uh, \Phi_0) \rd u \varphi(h).
    \]
Then by the definition of $\widetilde{f_+}$, this equals
    \[
    \int_{H(\bA)} \int_{Z(\bA) \bs S(\bA)} f(h^{-1} g)
    \langle \Phi, \mathrm{R}_{\mu^{-1}}(h^{-1}uh) \Phi_0 \rangle
    \prescript{}{P}{\Theta}(uh, \Phi_0) \rd u \varphi(h),
    \]
which simplifies to
    \[
    \int_{[H]_{P_H}} K_{f, P_G}(h, g)
    \prescript{}{P}{\Theta}(h, \Phi) \varphi(h).
    \]
This proves the lemma.
\end{proof}

If $P\subset Q$, $P, Q \in \cF$, we define a sign
    \[
    \epsilon_{P}^{Q} = (-1)^{\dim \fa^{Q_{n+1}}_{P_{n+1}}}.
    \]
We simply write $\epsilon_P = \epsilon_P^{J_n}$. Let $T \in
\fa_{n+1}$ be a truncation parameter and $(h, g') \in [H] \times [G']$. We
define a modified kernel
    \[
    K_{f_+}^T(h, g') = \sum_{P \in \cF}
    \epsilon_P \sum_{\substack{\gamma \in P_H(F) \bs H(F)\\
    \delta \in P'(F) \bs G'(F)}}
    \widehat{\tau}_{P_{n+1}}(H_{P_{n+1}}(\delta_1 g_1') - T_{P_{n+1}})
    K_{f_+, P}(\gamma h, \delta g').
    \]
For each $\chi \in \fX(G)$, we also define
    \[
    K_{f_+, \chi}^T(h, g') = \sum_{P \in \cF}
    \epsilon_P \sum_{\substack{\gamma \in P_H(F) \bs H(F)\\
    \delta \in P'(F) \bs G'(F)}}
    \widehat{\tau}_{P_{n+1}}(H_{P_{n+1}}(\delta_1 g_1') - T_{P_{n+1}})
    K_{f_+, P, \chi}(\gamma h, \delta g').
    \]
By~\cite{Arthur3}*{Lemma~5.1}, for $(h,g')$ fixed, in each sum the component $\delta_1$ can be taken in a finite
set depending on $g_1'$. The absolute convergence of the sums therefore follows from the
estimate~\eqref{eq:estimate_kernel_GL_strong}.

Let $F^{G_{n+1}'}(\cdot, T)$ be the characteristic function of Arthur
(characteristic function of the truncated Siegel set) for $G_{n+1}'$ defined
in Subsection~\ref{subsec:reduction_theory}. For a fixed $T$, it is
compactly supported modulo the center of $G_{n+1}'(\bA)$. In particular it is
compactly supported when restricted to $[G_{n}']$.

\begin{theorem} \label{thm:convergence_first_GL}
For every $N>0$, there is a continuous semi-norm $\aabs{\cdot}_{\cS, N}$ on
$\cS(G_+(\bA))$ such that
    \begin{equation}    \label{eq:asymptotics_first_GL}
    \sum_{\chi \in \fX(G)} \Abs{ K_{f_+, \chi}^T(h, g') -
    F^{G_{n+1}'}(g_1', T) K_{f_+, \chi}(h, g')} \le e^{-N\| T \|}
    \aabs{h}_{H}^{-N}
    \aabs{g'}_{G'}^{-N} \aabs{f_+}_{\cS, N}.
    \end{equation}
holds for all sufficiently positive $T$. In particular, for $T$ sufficiently
positive
    \[
    \sum_{\chi \in \fX(G)} \int_{[H]} \int_{[G']}
     \left| K_{f_+, \chi}^T(h, g') \right|
    \rd g' \rd h
    \]
is convergent and defines a continuous seminorm on $\cS(G_+(\bA))$.
\end{theorem}

\begin{proof}
We first note that the second statement follows directly from~\eqref{eq:extension_kernel_estimate1} and the
estimate~\eqref{eq:asymptotics_first_GL} by the fact that
$F^{G_{n+1}'}(\cdot, T)$ is compactly supported when restricted to $[G_n']$. So
we only need to prove the estimate~\eqref{eq:asymptotics_first_GL}.

We recall that $\cS^0([G_n'])$ is the space of (measurable) rapid decreasing function on
$[G_n']$, and it is a Fr\'{e}chet space with the seminorms
    \[
    \aabs{\varphi}_{\infty, N}  = \sup_{g \in [G_n']}
    \aabs{g}_{G_n'}^N \abs{\varphi(g)}.
    \]
We also recall that $\cT_{\cF}(G_n')$ is an LF space. Two
relative truncation operators
    \[
    \Lambda^{T, G_n'}, \Pi^{T, G_n'}:
    \cT_{\cF} (G_n') \to \cS^0([G_n']).
    \]
are defined in~\cite{BPCZ}*{Section~3.5}. For
$\varphi=(\prescript{}{P}{\varphi})_{P \in \cF} \in
\cT_{\cF}(G_n')$, and $g' \in [G_n']$, we have
    \[
    \begin{aligned}
    \Lambda^{T,G_n'}\varphi(g')
    &= \sum_{P \in \cF} \epsilon_P \sum_{\delta \in P_n'(F)
    \backslash G_n'(F)} \widehat{\tau}_{P_{n+1}}
    (H_{P_{n+1}}(\delta g')-T_{P_{n+1}}) {}_P \varphi(\delta g),\\
    \Pi^{T,G_n'}\varphi(g')
    &=F^{G_{n+1}'}(g',T) (\prescript{}{J_n}{\varphi})(g').
    \end{aligned}
    \]
By~\cite{BPCZ}*{Theorem~3.5.1.1}, for every $c>0$ and $N>0$, there exists a
continuous seminorm $\aabs{\cdot}_{c,N}$ on $\cT_{\cF}(G_n')$
such that for all sufficiently positive truncation parameter $T \in
\fa_{n+1}$ we have
    \begin{equation} \label{eq:mixed_truncation}
    \aabs{  \Lambda^{T,G_n'} \varphi - \Pi^{T,G_n'} \varphi }_{\infty,N}
    \ll e^{-c\|T\|}  \aabs{\varphi}_{c,N} ,
    \end{equation}
where $\aabs{ \cdot }_{\infty,N}$ is the norm on $\cS^0([G_n'])$.

For $f_+ \in \cS(G(\bA) \times \bA_{E, n})$ we consider the composition of
the following sequence of linear maps:
    \[
    \begin{tikzcd}
{\cT^0([H])\otimes \cT^0([G_n'])} \arrow[r] & {\cT_{\cF}^\Delta(G) \otimes \cT^0([G_n'])} \arrow[r, "{\langle \cdot,\cdot \rangle}"] & \cT_{\cF}(G_n) \arrow[r, "|_{G_n'}"] & \cT_{\cF}(G'_n) \arrow[r, "\Lambda^T", bend left] \arrow[r, "\Pi^T"', bend right] & {\cS^0([G_n'])}
\end{tikzcd}
    \]
where the first map is $\mathrm{R}^G(\widetilde{f_+}) \otimes \id$ (defined
using any fixed $\Phi_0 \in \cS(\bA_{E, n})$), the second is the pairing,
the third is the restriction, and the last is the truncation $\Lambda^{T,
G'_n}$ (resp. $\Pi^{T,G_n'}$), cf.
Lemmas~\ref{lemma:restriction_space_T},~\ref{lem:smoothed_constant_term},
and~\ref{lemma:space_T_pairing} for the description of these maps. Their composition will
be denoted by $L_{f_+}^T$ (resp. $P_{f_+}^T$). The fact that these maps are continuous is an easy consequence of the closed
graph theorem, see Remark~\ref{remark:closed_graph_theorem} after the proof.

By the definition of $K_{f_+, P}$ we see that the function $K_{f_+}^T$ (resp.
$F^{G_{n+1'}}(\cdot,T) K_{f_+}$) is the kernel function of the map
$L_{f_+}^T$ (resp. $P_{f_+}^T$). More precisely for any $\varphi \otimes
\varphi' \in \cT^0([H]) \otimes \cT^0([G_n'])$, we have
    \[
    \begin{aligned}
    L_{f_+}^T(\varphi \otimes \varphi')(g_1')
    &= \int_{[H]} \int_{[G_n']} K_{f_+}^T(h; g_1',g_2')
    \varphi(h) \varphi'(g_2'), \\
    P_{f_+}^T(\varphi \otimes \varphi')(g_1')
    &= \int_{[H]} \int_{[G_n']} F^{G_{n+1}'}(g_1',T)
    K_{f_+}(h;g_1',g_2') \varphi(h) \varphi'(g_2').
    \end{aligned}
    \]

For $\chi \in \fX(G)$, define
    \[
    L_{f_+,\chi}^T, P_{f_+, \chi}^T: \cT^0([H]) \otimes
    \cT^0([G_n']) \to \cS^0([G_n'])
    \]
the same way as $L_{f_+}^T$ and $P_{f_+}^T$ respectively, except that we
replace the first map by $\mathrm{R}_{\chi}^G(\widetilde{f_+}) \otimes \id$.
Then $L_{f_+,\chi}^T, P_{f_+, \chi}^T$ are separably continuous, and the functions $K_{f_+,\chi}^T$ and $F^{G_{n+1'}}(\cdot,T)
K_{f_+,\chi}$ are the kernel functions of the operators $L_{f_+,\chi}^T$ and $P_{f_+,\chi}^T$ respectively.

Fix $N >0$. By~\eqref{eq:mixed_truncation}, for all $\varphi \otimes \varphi'
\in \cT_N^0([H]) \otimes \cT_N^0([G_n'])$, we have
    \[
    \sum_{\chi \in \fX(G)} \left\| L_{f_+,\chi}^T (\varphi \otimes \varphi')
    -  P_{f_+,\chi}^T (\varphi \otimes \varphi') \right\|_{\infty,N}
    \ll e^{-N\aabs{T}}.
    \]
As in~\cite{BPCZ}*{Section 3.6}, the uniform boundedness principle (applied to
$\varphi$ and $\varphi'$) implies that
    \[
    \sum_{\chi \in \fX(G)} \left\| L_{f_+,\chi}^T (\varphi \otimes \varphi') -
    P_{f_+,\chi}^T (\varphi \otimes \varphi') \right\|_{\infty,N}
    \ll e^{-N\|T\|} \| \varphi \|_{1,-N} \| \varphi' \|_{1,-N}
    \]
holds for all $\varphi \otimes \varphi' \in \cT^0_N([H]) \otimes
\cT^0_N([G_n'])$ and $T$ sufficiently positive. We apply this estimate to $\varphi
= \delta_h$ and $\varphi' = \delta_{g_2'}$ and conclude that
    \[
    \sum_{\chi} \Abs{ K_{f_+, \chi}^T(h, g') -
    F^{G_{n+1}'}(g_1', T) K_{f_+, \chi}(h, g') } \ll e^{-N\aabs{T}}
    \aabs{h}_{H}^{-N} \aabs{g_2}_{G'}^{-N}.
    \]
By the uniform bounded principle again (applied to $f_+$), we see that there
exists a seminorm $\|\cdot \|_{\cS, N}$ on $\cS(G(\bA))$ such that
  \[
     \sum_{\chi} \left| K_{f_+, \chi}^T(h, g') -
    F^{G_{n+1}'}(g_1', T) K_{f_+, \chi}(h, g') \right| \leq e^{-N\| T \|}
    \|h\|_{H}^{-N}
    \|g'\|_{G'}^{-N} \|f_+\|_{\cS, N}.
    \]
This proves the estimate~\eqref{eq:asymptotics_first_GL}.
\end{proof}

\begin{remark}  \label{remark:closed_graph_theorem}
    We explain how the closed graph theorem is used in the proof to check continuity.
We temporarily denote by $X$ and $Y$ two topological spaces, and by $F(X)$ and $F(Y)$ certain
spaces of functions on $X$ and $Y$ respectively. Let $T : F(X) \to F(Y)$ be a linear map. Assume that
\begin{itemize}
    \item $T$ is continuous when $F(Y)$ is equipped with the pointwise convergence topology,
    \item the topology on $F(Y)$ is finer than the pointwise convergence topology,
    \item $F(X)$ and $F(Y)$ satisfy the closed graph theorem, i.e. $T$ is continuous if and only if its graph is closed (which is the case if $F(X)$ and $F(Y)$ are LF spaces).
\end{itemize}
Then we claim that $T$ is continuous. Indeed, if $(f_a)$ is a net in $F(X)$ such $\lim_{a \in A} f_a =f$ and $\lim_{a \in A} T(f_a)=g$, then we have for all $y \in Y$ the equality $T(f)(y)=\lim_{a \in A} T(f_a)(y)=g(y)$, which implies that $T(f)=g$ and therefore that $T$ is continuous.
\end{remark}

\subsection{The coarse spectral expansion}
Let $f_+ \in \cS(G_+(\bA))$ and $\chi \in \fX(G)$. We put
    \[
    I^T(f_+) = \int_{[H]} \int_{[G']}
    K_{f_+}^T(h, g')
    \eta_{n+1}(g')
    \rd g' \rd h
    \]
and
    \[
    I_{\chi}^T(f_+) = \int_{[H]} \int_{[G']}
    K_{f_+, \chi}^T(h, g')    \eta_{n+1}(g')
    \rd g' \rd h.
    \]
By Theorem~\ref{thm:convergence_first_GL}, these integrals are absolutely
convergent.

\begin{theorem} \label{thm:coarse_spectral_expansion_GL}
For $T$ sufficiently positive, the functions $I^T(f_+)$ and $I_{\chi}^T(f_+)$ are the
restrictions of exponential-polynomial functions whose purely polynomial parts
are constants. We denote them by $I(f_+)$ and $I_{\chi}(f_+)$ respectively.
Then $I$ and $I_\chi$ are continuous as distributions on $\cS(G_+(\bA))$
and for any $f_+ \in \cS(G_+(\bA))$ we have
    \[
    \sum_{\chi \in \fX(G)} I_{\chi}(f_+) = I(f_+),
    \]
where the sum is absolutely convergent.
\end{theorem}

Before we delve into the proof of this theorem, let us first explain a
variant of the construction of the modified kernel for parabolic subgroups.
For $f'_+ \in \cS(M_{Q_+}(\bA))$ and $P \subset Q  \in \cF$,
and $T \in \fa_{n+1}$, we
define a kernel function $K_{f_+',P \cap M_{Q}}(m_H,m_G)$ on
$[M_{Q_H}]_{P_H \cap M_{Q_H}} \times [M_{Q_{G}}]_{P_{G} \cap
M_{Q_{G}}}$ to be
    \[
    \left\{ \begin{aligned}
    &\mu^{-1}(\det m_H)
    \sum_{m_+ \in M_{P_+}(F)} \int_{(N_{P_+} \cap M_{Q_+})(\bA)}
    f_+'(m_+ n_+ \cdot (m_H,m_G)) \rd n_+, &&\text{$Q$ is of type I};\\
    &|\det m_{H,k}|_E^{\frac 12} \mu^{-1}(\det m_H)
    \sum_{m_+ \in M_{P_+}(F)} \int_{(N_{P_+} \cap M_{Q_+})(\bA)}
    f_+'(m_+ n_+ \cdot (m_H,m_G)) \rd n_+, &&\text{$Q$ is of type II}.
    \end{aligned}\right.
    \]
Here we assume that $Q_n$ stabilizes the flag~\eqref{eq:increasing_sequence}. The Levi subgroup of $Q_H \simeq Q_n$ is isomorphic to $\prod_{i = 1}^r \GL(L_i/L_{i-1})$, and we write $m_H\in M_{Q_H}(\bA)$ as $(m_{H, 1}, \hdots, m_{H, r})$ where
$m_{H, i} \in \GL(L_{i}/L_{i-1})(\bA_E)$ , and $\det m_H$ means taking the determinant
of $m_H$ as an element in $G_n$.

Similar to what we have done in Subsection~\ref{subsec:modified_kernel_1},
$K_{f'_+,P \cap M_Q}(m_H,m_G)$ is the kernel function of a continuous map
$\cT^0([M_{Q_H}]_{P_H \cap M_{Q_H}}) \to \cT([M_{Q_G}]_{P_G \cap M_{Q_G}})$.
For $\chi' \in \fX(M_{Q_G})$, composing this map with the projection to the
$\chi'$-component of $\cT([M_{Q_G}]_{P_G \cap M_{Q_G}})$ is given by another
kernel function, denoted by $K_{f_+',P \cap M_Q, \chi'}(m_H,m_G)$. Recall
that there is a natural finite-to-one map $\fX(M_{Q_G}) \to \fX(G)$ which we
temporarily denote by $\iota_Q$. Then for $\chi \in \fX(G)$, we define
    \[
    K_{f_+', P \cap M_Q, \chi}(m_H, m_G) = \sum_{\chi' \in \iota_Q^{-1}(\chi)}
    K_{f_+', P \cap M_Q, \chi'}(m_H, m_G).
    \]
For $T \in \fa_{n+1}$ and $(m_H, m') \in [M_{Q_H}] \times [M_{Q_{G'}}]$ we put
    \begin{align*}
    & K_{f'_+,\chi}^{M_Q,T}(m_H,m') \\
    =& \sum_{\substack{P \in \cF \\ P \subset Q}}
    \epsilon_P^Q \sum_{ \substack{\gamma \in (P_H(F) \cap M_{Q_H}(F))
    \backslash M_{Q_H}(F) \\ \delta \in (P_{G'}(F) \cap M_{Q_{G'}}(F))
    \backslash M_{Q_{G'}}(F)}}
    \widehat{\tau}_{P_{n+1}}^{Q_{n+1}}(H_{P_{n+1}}
    (\delta_1 m_1')-T_{P_{n+1}}) K_{f'_+,P \cap M_Q,\chi}
    (\gamma m_H,\delta m').
    \end{align*}

The group $A_{M_Q}^\infty$ which embeds in $M_{Q_H} \times M_{Q_{G'}}$ diagonally.
Note if $Q = J_n$, then $A_{M_Q}^\infty$ is trivial. By the same
method of proof of Theorem~\ref{thm:convergence_first_GL}, one can show that
for any $f_+' \in \cS(M_{Q,+}(\bA))$ and $T$ sufficiently positive, the
expression
    \begin{equation} \label{eq:convergent_Levi}
    \sum_{\chi \in \fX(G)}
    \int_{A_{M_Q}^\infty \backslash [M_{Q_H}] \times [M_{Q_{G'}}]}
    \Abs{K_{f_+',\chi}^{M_Q,T}(m_H,m')} \rd m_H \rd m'
    \end{equation}
is finite and defines a continuous seminorm on $\cS(M_{Q_+}(\bA))$. Define
a distribution on $\cS(M_{Q_+}(\bA))$ by
    \[
    I^{M_Q,T}_\chi(f'_+) =
    \int_{A_{M_Q}^\infty \backslash [M_{Q_H}] \times [M_{Q_{G'}}]}
    K_{f'_+,\chi}^{M_Q,T}(m_H,m') \eta_{n+1}(m') \rd m_H \rd m'.
    \]
Here and below in the proof, $\eta_{n+1}(m')$ means that we evaluate $\eta_{n+1}$
at $m'$ when $m'$ is viewed as an element in $G'(\bA)$.

\begin{proof}[Proof of Theorem~\ref{thm:coarse_spectral_expansion_GL}]

By~\cite{Arthur1}*{Section~2}, there exist functions $\Gamma_{P_{n+1}}'$ on
$\fa_{P_{n+1}}^{G_{n+1}} \times \fa_{P_{n+1}}^{G_{n+1}}$, for $P \in
\cF$, that are compactly supported in the first variable when
the second variable stays in a compact set and such that
    \begin{equation}    \label{eq:Gamma_Q'}
    \widehat{\tau}_{P_{n+1}}(H-X) = \sum_{ \substack{ Q \in \cF \\ Q \supset P }}
    \epsilon_{Q} \widehat{\tau}_{P_{n+1}}^{Q_{n+1}}(H)
    \Gamma'_{Q_{n+1}}(H, X), \quad H, X \in \fa_{P_{n+1}}^{G_{n+1}}.
    \end{equation}
Here following Arthur, by $\Gamma'_{Q_{n+1}}(H, X)$ we mean evaluating
$\Gamma'_{Q_{n+1}}$ at the projection of $H$ and $X$ to
$\fa_{Q_{n+1}}^{G_{n+1}}$.

We set $\underline{\rho}_Q=\rho_{Q_{n+1}}-\rho_{Q_n} \in
\fa_{Q_{n+1}}^{G_{n+1},*}$. If $Q_n$ stabilizes the
flag~\eqref{eq:increasing_sequence}, then for $m \in M_{Q_n}(\bA)$ of the form $m =
(m_1,\hdots,m_r)$ with $m_i \in \GL(L_{i}/L_{i-1})(\bA_E)$, one checks
directly that
    \begin{equation} \label{eq:underline_rho_Q_explicit}
    e^{\langle \underline{\rho}_Q, H_{Q_{n+1}}(m) \rangle}
    = \left\{ \begin{aligned}
    & \prod_{i \le k} \left|\det m_i\right|_E^{\frac{1}{2}} \cdot
    \prod_{i > k} \left|\det m_i\right|_E^{-\frac{1}{2}},
    &&\text{$Q$ is of type I};\\
    & \prod_{i \le k} \left|\det m_i\right|_E^{\frac{1}{2}} \cdot
    \prod_{i > k+1} \left|\det m_i\right|_E^{-\frac{1}{2}},
    &&\text{$Q$ is of type II}.
    \end{aligned} \right.
    \end{equation}
Define a function $p_Q$ on $\fa_{Q_{n+1}}^{G_{n+1}}$ by
    \begin{equation}    \label{eq:polynomial_exponential}
    p_Q(X) = \int_{\fa_{Q_{n+1}}^{G_{n+1}}}
    e^{\langle \underline{\rho}_Q, H \rangle} \Gamma'_{Q_{n+1}}(H,X) \rd H.
    \end{equation}
By~\cite{Zydor3}*{Lemma~3.5}, $p_Q$ is an exponential polynomial on
$\fa_{Q_{n+1}}^{G_{n+1}}$ whose exponents are contained in the set
$\{\underline{\rho}_R \mid R \supset Q \}$ and whose pure polynomial term is
the constant $\epsilon_Q \widehat{\theta}_Q(\underline{\rho}_Q)^{-1}$.
Here $\widehat{\theta}_Q$ is the homogeneous polynomial function on
$\fa_{P_{n+1}}^{G_{n+1}}$ defined in~\cite{Arthur1}*{Section~2}. We do not
need its precise definition.

For $f_+ \in \cS(G_+(\bA))$ we define for $(h, g') \in [H]_{Q_H} \times
[G']_{Q_{G'}}$ a modified kernel function
    \[
    K_{f_+,\chi}^{Q,T}(h,g') =
    \sum_{\substack{P \in \cF \\ P \subset Q}}
    \epsilon_P^Q \sum_{ \substack{\gamma \in P_H(F) \backslash Q_H(F)
    \\ \delta \in P_{G'}(F) \backslash Q_{G'}(F)}}
    \widehat{\tau}_{P_{n+1}}^{Q_{n+1}}
    (H_{P_{n+1}}(\delta_1 g_1')-T_{P_{n+1}})
    K_{f_+, P, \chi}(\gamma h,\delta g'),
    \]
By~\eqref{eq:Gamma_Q'}, for $T, T' \in \fa_{n+1}$, we have an inversion
formula
    \[
    K_{f_+,\chi}^T(h,g') =
    \sum_{Q \in \cF}
    \sum_{ \substack{\gamma \in Q_H(F) \backslash H(F)
    \\ \delta \in Q_{G'}(F) \backslash G'(F)}}
    \Gamma'_{Q_{n+1}}(H_{Q_{n+1}}(\delta_1 g_1')-T'_{Q_{n+1}},
    T_{Q_{n+1}}-T'_{Q_{n+1}})
    K_{f_+,\chi}^{Q,T'}(\gamma h,\delta g').
    \]
We integrate both sides over $[H] \times [G']$. For $T$ sufficiently
positive, the integral of $K_{f_+,\chi}^T(h,g')$ is absolutely convergent by
Theorem~\ref{thm:convergence_first_GL}. Moreover, when $T'$ is also sufficiently
positive, the computation below will show that the integral of each terms in
the sum over $Q$ is also absolutely convergent. It follows that $I_{f_+,
\chi}^T$ equals
    \begin{equation}    \label{eq:I_chi_inversion}
    \sum_{Q \in \cF} \int_{[H]_{Q_H}} \int_{[G']_{Q_{G'}}}
    \Gamma'_{Q_{n+1}}(H_{Q_{n+1}}(g_1')-T'_{Q_{n+1}},T_{Q_{n+1}}-T'_{Q_{n+1}})
    K_{f_+,\chi}^{Q,T'}(h,g') \eta_{n+1}(g') \rd h \rd g'.
    \end{equation}

We now relate $K_{f_+, \chi}^{Q, T}$ and  $K_{f_+', \chi}^{M_Q, T}$ via
parabolic descent. Recall that we have fixed good maximal compact subgroups
$K_H$ and $K_{G'}$ of $H(\bA)$ and $G'(\bA)$ respectively. For $f_+ \in
\cS(G_+(\bA))$, we define its parabolic descent as a function on
$M_{Q, +}(\bA)$ given by
    \begin{equation}    \label{eq:parabolic_descent_spectral_GL}
    \begin{aligned}
    & f_{+,Q}(m, l) = e^{\langle \rho_{Q_G},H_{Q_G}(m) \rangle}  \times \\
    & \int_{K_H \times K_{G'}} \int_{{N_{Q_G}}(\bA)} \int_{N_{Q_{L^\vee}(\bA)}}
    f_+(k_H^{-1}mnk',(l + u)k_H) \eta_{n+1}(k') \mu(\det k_H)^{-1} \rd u \rd n  \rd k_H \rd k'.
    \end{aligned}
    \end{equation}
Here $m \in M_G(\bA)$ and $l \in M_{L^\vee}(\bA)$, and we take the convention that $l = 0$ if $Q$ is of type I.
Then $f_{+, Q} \in
\cS(M_{Q, +}(\bA))$. For $(m_H, m') \in [M_{Q_H}] \times [M_{Q_{G}}]$ we have
    \[
    \begin{aligned}
    &\int_{K_H \times K_{G'}} K_{f_+,P, \chi}(m_H k_H,m' k') \eta_{n+1}(k')
    \rd k_H \rd k'\\
    =
    &e^{ \langle \rho_{Q_G},H_{Q_G}(m')+H_{Q_G}(m_H) \rangle}
    \cdot e^{ \langle \underline{\rho}_Q,H_{Q_{n+1}}(m_H) \rangle}
    K_{f_{+,Q},P \cap M_Q, \chi}(m_H,m_G).
    \end{aligned}
    \]
Indeed, using~\eqref{eq:underline_rho_Q_explicit}, we directly check this
identity without the $\chi$. The argument in~\cite{Zydor3}*{Lemma~1.3} shows
that this implies the identity with the $\chi$.
Therefore for $T \in \fa_{n+1}$ and $(m_H, m') \in [M_{Q_H}] \times [M_{Q_{G'}}]$,
we have
    \begin{equation}    \label{eq:distribution_parabolic_descent}
    \begin{aligned}
    &\int_{K_H \times K_{G'}} K_{f_+,\chi}^{Q,T} (m_H k_H,m' k') \eta_{n+1}(k')
    \rd k_H \rd k'\\
    = &e^{ \langle \rho_{Q_G},H_{Q_G}(m')+H_{Q_G}(m_H) \rangle}
    \cdot e^{ \langle \underline{\rho}_Q,H_{Q_{n+1}}(m_H) \rangle}
    K_{f_{+,Q},\chi}^{M_Q,T}(m_H,m').
    \end{aligned}
    \end{equation}

Fix a $Q \in \cF$. By the Iwasawa decomposition, the summand
corresponding to $Q$ in~\eqref{eq:I_chi_inversion} equals
    \begin{align*}
    \int_{[M_{Q_H}] \times [M_{Q_{G'}}]} \int_{K_H \times K_{G'}}
    &e^{\langle -2\rho_{Q_H},H_{Q_H}(m_H) \rangle}
    e^{\langle -2\rho_{Q_{G'}},H_{Q_{G'}}(m') \rangle}\\
    &\Gamma'_{Q_{n+1}}(H_{Q_{n+1}}(m_1')-T'_{Q_{n+1}},T_{Q_{n+1}}-T'_{Q_{n+1}})\\
    &K_{f_+,\chi}^{Q,T'}(m_Hk_H,m'k') \eta_{n+1}(m'k')
    \rd m' \rd m_H \rd k' \rd k_H.
    \end{align*}
By~\eqref{eq:distribution_parabolic_descent}, it equals
    \[
    \begin{aligned}
    \int_{[M_{Q_H}] \times [M_{Q_{G'}}]}
    &e^{ \langle \underline{\rho}_Q, H_{Q_{n+1}}(m_H) \rangle}
    \Gamma'_{Q_{n+1}}(H_{Q_{n+1}}(m_1')-T'_{Q_{n+1}},T_{Q_{n+1}}-T'_{Q_{n+1}})\\
    &K_{f_{+,Q},\chi}^{M_Q,T'}(m_H,m') \eta_{n+1}(m') \rd m_H \rd m'.
    \end{aligned}
    \]
We break the integral into an integral over the split center $A_{M_Q}^\infty$
and another over $A_{M_Q}^\infty \backslash [M_{Q_H}] \times [M_{Q_{G'}}]$ to write
the above integral as
    \begin{equation}    \label{eq:polynomial_exponential_final_step}
    \begin{aligned}
    &\int_{A_{M_Q}^\infty \backslash [M_{Q_H}] \times [M_{Q_{G'}}]}
    e^{ \langle \underline{\rho}_Q, H_{Q_{n+1}}(m_H)-H_{Q_{n+1}}(m_1') \rangle}
    K_{f_{+,Q},\chi}^{M_Q,T'}(m_H,m') \eta_{n+1}(m') \\
    &\left( \int_{A_{M_Q}^\infty}
    e^{\langle \underline{\rho}_Q, H_{Q_{n+1}}(a m_1') \rangle}
    \Gamma'_{Q_{n+1}}(H_{Q_{n+1}}(am_1')-T'_{Q_{n+1}},T_{Q_{n+1}}-T'_{Q_{n+1}})
    \rd a \right) \rd m_H \rd m'
    \end{aligned}
    \end{equation}

Note that the composition of the embeddings $\fa_{Q} \subset \fa_{Q_n} \subset \fa_{Q_{n+1}}$ with the projection $\fa_{Q_{n+1}}\to
\fa^{G_{n+1}}_{Q_{n+1}}$ yields an isomorphism $\fa_{Q} \simeq
\fa^{G_{n+1}}_{Q_{n+1}}$ whose Jacobian we denote by $c_Q$. The inner integral
in~\eqref{eq:polynomial_exponential_final_step} thus equals
    \[
    c_Q \int_{\fa_{Q_{n+1}}^{G_{n+1}}}
    e^{\langle \underline{\rho}_Q, X + H_{Q_{n+1}}(m_1') \rangle}
    \Gamma'_{Q_{n+1}}(X + H_{Q_{n+1}}(m_1')-T'_{Q_{n+1}}, T_{Q_{n+1}}-T'_{Q_{n+1}})
    \rd X.
    \]
Changing the variable $X \mapsto X - H_{Q_{n+1}}(m_1') + T'_{Q_{n+1}}$, we obtain that
this equals
    \[
    c_Q \int_{\fa_{Q_{n+1}}^{G_{n+1}}}
    e^{\langle \underline{\rho}_Q, X + T'_{Q_{n+1}} \rangle}
    \Gamma'_{Q_{n+1}}(X, T_{Q_{n+1}}-T'_{Q_{n+1}})
    \rd X.
    \]
By definition, cf.~\eqref{eq:polynomial_exponential}, this equals
    \[
    c_Q e^{\langle \underline{\rho}_Q, T' \rangle} p_Q(T-T').
    \]
In particular it is independent of $m'$. We then have
    \[
    \begin{aligned}
    \eqref{eq:polynomial_exponential_final_step} = \,
    &c_Q e^{\langle \underline{\rho}_Q, T' \rangle} p_Q(T-T') \\
    &\int_{A_{M_Q}^\infty \backslash [M_{Q_H}] \times [M_{Q_{G'}}]}
    e^{ \langle \underline{\rho}_Q, H_{Q_{n+1}}(m_H)-H_{Q_{n+1}}(m_1') \rangle}
    K_{f_{+,Q},\chi}^{M_Q,T'}(m_H,m') \eta_{n+1}(m') \rd m_H \rd m'.
    \end{aligned}
    \]

This integral is absolutely convergent by~\eqref{eq:convergent_Levi}. If we
define
    \[
    \widetilde{f_{+,Q}}(m,l)= e^{ \langle -\underline{\rho}_Q,
    H_{Q_{n+1}}(m)\rangle} f_{+,Q}(m,l), \quad (m, l) \in M_{Q, +}(\bA),
    \]
then
    \begin{equation}    \label{eq:distribution_parabolic_descent_normalized}
    K_{\widetilde{f_{+,Q}},\chi}^{M_Q,T'}(m_H,m')
    = e^{\langle \underline{\rho}_Q, H_{Q_{n+1}}(m_H)-H_{Q_{n+1}}(m_1') \rangle}
    K_{f_{+,Q},\chi}^{M_Q,T'}(m_H,m').
    \end{equation}
Thus
    \[
    \eqref{eq:polynomial_exponential_final_step} =
    c_Q e^{\langle \underline{\rho}_Q, T' \rangle} p_Q(T-T')
    I_{\chi}^{M_Q,T'}(\widetilde{f_{+,Q}}).
    \]

In conclusion, we have
    \[
    I_\chi^T(f_+) =
    \sum_{Q \in \cF}
    c_Q e^{\langle \underline{\rho}_Q,T' \rangle} p_Q(T-T')
    I_{\chi}^{M_Q,T'}(\widetilde{f_{+,Q}}).
    \]
That is, $I_\chi^T(f_+)$ is an exponential-polynomial in $T$ and the pure
polynomial term is the constant
    \begin{equation} \label{I_chi_explicit}
    I_\chi(f_+)= \sum_{Q \in \cF}
    c_Q \epsilon_Q \widehat{\theta}_Q(\underline{\rho}_Q)^{-1}
    e^{\langle \underline{\rho}_Q,T' \rangle}
    I_{\chi}^{M_Q,T'}(\widetilde{f_{+,Q}}) .
    \end{equation}
This proves the theorem.
\end{proof}

\begin{remark} \label{rmk:distribution_on_Levi}
Using the same method, one can show that for any $Q \in \cF$, and
any $f_+' \in \cS(M_{Q_+}(\bA))$ and $T$ sufficiently positive, as a function
of $T$, $I^{M_Q,T}_\chi(f_+')$ is an exponential polynomial whose exponents
are contained in $\{ \underline{\rho}_P - \underline{\rho}_Q \mid P \subset Q
\}$ and only depend on the image $T_{Q_{n+1}}$ of projection of $T$ to
$\fa_{Q_{n+1}}$, but the pure polynomial term is not necessarily a constant.
\end{remark}

\begin{prop}    \label{prop:invariant_I_chi}
The distribution $I_\chi$ is left $H(\bA)$-invariant and right
$(G'(\bA),\eta_{G'})$-equivariant. More precisely for all $f_+ \in
\cS(G_+(\bA))$, $h \in H(\bA)$ and $g' \in G'(\bA)$ we have
    \[
    I_\chi((h, g') \cdot f_+)= \eta_{n+1}(g') I_\chi(f_+)
    \]
where $((h, g') \cdot f_+)(g,u)=\abs{\det h}_E^{1/2} \mu(\det h)^{-1}
f_+(h^{-1} g g',uh)$.
\end{prop}

\begin{proof}
Fix an $(h_0,g_0') \in H(\bA) \times G'(\bA)$. Let $T \in \fa_{n+1}$ be
sufficiently positive. By~\eqref{eq:Gamma_Q'} again we have
    \begin{align*}
     \widehat{\tau}_{P_{n+1}}(H_{P_{n+1}}(\delta_1 g_1'g_{0,1}^{\prime})
     -T_{P_{n+1}})
     = & \sum_{P \subset Q}  \epsilon_Q \widehat{\tau}_{P_{n+1}}^{Q_{n+1}}
     ( H_{P_{n+1}}(\delta_1 g_1') - T_{P_{n+1}} )\\
     &\Gamma'_{Q_{n+1}} ( H_{Q_{n+1}}(\delta_1 g_1')-T_{Q_{n+1}},H_{Q_{n+1}}
     (\delta_1 g_1') - H_{Q_{n+1}}(\delta_1 g_1' g_{0,1}^{\prime})).
    \end{align*}
Therefore we see that $K_{f_+,\chi}^T(h h_0, g' g_0')$ equals
    \[
    \begin{aligned}
    \sum_{P \in \cF}
    \sum_{P \subset Q} \sum_{ \substack{\gamma \in P_H(F)
    \backslash H(F) \\ \delta \in P_{G'}(F)\backslash G'(F)}}
    & \epsilon_Q \widehat{\tau}_{P_{n+1}}^{Q_{n+1}}
     ( H_{P_{n+1}}(\delta_1 g_1') - T_{P_{n+1}} )
     K_{f_+, P, \chi}(\gamma h h_0,\delta g' g_0')\\
    &\Gamma'_{Q_{n+1}}(H_{Q_{n+1}}(\delta_1 g_1')-T_{Q_{n+1}},H_{Q_{n+1}}
    (\delta_1 g_1')-H_{Q_{n+1}}(\delta_1 g_1' g_{0,1}^{\prime})).
    \end{aligned}
    \]
Note that
    \[
    K_{(h_0, g_0') \cdot f_+, P, \chi}(h,g')=
    K_{f_+, P, \chi}(h h_0, g' g_0').
    \]
Thus by summing over $P \in \cF$ first in  the above expression, we see that
it simplifies to
    \begin{equation}    \label{eq:invariance_GL_1}
    \begin{aligned}
    \sum_{Q \in \cF} \sum_{ \substack{\gamma \in Q_H(F)
    \backslash H(F) \\ \delta \in Q_{G'}(F)\backslash G'(F)}}
    &\Gamma'_{Q_{n+1}}(H_{Q_{n+1}}(\delta_1 g_1')-T_{Q_{n+1}},H_{Q_{n+1}}
    (\delta_1 g_1')-H_{Q_{n+1}}(\delta_1 g_1' g_{0,1}^{\prime}))\\
    &K_{(h_0, g_0') \cdot f_+, \chi}^{Q, T}(\gamma h,\delta g').
    \end{aligned}
    \end{equation}

We integrate $K_{f_+,\chi}^T(h h_0, g' g_0') \eta_{n+1}(g')$ over $[H] \times [G']$.
On the one hand by definition it equals
    \[
    \int_{[H] \times [G']}
    K_{f_+,\chi}^T(hh_0,g'g_0^{\prime}) \eta_{n+1}(g') \rd g' \rd h
    =  \eta_{n+1}(g'_0)  I^T_\chi(f_+).
    \]
On the other hand, by~\eqref{eq:invariance_GL_1},
it equals the sum over all $Q \in \cF$ of the terms
    \begin{equation}    \label{eq:invariant_Q_GL}
    \begin{aligned}
    \int_{[H]_{Q_H}} \int_{[G']_{Q_{G'}}}
    &\Gamma_{Q_{n+1}}'(H_{Q_{n+1}}(g_1')-T_{Q_{n+1}},H_{Q_{n+1}}(g_1')
    -H_{Q_{n+1}}(g_1' g_{0,1}'))\\
    & K_{(h_0, g_0') \cdot f_+, \chi}^{Q, T}(h, g')
    \eta_{n+1}(g') \rd h \rd g'.
    \end{aligned}
    \end{equation}
We calculate this as in the proof of
Theorem~\ref{thm:coarse_spectral_expansion_GL}. By the Iwasawa decomposition
it equals
    \begin{equation} \label{eq:invariant_Iwasawa}
            \begin{aligned}
    \int_{[M_{Q_H}] \times [M_{Q_{G'}}]} \int_{K_H \times K_{G'}}
    &e^{\langle -2\rho_{Q_H}, H_{Q_H}(m_H) \rangle}
    e^{\langle -2\rho_{Q_{G'}}, H_{Q_{G'}}(m') \rangle} \\
    &\Gamma'_{Q_{n+1}}(H_{Q_{n+1}}(m_1')-T_{Q_{n+1}},
    -H_{Q_{n+1}}(k_1' g_{0,1}'))\\
    & K_{(h_0, g_0') \cdot f_+,\chi}^{Q,T}(m_H k_H, m'k')
    \eta_{n+1}(m'k') \rd k' \rd k_H \rd m' \rd m_H .
    \end{aligned}
    \end{equation}

For $(m, l) \in M_{Q, +}(\bA)$, we put
    \begin{align*}
    f_{+, Q, h_0, g_0'}(m, l) =
    e^{\langle \rho_{Q_G},H_{Q_G}(m) \rangle}
    &\int_{K_H \times K'} \int_{{N_{Q_G}}(\bA)} \int_{N_{Q_{L^\vee}(\bA)}}
    ((h_0, g_0') \cdot f_+)(k_H^{-1}m n k',(l+ u)k_H)\\
    & p_Q(-H_{Q_{n+1}}(k_1' g_{0,1}')) \eta_{n+1}(k') \mu^{-1}(k_H)
    \rd n \rd  u \rd k_H \rd k'.
    \end{align*}

One check directly that for an $(m_H,m') \in [M_{Q_H}] \times [M_{Q_{G'}}]$, the expression
    \begin{align*}
       & \int_{K_H \times K_{G'}} \int_{A_{M_Q}^\infty}
    e^{\langle -2\rho_{Q_H}, H_{Q_H}(am_H) \rangle}
    e^{\langle -2\rho_{Q_{G'}}, H_{Q_{G'}}(am') \rangle}
    \Gamma'_{Q_{n+1}}(H_{Q_{n+1}}(am_1')-T_{Q_{n+1}},
    -H_{Q_{n+1}}(k_1' g_{0,1}'))\\
    & K_{(h_0, g_0') \cdot f_+,P}(am_H k_H, am'k')
    \eta_{n+1}(am'k')  \rd a \rd k_H \rd k'
    \end{align*}
    equals to
    \[
    e^{\langle \underline{\rho}_Q, H_{Q_{n+1}}(m_H) - H_{Q_{n+1}}(m_1') \rangle} c_Q e^{ \langle \underline{\rho}_Q, T \rangle} K_{f_{+,Q,h_0,g_0'},P \cap M_Q}^{M_Q}(m_H,m').
    \]

Then for any $\chi \in \fX(G)$,
 \begin{align*}
       & \int_{K_H \times K_{G'}} \int_{A_{M_Q}^\infty}
    e^{\langle -2\rho_{Q_H}, H_{Q_H}(am_H) \rangle}
    e^{\langle -2\rho_{Q_{G'}}, H_{Q_{G'}}(am') \rangle}
    \Gamma'_{Q_{n+1}}(H_{Q_{n+1}}(am_1')-T_{Q_{n+1}},
    -H_{Q_{n+1}}(k_1' g_{0,1}'))\\
    & K_{(h_0, g_0') \cdot f_+,P,\chi}(am_H k_H, am'k')
    \eta_{n+1}(am'k')  \rd a \rd k_H \rd k'
    \end{align*}
    equals to
    \[
    e^{\langle \underline{\rho}_Q, H_{Q_{n+1}}(m_H) - H_{Q_{n+1}}(m_1') \rangle} c_Q e^{ \langle \underline{\rho}_Q, T \rangle} K_{f_{+,Q,h_0,g_0'},P \cap M_Q,\chi}^{M_Q}(m_H,m').
    \]

Therefore ~\eqref{eq:invariant_Iwasawa} equals
    \[
    c_Q e^{\langle \underline{\rho}_Q,T \rangle}
    I_{\chi}^{M_Q,T}(\widetilde{f_{+, Q, h_0, g_0'}}),
    \]
where
    \[
       \widetilde{f_{+,Q,h_0,g_0'}}(m,l)= e^{ \langle -\underline{\rho}_Q,
    H_{Q_{n+1}}(m)\rangle} f_{+,Q,h_0,g_0'}(m,l), \quad (m, l) \in M_{Q, +}(\bA)
    \]

In conclusion, we have
    \[
    \eta_{n+1}(g_0') I_\chi^{T}(f_+) = \sum_{Q \in \cF}
    c_Q  e^{\langle \underline{\rho}_Q, T \rangle}
    I_{\chi}^{M_Q,T}(\widetilde{f_{+, Q, h_0, g_0'}}).
    \]
Each $I_{\chi}^{M_Q,T}(f_{+, Q, h_0, g_0'})$ is a exponential polynomial,
whose exponents are in the set $\{\underline{\rho}_P - \underline{\rho}_Q
\mid P \subset Q \}$. Since $\underline{\rho}_P$ is not trivial unless
$P = J_n$, the only term on the right hand side
that has a (possibly) nonzero purely polynomial part correspond to $Q =
J_n$. In this case $f_{+, Q, h_0, g_0'} = (h_0, g_0') \cdot f_+$ and
the pure polynomial part of the right hand side equals
$I_\chi((h_0, g_0') \cdot f_+)$.
\end{proof}

\subsection{A second modified kernel}
For later use we will need another modified kernel. For $f_+ \in
\cS(G_+(\bA))$, $T \in \fa_{n+1}$, and $(h, g') \in [H] \times [G']$, we
define
    \[
    \kappa_{f_+, \chi}^T(h, g') = \sum_{P \in \cF}
    \epsilon_P \sum_{\substack{\gamma \in P_H(F) \bs H(F)\\
    \delta \in P'(F) \bs G'(F)}}
    \widehat{\tau}_{P_{n+1}}(H_{P_{n+1}}(\gamma h) - T_{P_{n+1}})
    K_{f_+, P, \chi}(\gamma h, \delta g').
    \]

\begin{prop}    \label{prop:asymptotic_second_GL}
For every $N>0$, there is a continuous seminorm $\aabs{\cdot}_{\cS, N}$ on
$\cS(G(\bA) \times \bA_{E,n})$ such that for all $h \in [H]$ and $g' \in
[G']$ we have
    \begin{equation}    \label{eq:asymptotic_second_modified_GL}
    \sum_{\chi} \Abs{ \kappa_{f_+, \chi}^T(h, g') -
    F^{G_{n+1}}(h, T) K_{f_+, \chi}(h, g')} \le e^{-N\aabs{T}}
    \aabs{h}_{H}^{-N}
    \aabs{g'}_{G'}^{-N} \aabs{f_+}_{\cS, N},
    \end{equation}
In particular
    \[
    \sum_{\chi \in \fX(G)} \int_{[H]} \int_{[G']}
    \left|\kappa_{f_+, \chi}^T(h, g')\right|
    \rd g' \rd h
    \]
for $T$ sufficiently large.
\end{prop}

\begin{proof}
As in the case of Theorem~\ref{thm:convergence_first_GL}, we only need to
prove the estimate~\eqref{eq:asymptotic_second_modified_GL}, and second
assertion on the absolute convergence follows from it. The proof is very
similar to (and in fact simpler than) that of
Theorem~\ref{thm:convergence_first_GL}, so we will only sketch the
differences.

We first note that we only need to prove that there are continuous semi-norms
$\aabs{\cdot}_{\cS, 1}$ and $\aabs{\cdot}_{\cS, 2}$ on $\cS(G(\bA))$ and
$\cS(\bA_{E, n})$ respectively, such that
    \[
    \sum_{\chi} \Abs{ \kappa_{f \otimes \Phi, \chi}^T(h, g') -
    F^{G_{n+1}}(h, T) K_{f \otimes \Phi, \chi}(h, g')} \le e^{-N\aabs{T}}
    \aabs{h}_{H}^{-N}
    \aabs{g'}_{G'}^{-N} \aabs{f}_{\cS, 1} \aabs{\Phi}_{\cS, 2}.
    \]
Once we have this, the estimate~\eqref{eq:asymptotic_second_modified_GL} holds
for $f_+ \in \cS(G(\bA)) \otimes \cS(\bA_{E, n})$.
For fixed $h$ and $g'$, the left hand side
of~\eqref{eq:asymptotic_second_modified_GL} is continuous with respect to
$f_+$ by definition. Thus the
estimate~\eqref{eq:asymptotic_second_modified_GL} is obtained by taking
limits.

We assume $f_+ = f \otimes \Phi$ from now on. Define
$\cT'_{\cF}(G)$ to be the space of tuples
$(\prescript{}{P}{\varphi})_{P \in \cF}$ such that
    \[
    \prescript{}{P}{\varphi} \in \cT([G]_{P_G}), \quad
    \prescript{}{P}{\varphi} - (\prescript{}{Q}{\varphi})_{P_G} \in
    \cS_{d_{P_G}^{Q_G}}([G]_{P_G}).
    \]
Let us consider the series of linear maps
\[
 \begin{tikzcd}
{\cT^0([G])} \arrow[r, "\mathrm{R}(f)"] & \cT'_{\cF}(G) \arrow[r, "\times \theta"] & \cT_\cF(\widetilde{G}) \arrow[r, "|_H"] & \cT_\cF(H) \arrow[r, "\Lambda^{T,H}", bend left] \arrow[r, "\Pi^{T,H}"', bend right] & {\cS^0([H])}
\end{tikzcd}
\]
The first map is
    \[
    \varphi \mapsto \mathrm{R}(f)\varphi_{P_{G'}}.
    \]
The second map is given by
    \[
    \prescript{}{P}{\varphi} \mapsto
    \left( (j, g_2) \mapsto \prescript{}{P}{\varphi}(j, g_2)
    \prescript{}{P}{\Theta}(j,\Phi) \right),
    \]
where $(j, g_2) \in J_n \times G_n$ and $g_1$ is the image of
$j$ in $G_n$. The third one is the restriction to $[H]_{P_H}$.
The last map is one of the two truncation operators
    \[
    \Lambda^{T, H}, \Pi^{T, H}:
    \cT_{\cF} (H) \to \cS^0([H]).
    \]
defined in~\cite{BPCZ}*{Section~3.5}. The fact that the family $P \mapsto
\mathrm{R}(f)\varphi_{P_{G'}}$ belongs to $\cT'_{\cF}(G)$
follows from~\cite{BPCZ}*{Proposition~3.4.2.1.2}. The fact that rest of the
maps make sense follows directly from the definition. Continuity of these maps
follow from the closed graph theorem,
cf.~Remark~\ref{remark:closed_graph_theorem}. The compositions of
these maps are given by integral kernels $\kappa_{f \otimes \Phi}^T(h, g')$
and $F^{G_{n+1}}(h, T) \kappa_{f \otimes \Phi}(h, g')$ respectively.

Let $\chi \in \fX(G)$. We modify the maps slightly, by using the map
$\mathrm{R}_{\chi}(f)$ in the first one, where $\mathrm{R}_{\chi}(f)$ is the
composition of $\mathrm{R}(f)$ followed by projection to the
$\chi$-component. The kernel function associated to the resulting maps are
$\kappa_{f \otimes \Phi, \chi}^T(h, g')$ and $F^{G_{n+1}}(h, T) \kappa_{f
\otimes \Phi, \chi}(h, g')$ respectively.

The rest of the proof is exactly the same as that of
Theorem~\ref{thm:convergence_first_GL}.
\end{proof}

Let $T \in \fa_{n+1}$ be sufficiently large and put
    \[
    i_{\chi}^T( f_+) = \int_{[H]} \int_{[G']}
    \kappa_{f_+, \chi}^T(h, g')
    \eta_{n+1}(g') \rd g' \rd h
    \]

\begin{prop}    \label{prop:modified_kernel_second}
As a function of $T$, the function $i_{\chi}^T(f_+)$ is the restriction of
an exponential polynomial function whose purely polynomial part is a constant
and equals $I_\chi(f_+)$.
\end{prop}

\begin{proof}
By~\eqref{eq:Gamma_Q'}, we have
    \begin{align*}
     \widehat{\tau}_{P_{n+1}}(H_{P_{n+1}}(\gamma h)-T_{P_{n+1}})
     = &\sum_{Q \supset P} \epsilon_Q \widehat{\tau}_{P_{n+1}}^{Q_{n+1}}
     ( H_{P_{n+1}}(\delta_1 g_1')-T_{P_{n+1}} )\\
     &\Gamma'_{Q_{n+1}} ( H_{Q_{n+1}}(\delta_1 g_1')-T_{Q_{n+1}},
     H_{Q_{n+1}}(\delta_1 g_1') - H_{Q_{n+1}}(\gamma h) ).
    \end{align*}
Plugging into the definition of $\kappa_{f_+, \chi}^T$, we obtain that
    \[
    \begin{aligned}
    &\kappa_{f_+,\chi}^T(h,g')\\ = &\sum_{Q \in \cF}
    \sum_{ \substack{\gamma \in Q_H(F) \backslash G(F) \\
    \delta \in Q_{G'}(F) \backslash G'(F)}}
    \Gamma'_{Q_{n+1}}(H_{Q_{n+1}}(\delta_1 g_1')-T_{Q_{n+1}},
    H_{Q_{n+1}}(\delta_1 g_1')-H_{Q_{n+1}}(\gamma h))
    K_{f_+,\chi}^{Q,T}(\gamma h,\delta g').
    \end{aligned}
    \]
Define $f_{+, Q}$ by~\eqref{eq:distribution_parabolic_descent} and
$\widetilde{f_{+, Q}}$
by~\eqref{eq:distribution_parabolic_descent_normalized} as in the proof of
Theorem~\ref{thm:coarse_spectral_expansion_GL}, then the same computation as
in the proof of Theorem~\ref{thm:coarse_spectral_expansion_GL} gives (we
follow the same notation there)
    \[
    i_\chi^T(f_+) = \sum_{Q \in \cF}
    c_Q e^{\langle \underline{\rho}_Q, T \rangle}
    I_\chi^{M_Q,T}(\widetilde{f_{+, Q}}),
    \]
As in the proof of Theorem~\ref{thm:coarse_spectral_expansion_GL} and
Proposition~\ref{prop:invariant_I_chi}, cf. also
Remark~\ref{rmk:distribution_on_Levi},
we see that $i_\chi^T(f_+)$ is an
exponential polynomial function of $T$. The only term on the right hand side
with a (possibly) nonzero purely polynomial part is the one that correspond
to $Q = J_n$. This term equals $I_{\chi}^T(f_+)$ and
the purely polynomial part is the constant $I_{\chi}(f_+)$.
\end{proof}

\section{The coarse geometric expansion: general linear groups}
\label{sec:geo_GL}

\subsection{Geometric modified kernels}
\label{subsec:geo_GL_modified_kernel}

We keep the notation from the previous section. We will need in addition the
following list of notation.

\begin{itemize}
\item Recall that $L = E^n$ and $L^\vee = E_n$. Let $L^-$ and $L^{\vee, -}$
    be the purely imaginary part, i.e. $L^- = E^{-, n}$ and $L^{\vee, -} =
    E_n^-$. If $A$ is a subset of $L^\vee \times L$, we define $A^-$ to be
    its intersection with $L^{\vee, -} \times L^-$.

\item We put $G^+ = G \times L^{\vee, -} \times L^-$, the group structure being
    the product of $G$ and $L^{\vee, -} \times L^-$. Note that this is not a
    subgroup of $\widetilde{G}$, but merely a subvariety.
    The group $H \times G'$ acts from the right on $G^+$ by
    \begin{equation}    \label{eq:action_G^+}
    (g,w,v) \cdot (h,g') = (h^{-1}gg',w g_1',g_1'^{-1} v).
    \end{equation}

\item Let $P = MN \in \cF$. Put $M^+ = M_G \times
    M_{L^{\vee}}^{-} \times M_L^{-}$, $N^+ = N_G \times N_{L^{\vee}}^-
    \times N^{-}_L$, and $P^+ = M^+ N^+$. We often write an element in $P^+$ as
    $m^+n^+$, but one should note that the product is the one in $G^+$, not
    the one in $\widetilde{G}$.

\item Let $q: G^+ \to \cA= G^+//(H \times G')$ be     the GIT quotient. We define a morphism $G^+: \to \Res_{E/F} \mathbf{A}_{2n, E}$ by
        \begin{equation} \label{eq:gl_GIT_map}
        ((g_1, g_2), w, v) \mapsto (a_1, \hdots, a_n; b_1, \hdots, b_n),
        \end{equation}
        where $\mathbf{A}_{2n, E}$ denotes the $2n$-dimensional affine space over $E$, and we put
        \[
        s = (g_1^{-1} g_2)(g_1^{-1}g_2)^{\mathsf{c}, -1}, \quad
        a_i = \Trace \wedge^i s,
        \quad b_i = w  s^i v, \quad i = 1, 2, \hdots, n.
        \]
    This map descends to a locally closed embedding $\cA \hookrightarrow \mathrm{Res}_{E/F} \mathbf{A}_{2n, E}$. We always consider $\cA$ as a locally closed subscheme of $\Res_{E/F}\mathbf{A}_{2n, E}$.

\item If $\alpha \in \cA(F)$, we define $G^+_{\alpha}$ to be the inverse
    image of $\alpha$ (as a closed subscheme of $G^+$). For $P = MN \in \cF$, we put
    $M^+_{\alpha} = G^+_\alpha \cap M^+$.
\end{itemize}

Let us now introduce the geometric counterparts of the modified kernels. For
$\alpha \in \cA(F)$, $f^+ \in \cS(G^+(\bA))$, and $P \in \cF$,
we define kernel functions on $[H]_{P_H} \times [G']_{P_{G'}}$ by
    \[
    k_{f^+, P, \alpha}(h,g')=
    \sum_{m^+ \in M^+_{\alpha}(F)}
    \int_{N^+(\bA)}
    f^+(m^+ n^+ \cdot (h,g'))
    \rd n^+,
    \]
and
    \[
    k_{f^+,P}(h, g')=
    \sum_{m^+ \in M^+(F)}
    \int_{N^+(\bA)}
    f^+( m^+ n^+ \cdot (h,g'))
    \rd n^+.
    \]

\begin{lemma}   \label{lemma:convergence_geometric_kernel_definition}
There is an $N>0$ and a seminorm $\aabs{\cdot}_{\cS}$ on $\cS(G^+(\bA))$
such that for all $(h, g') \in [H]_{P_H} \times [G']_{P_{G'}}$, we have
    \[
    \sum_{m^+ \in M^{+}(F)}
    \int_{N^+(\bA)}
    \abs{f^+(m^+n^+ \cdot(h,g'))}
    \rd n \rd l
    \leq \aabs{f^+}_{\cS} \aabs{h}_{P_H}^N \aabs{g'}_{P_{G'}}^N.
    \]
In particular the defining expressions of $k_{f^+, P,\alpha}$ and $k_{f^+,
P}$ are absolutely convergent and we have
    \[
    \sum_{\alpha \in \cA(F)} k_{f^+,P, \alpha}(h,g') = k_{f^+,P}(h,g').
    \]
\end{lemma}

\begin{proof}
For any large enough
$d$ and $N$, we can find a continuous seminorm $\aabs{\cdot}_{d}$ on
$\cS(G^+(\bA))$ such that
    \[
    \abs{f^+(m^+n^+ \cdot(h,g'))}
    \leq \aabs{m^+n^+}_{P^+(\bA)}^{-d} \aabs{f^+}_d
    \aabs{h}_{P_H}^N \aabs{g'}_{P_{G'}}^N
    \]
where $\aabs{\cdot}_{P^+(\bA)}$ stands for a height function on
$P^+(\bA)$. The lemma then follows
from~\cite{BP}*{Proposition~A.1.1.(v-vi)}.
\end{proof}

We make the following key observation. We identify $\Res_{E/F} \mathbf{A}_{2n, E}$ with the $4n$-dimensional affine space $\mathbf{A}^{4n}$ over $F$, and denote the morphism $G^+ \to \mathbf{A}^{4n}$ again by $q$. We extend the definition of $k_{f^+, P, \alpha}$ to all $\alpha \in F^{4n}$ by setting $k_{f^+, P, \alpha} = 0$ if $\alpha \not\in \cA(F)$.

Let $C \subset G^+(\bA_{F,
f})$ be an open compact subset such that $\supp f^+ \subset C \times
G^+(F_\infty)$. Since $q(C) \subset \bA_{F, f}^{4n}$ is compact, there exists
$d \in F$ depending only on $C$ such that $q(C) \cap F^{4n} \subset (d \cO_F)^{4n}
$. Then if $k_{f^+, P, \alpha}$ is not identically zero for some $\alpha \in
F^{4n}$ and $P \in \cF$, then $\alpha \in (d\cO_F)^{4n}$.

Put $\Lambda = (d\cO_F)^{4n}$ which is a lattice in $F_\infty^{4n}$. For each
$\alpha \in \Lambda$, take $u_{\alpha} \in C_c^\infty(F^{4n}_\infty)$ and define
the functions $f^+_\alpha = f^+ \cdot u_{\alpha}$ as in
Subsection~\ref{subsec:decomposition_of_Schwartz_functions} .

\begin{lemma}   \label{lem:kernel_function_of_truncated_function_linear}
We have
    \[
    k_{f^+_{\alpha}, P}(h, g') =
    k_{f^+_{\alpha}, P, \alpha}(h, g') = k_{f^+, P, \alpha}(h, g')
    \]
for all $(h, g') \in [H]_{P_H} \times [G']_{P_{G'}}$. In particular, $k_{f^+_\alpha, P}(h, g')$ is independent of the choice of the bump function
$u_\alpha$.
\end{lemma}

\begin{proof}
Take $\beta \in \cA(F)$.
For any $(h, g') \in [H]_{P_H} \times [G']_{P_{G'}}$, and any $m^+ \in
M^+_\beta(F)$, $n^+ \in N^+(\bA)$, we have $m^+n^+ \cdot (h, g') \in
q^{-1}(\beta)$. By the definition of $f^+_{\alpha}$, we
have $f^+_{\alpha}(m^+ n^+ \cdot (h, g')) = f^+(m^+ n^+ \cdot (h, g'))$ if $\alpha = \beta$ and $0$ if
$\alpha \not=\beta$. The result follows.
\end{proof}

The kernel $K_{f_+, P}$ introduced in
Subsection~\ref{subsec:modified_kernel_1} is closely related to $k_{f^+, P}$. First we recall
that the action $\mathrm{R}_{\mu^{-1}}$ of $G_n(\bA)$ on $\cS(\bA_{E, n})$ given
by
    \[
    \mathrm{R}_{\mu^{-1}}(g) \Phi(x) =
    \mu(\det g)^{-1} \abs{\det g}_E^{\frac{1}{2}} \Phi(xg), \quad
    g \in G_n(\bA), \quad \Phi \in \cS(\bA_{E, n}).
    \]
We define a Fourier transform $\cS(\bA_{E, n}) \to \cS(\bA_{E, n}^- \times
\bA_E^{n, -})$ by
    \begin{equation}    \label{eq:global_FT_GL}
    \Phi^\dag(w, v) = \int_{\bA_{ n}} \Phi(x+w)
    \psi((-1)^{n} \tau x v) \rd x, \quad (w, v) \in
    \bA_{E, n}^- \times \bA_E^{n, -}.
    \end{equation}
Then there is a unique action $\mathrm{R}^\dag_{\mu^{-1}}$ of $G_n(\bA)$ on
$\cS(\bA_{E, n}^- \times \bA_E^{n, -})$ such that the Fourier transform is
equivariant. Direct computation gives that if $g' \in G'_n(\bA)$ and
$\Phi \in \cS(\bA_{E, n})$ then
    \[
    \mathrm{R}_{\mu^{-1}}^{\dag}(g')\Phi^\dag(w, v) = \eta(\det g')
    \Phi^\dag(wg', g'^{-1} v).
    \]
This Fourier transform extends to a continuous linear
isomorphism
    \[
    \cS(G_+(\bA)) \to \cS(G^+(\bA)), \quad
    f_+ \mapsto f_+^\dag,
    \]
given by
    \begin{equation}    \label{eq:dag_map_global}
    f_+^\dag(g, w, v) = \left(\mathrm{R}_{\mu^{-1}}^\dag(g_1^{-1})
    f_+(g, \cdot)^\dag \right)(w, v), \quad
    (g, w, v) \in G^+(\bA).
    \end{equation}
Here the expression on the right hand side is interpreted as follows. We
evaluate $f_+$ at $g$ first to obtain a Schwartz function on $\bA_{E, n}$, then take the Fourier transform $-^\dag$, then make $g_1^{-1}$ act on it via
$\mathrm{R}_{\mu^{-1}}^\dag$, and finally evaluate the result at $(w, v)$. We denote by
$-_\dag$ the inverse integral transform of $-^\dag$. If $f_+ = f \otimes
\Phi$ where $f \in \cS(G(\bA))$ and $\Phi \in \cS(\bA_{E, n})$ then we have a
cleaner expression
    \[
    f_+^\dag(g, w, v) = f(g)
    \left( \mathrm{R}_{\mu^{-1}}(g_1^{-1}) \Phi \right)^\dag(w, v) =
    f(g)
    \mathrm{R}_{\mu^{-1}}^\dag(g_1^{-1}) \Phi^\dag(w, v).
    \]

\begin{lemma}   \label{lemma:k=K}
For all $(h, g') \in [H]_{P_H} \times [G']_{P_{G'}}$ and $f_+ \in
\cS(G_+(\bA))$, we have
    \[
    K_{f_+, P}(h, g') = k_{f_+^\dag, P}(h, g')
    \eta(\det g_1').
    \]
\end{lemma}

\begin{proof}
By Lemma~\ref{lemma:convergence_geometric_kernel_definition}, for fixed $h$
and $g'$, both sides are continuous linear forms on $f_+$. Therefore we only
need to prove the lemma when $f_+ = f \otimes \Phi$ where $f \in
\cS(G(\bA))$ and $\Phi \in \cS(\bA_{E, n})$. To best illustrate the ideas,
we prove the lemma in the case $P = G$. The general case follows by the same
computation but with messier notation. By definition $K_{f \otimes \Phi}(h, g')$
equals
    \[
    \sum_{\gamma \in G(F)} \sum_{x \in E_n} f(h^{-1} \gamma g')
    \mathrm{R}_{\mu^{-1}}(h) \Phi(x).
    \]
This equals
    \[
    \sum_{\gamma \in G(F)} \sum_{x \in E_n} f(h^{-1} \gamma g')
    \mathrm{R}_{\mu^{-1}}(g_1^{-1} \gamma_1^{-1} h)
    \Phi(xg_1') \abs{\det g_1'} \eta(\det g_1').
    \]
The Poisson summation formula then gives
    \[
    \sum_{\gamma \in G(F)} \sum_{(w, v) \in E_n^- \times E^{n, -}}
    f(h^{-1} \gamma g')
    (\mathrm{R}_{\mu^{-1}}(g_1'^{-1} \gamma_1^{-1} h) \Phi)^\dag
    (wg_1, g_1^{-1} v) \eta(\det g_1'),
    \]
which equals $k_{f_+^\dag, P}(h, g') \eta(\det g_1')$.
\end{proof}

\begin{lemma}   \label{lemma:FT_invariance_GL}
Let $f^+ \in \cS(G^+(\bA))$. We have
    \[
    \eta(g_1')
    ((h, g') \cdot f^+)_{\dag} =
    (h, g') \cdot (f^+)_\dag,
    \]
where on the left hand, the action is the one induced from the
action ~\eqref{eq:action_G^+},
and on the right hand side the action is the one defined in
Proposition~\ref{prop:invariant_I_chi}.
\end{lemma}

\begin{proof}
This is a direct computation.
\end{proof}

Let $T \in \fa_{n+1}$ be a truncation parameter, and let $\alpha \in \cA(F)$. We
define
    \[
    k_{f^+, \alpha}^T(h, g') = \sum_{P \in \cF} \epsilon_P
    \sum_{\substack{\gamma \in P_H(F) \bs H(F)\\
    \delta \in P_{G'}(F) \bs G'(F)}}
    \widehat{\tau}_{P_{n+1}}(H_{P_{n+1}}(\delta_1 g'_1) - T_{P_{n+1}})
    k_{f^+, P, \alpha}(\gamma h, \delta g'),
    \]
and
    \[
    k_{f^+}^T(h, g') = \sum_{P \in \cF} \epsilon_P
    \sum_{\substack{\gamma \in P_H(F) \bs H(F)\\
    \delta \in P_{G'}(F) \bs G'(F)}}
    \widehat{\tau}_{P_{n+1}}(H_{P_{n+1}}(\delta_1 g'_1) - T_{P_{n+1}})
    k_{f^+, P}(\gamma h, \delta g').
    \]

\begin{lemma}   \label{lem:integration_of_modified_kernel_seminorm_linear}
For $f^+ \in \cS(G^+(\bA))$, if $T$ is sufficiently positive the
integral
    \[
    \int_{[H]\times [G']} \Abs{k_{f^+}^T(h, g')} \rd h \rd g'
    \]
is convergent and defines a seminorm on $\cS(G^+(\bA))$. Put
    \[
    i^T(f^+) = \int_{[H]\times [G']}k_{f^+}^T(h, g')
    \eta_{G'}(g') \rd h \rd g'.
    \]
Then $i^T(f^+)$ is the restriction of a polynomial exponential whose purely
polynomial part is a constant that equals $I((f^+)_\dag)$.
\end{lemma}

\begin{proof}
It follows from Lemma~\ref{lemma:k=K} that if $f^+ \in \cS(G_+(\bA))$ then
    \[
    k_{f^+}^T(h, g') \eta(\det g_1') = K_{f^+_{\dag}}^T(h, g')
    \]
for all $(h, g') \in [H] \times [G']$. The lemma then follows directly from
Theorem~\ref{thm:convergence_first_GL}.
\end{proof}

\subsection{The coarse geometric expansion}
We now develop the coarse geometric expansion.

\begin{theorem} \label{thm:geometric_linear}
We have the following assertions.
\begin{enumerate}
\item For $T$ sufficiently positive, the expression
     \begin{equation*}
        \sum_{\alpha \in \cA(F)} \int_{[G']}  \int_{[H]}
        \left| k_{f^+, \alpha}^T(h, g') \right|
        \rd h \rd g'
        \end{equation*}
    is convergent and defines a continuous semi-norm on $\cS(G^+(\bA))$.

\item For $\alpha \in \cA(F)$, we define
    \[
    i_\alpha^T(f^+) = \int_{[G']} \int_{[H]} k^T_{f^+, \alpha}(h, g')
    \eta_{G'}(g') \rd h \rd g'.
    \]
    Then as a function of $T$, when $T$ is sufficiently positive, $i^T_\alpha(f)$ is the restriction of a
    polynomial exponential function whose purely polynomial part is a constant. We
    denote this constant by $i_{\alpha}(f^+)$.

\item The distribution $f^+ \mapsto i_\alpha(f^+)$ satisfies the invariance
    property that
        \[
        i_{\alpha}((h, g') \cdot f^+) =
        \eta_{G'}(g')  i_{\alpha}(f^+),
        \]
    where the action $(h, g') \cdot f^+$ is induced from the
    action~\eqref{eq:action_G^+}.
\end{enumerate}
\end{theorem}

\begin{proof}
Recall that we introduced a lattice $\Lambda \subset F_\infty^N$, and defined functions $u_{\alpha}$ and $f^+_\alpha$ for $\alpha \in \Lambda$ before Lemma~\ref{lem:kernel_function_of_truncated_function_linear}.  By Proposition~\ref{prop:decomposition_of_Schwartz}, $(f^+_\alpha)_{\alpha \in
\Lambda}$ is absolutely summable in $\cS(G^+(\bA))$. Therefore by
Lemma~\ref{lem:integration_of_modified_kernel_seminorm_linear}, we have
    \[
    \sum_{\alpha \in \Lambda} \int_{[H]} \int_{[G']}
    \left|  K_{f^+_\alpha}^T(h,g') \right| \rd h \rd g' < \infty.
    \]
By Lemma~\ref{lem:kernel_function_of_truncated_function_linear}, we have
    \[
    K_{f^+_\alpha}^T(h,g') = K_{f^+, \alpha}^T(h, g').
    \]
By the construction of the lattice $\Lambda$ we conclude that
    \[
    \sum_{\alpha \in \cA(F)}  \int_{[H]} \int_{[G']}
    \left| K_{f^+,\alpha}^T(h,g') \right| \rd h \rd g' =
    \sum_{\alpha \in \Lambda}  \int_{[H]} \int_{[G']}
    \left|  K_{f_\alpha}^T(h,g') \right| \rd h \rd g'.
    \]
This proves the first assertion on absolute convergence. By the uniform boundedness principle, it defines a continuous semi-norm.

By
Lemma~\ref{lem:kernel_function_of_truncated_function_linear} and Lemma~\ref{lemma:k=K} we have
    \[
    i^T_{\alpha}(f^+) = i^T(f^+_{\alpha}) = I^T( (f^+_{\alpha})_{\dag}).
    \]
This implies the second assertion. The third assertion follows from
Lemma~\ref{lemma:FT_invariance_GL} and
Proposition~\ref{prop:invariant_I_chi}.
\end{proof}

\subsection{Synthesis of the results: the coarse relative trace formula}

We now summarize what we have done. For $f_+ \in \cS(G_+(\bA))$ we put
    \[
    I_{\alpha}^T(f_+) = i_{\alpha}^T(f_+^\dag), \quad
    I_{\alpha}(f_+) = i_{\alpha}(f_+^\dag).
    \]
Then Theorem~\ref{thm:geometric_linear} tells us that if $T$ is sufficiently
positive then $I_{\alpha}^T(f_+)$ is the restriction of a polynomial
exponential and its purely polynomial part is a constant that equals
$I_{\alpha}(f_+)$. By Lemma~\ref{lemma:FT_invariance_GL}, the distribution
$I_{\alpha}$ is left $H(\bA)$-invariant and right
$(G'(\bA),\eta)$-equivariant in the sense of
Proposition~\ref{prop:invariant_I_chi}.

We summarize the coarse relative trace formulae on the general linear groups
as the following theorem.

\begin{theorem} \label{thm:coarse_GL}
Let $f_+ \in \cS(G_+(\bA))$ be a test function. Then we have
    \[
    \sum_{\chi \in \fX(G)} I_{\chi}(f_+) =
    \sum_{\alpha \in \cA(F)} I_{\alpha}(f_+).
    \]
The summations on both sides are absolutely convergent and each summand is left $H(\bA)$-invariant and right
$(G'(\bA),\eta)$-equivariant.
\end{theorem}

This is simply a combination of
Theorem~\ref{thm:coarse_spectral_expansion_GL},
Proposition~\ref{prop:invariant_I_chi},
Lemma~\ref{lem:integration_of_modified_kernel_seminorm_linear} and
Theorem~\ref{thm:geometric_linear}.

\section{The coarse spectral expansion: unitary groups}
\label{sec:u_spectral}

\subsection{Setup}
\label{subsec:spec_unitary_setup}

The following notation will be used throughout this section.
\begin{itemize}
    \item Let $(V, q_V)$ be a nondegenerate skew-Hermitian space  of dimension $n$. Put $\U_V
        =  \U(V) \times \U(V)$. If $g \in \U_V$, without mentioning
        explicitly the contrary, we will denote by $g = (g_1, g_2)$ where
        $g_i \in \U(V)$. Let $\U'_V \subset \U_V$ denote the diagonal
        subgroup, which is isomorphic to $\U(V)$.

    \item Let $S(V)$ be the Heisenberg group attached to $V$, and $J(V) =
        S(V) \rtimes \U(V)$ be the Jacobi group (see Subsection~\ref{subsec:u_Jacobi}). Put $\widetilde{\U_V} =
        \U(V) \times J(V)$. If $g \in \U(V)$, an element in $J(V)$ whose
        image in $\U(V)$ is $g$ is usually denoted by $\widetilde{g}$.
        An element in $\widetilde{\U_V}$ is usually denoted by $\widetilde{x}$.
        This means $\widetilde{x} = (x_1, \widetilde{x_2})$,
        $x = (x_1, x_2) \in \U_V$ and the image of
        $\widetilde{x}$ in $\U_V$ is $x$.

    \item The group $\U_V'$ diagonally embeds in $\widetilde{\U_V}$.
        Its image is
        again denoted by $\U_V'$. There is a natural map $J(V) \to
        \widetilde{\U_V}$ and we let $\widetilde{\U'_V}$ its image.

    \item We keep the notation from Subsection~\ref{subsec:u_Jacobi}. In
        particular, we fix a minimal parabolic subgroup $P_0$ of $\U(V)$
        which fixes the maximal isotropic flag~\eqref{eq:maximal_isotropic_flag}.
        Standard parabolic subgroups of $\U(V)$ are those containing $P_0$.
        Let $\cF_V$ be the subset of standard
        D-parabolic subgroups of $J(V)$, and let
        $\cF'_V$ be the set of standard parabolic subgroup of $\U'_V$. We put
        $P' = P \cap \U(V)$ for $P \in \cF_V$. Then by
        Lemma~\ref{lemma:Jacobi_parabolic_subgroup_bijection}, the map
        $P \mapsto P'$ is a bijection from $\cF_V$ to $\cF_{V'}$.

    \item Let $P \in \cF_V$. We denote by $P_1$ and $P_2$ respectively
    the subgroup $P'$ of the first and second factor of $\U_V$. We put
    \begin{equation}
        \label{eq:P_tilde}
        \widetilde{P} = P_1 \times P, \quad P_{\U} = P_1 \times P_2,
    \end{equation}
     which are D-parabolic subgroups of $\widetilde{\U_V}$ and
     $\U_V$ respectively. The notation $P'$ usually specifically means
     the parabolic subgroup of $\U_V'$.

    \item  Let $\Res V$ be the symplectic space defined in
        Subsection~\ref{subsec:u_Jacobi}. We fix a polarization
        $\Res V= L \oplus L^\vee$ as in Subsection~\ref{subsec:theta_series_U}.
        We have the Weil representation $\omega = \omega_{\psi, \mu}$ of
        $J(V)(\bA)$ realized on $\cS(L^\vee(\bA))$. For $\phi \in \cS(L^\vee(\bA))$ and $P \in \cF_V$ we have theta theta function $\prescript{}{P}{\theta}(\cdot, \phi)$. We also have
        $\omega^\vee = \omega_{\psi^{-1}, \mu^{-1}}$ and $\prescript{}{P}{\theta}^\vee(\cdot, \phi)$ defined in terms of $\omega^\vee$.

    \item Put $\U_{V,+} = \U_V \times L^\vee \times L^\vee$.
        The group structure is given by the product group structure of
        $\U(V)$ and the additive group $L^\vee \times L^\vee$. Let $P \in \cF_V$.
        We put
            \[
            M_{P, +} = M_{P_{\U}} \times M_{L^\vee} \times M_{L^\vee}, \quad
            N_{P, +} = N_{P_{\U}} \times N_{L^\vee} \times N_{L^\vee} = N_{P_{\U}},\quad
            P_+ = M_{P, +} N_{P, +},
            \]
        where we recall that $M_{L^\vee}$ and $N_{L^\vee}$ are defined in Subsection~\ref{subsec:theta_series_U}.
        These are subgroups of $\U_{V, +}$. We often write an element in $M_{P, +}$
        as $(m, l_1, l_2)$ where $m \in M_{P_{\U}}$ and $l_1,l_2 \in L^\vee$.

    \item Put $\fa_0 = \fa_{P_0}$. A truncation parameter is an element in
        $\fa_0$.
\end{itemize}

\subsection{Technical preparations}

Let $P \in \cF_V$ and $w$ be a weight on $[\widetilde{\U_V}]_{\widetilde{P}}$.
We can define various function spaces as in
Subsection~\ref{subsec:spaces_of_function}. Of particular interest to us are
$\cS_w([\widetilde{\U_V}]_{\widetilde{P}}, \psi)$ and
$\cT_w([\widetilde{\U_V}]_{\widetilde{P}}, \psi)$.

The ``approximation by the constant term'' for the group $\widetilde{\U_V}$ takes the following form.

\begin{prop} \label{lem:tildeu_approximation_by_constant_term}
Let $N>0,r \ge 0$, $X \in \cU(\widetilde{\fu_{V}}_{\infty})$ and $P,Q \in
\cF_V$. Then there exists a continuous seminorm $\| \cdot \|_{N,X,r}$ on
$\cT_N([\widetilde{\U_V}]_{\widetilde{Q}},\psi)$, such that
    \[
    \left| \mathrm{R}(X)\varphi(g)-
    \mathrm{R}(X)\varphi_{\widetilde{P}}(g) \right|
    \le \| g \|^N_P d_{P_{\U}}^{Q_{\U}}(g)^{-r} \| \varphi \|_{N,X,r}
    \]
holds for all $\varphi \in \cT_N([\widetilde{\U_V}]_{\widetilde{Q}},\psi)$ and
$g \in \U_V(\bA)$.
\end{prop}

\begin{proof}
For $g \in \U_V(\bA)$, we have $ d_{\widetilde{P}}^{\widetilde{Q}}(g) = d_{P_1}^{Q_1}(g_1)d_P^Q(g_2)$ and $d_{P_{\U}}^{Q_{\U}}(g) =   d_{P_1}^{Q_1}(g_1)d_{P_2}^{Q_2}(g_2)$.
Therefore by Lemma ~\ref{lem:d_P^Q_unitary_equivalent}, we have $\min\{ d_{P_{\U}}^{Q_{\U}}(g),d_{P_1}^{Q_1}(g_1)d_{P_2}^{Q_2}(g_2)^{\frac 12} \} \ll
    d_{\widetilde{P}}^{\widetilde{Q}}(g) \ll d_{P_{\U}}^{Q_{\U}}(g)$,
and hence the proposition follows from Theorem~\ref{thm:approximation_by_constant_term}.
\end{proof}

Similarly to what we have done in the case of general linear groups, we introduce
various auxiliary spaces of functions.

Recall that we have the space of functions $\cT_{\cF_V'}(\U'_V)$ introduced in
Subsection~\ref{subsec:arthur's_truncation}. Because there is a bijection between
$\cF_V$ and $\cF_V'$, we denote it by $\cT_{\cF_V}(\U'_V)$ and use $\cF_{V}$
as indices. By Theorem~\ref{thm:approximation_by_constant_term}, the
space $\cT_{\cF_V}(\U'_V)$ consists of tuples
    \[
    (\prescript{}{P}{\varphi})_{P \in \cF_V} \in
    \prod_{P \in \cF_V} \cT([\U_V']_{P'}),
    \]
such that
    \begin{equation}    \label{eq:characterization_T_U'_V}
    \prescript{}{P}{\varphi} - (\prescript{}{Q}{\varphi})_{P'}
    \in \cS_{d_{P'}^{Q'}}([\U_V']_{P'})
    \end{equation}
for any $P \subset Q \in \cF_V$.

\begin{lemma}   \label{lemma:multiplication_by_theta}
If $({}_P \varphi) \in \cT_{\cF_V}(\U_V')$ and $\phi \in \cS(L^\vee(\bA))$,
then the family of products
    \[
    P \mapsto ({}_P \varphi  \prescript{}{P}{\theta(\cdot, \phi)})
    \]
belongs to $\cT_{\cF_V}(\U_V')$.
\end{lemma}

\begin{proof}
By Proposition~\ref{prop:u_theta_property} and the above
characterization~\eqref{eq:characterization_T_U'_V}, the family $P \mapsto
\prescript{}{P}{\theta(\cdot, \phi)}$ belongs to $\cT_{\cF_V}(\U_V')$. We thus
prove a stronger statement: if $(\prescript{}{P}{\varphi})$,  $(\prescript{}{P}{\varphi'}) \in \cT_{\cF_V}(\U_V')$, then the family
    \[
    P \mapsto (\prescript{}{P}{\varphi} \prescript{}{P}{\varphi'})
    \]
belongs to $\cT_{\cF_V}(\U_V')$. Using the Leibniz rule, we are reduced to
that there exists $N_0>0$ such that for any $X,Y \in \cU((\fu'_V)_\infty)$
and any $r>0$, we have
    \[
    \left| \mathrm{R}(X) {}_P \varphi(g) \mathrm{R}(Y) {}_P \varphi'(g)
    -  \mathrm{R}(X) {}_Q \varphi(g) \mathrm{R}(Y) {}_Q \varphi'(g) \right|
    \ll_{X,Y,r} \|g\|_{P'}^{N} d_{P'}^{Q'}(g)^{-r}
    \]
for any $g \in [\U_V']_{P'}$. The left hand side is bounded by
    \[
     \left| \mathrm{R}(X) {}_P \varphi(g)
     \left(\mathrm{R}(Y) {}_P \varphi'(g)
     - \mathrm{R}(Y) {}_Q \varphi'(g) \right) \right|
     + \left| \mathrm{R}(Y) {}_Q \varphi'(g)
     \left( \mathrm{R}(X) {}_P \varphi(g)
     - \mathrm{R}(X) {}_Q \varphi(g) \right) \right|.
    \]
The desired inequality then follows.
\end{proof}

For $P \in \cF_V$, define a weight function $\Delta_P$ on $[\U_V]_{P_{\U}}$ by
    \[
    \Delta_P(g) = \inf_{\gamma \in M_{P'}(F)N_{P'}(\bA)}
    \aabs{g_1^{-1} \gamma g_2}_{P'}.
    \]
For $P,Q \in \cF_V$ with $P \subset Q$, define a weight
$d_P^{Q,\Delta}$ on $[\U_V]_{P_{\U}}$ by
    \[
    d_P^{Q,\Delta}(g) = \min \{ d_{P_1}^{Q_1}(g_1), d_{P_2}^{Q_2}(g_2) \}.
    \]
Pulling back under the projection $[\widetilde{\U_V}]_{\widetilde{P}} \to
[\U_V]_{P_{\U}}$, we get two weights on $[\widetilde{\U_V}]_{\widetilde{P}}$,
which are still denoted by $\Delta_P$ and $d_P^{Q,\Delta}$.

We define the space $\cT_{\cF_V}^\Delta(\U_V)$ (resp.
$\cT_{\cF_V}^\Delta(\widetilde{\U_V},\psi)$) to be the space of functions
    \begin{equation} \label{eq:TDelta_first_condition}
     (\prescript{}{P}{\varphi}) \in \prod_{P \in \cF_V} \cS_{\Delta_P}([\U_V]_{P_{\U}})
     \quad \text{resp.} \quad
     \prod_{P \in \cF_V} \cS_{\Delta_P}([\widetilde{\U_V}]_{\widetilde{P}},\psi).
    \end{equation}
such that for any $P \subset Q \in \cF_V$, we have
    \[
    {}_P \varphi - ({}_Q \varphi)_P \in \cS_{d_P^{Q,\Delta}}([\U_V]_{P_{\U}})
    \quad \text{resp.}
    \quad {}_P \varphi - ({}_Q \varphi)_{\widetilde{P}}
    \in \cS_{d_P^{Q,\Delta}}([\widetilde{\U_V}]_{\widetilde{P}},\psi).
    \]
Since $d_P^{Q,\Delta} \ll d_P^Q$, by
Theorem~\ref{thm:approximation_by_constant_term} and
Proposition~\ref{lem:tildeu_approximation_by_constant_term} respectively,
that a family $({}_P \varphi)$ belongs to $\cT_{\cF_V}^\Delta(\U_V)$ (resp.
$\cT_{\cF_V}^\Delta(\widetilde{\U_V}, \psi)$) is equivalent to
$\prescript{}{P}{\varphi} \in \cS_{\Delta_P}([\U_V]_{P_{\U}})$ (resp.
$\cS_{\Delta_P}([\widetilde{\U_V}]_{\widetilde{P}}, \psi)$ for each $P$ and that
for all $P \subset Q \in \cF_V$, there exists an $N>0$ such that for all $X \in
\cU((\fu_{V})_{\infty})$ and all $r>0$, we have
    \begin{equation} \label{eq:TDelta_equivalent}
      \left| \mathrm{R}(X) {}_P\varphi(g)
       - \mathrm{R}(X) {}_Q \varphi(g) \right| \ll_{r,X}
       \|g\|_P^N d_P^{Q,\Delta}(g)^{-r}.
    \end{equation}
holds for all $g \in [\U_V]_{P_{\U}}$ (resp. $[\widetilde{\U_V}]_{\widetilde{P}}$).

\begin{lemma} \label{lem:T_space_relation}
We have the following assertions.
\begin{enumerate}
\item For $({}_P \varphi) \in \cT^\Delta_{\cF_V}(\widetilde{\U_V},\psi)$, then
the family of restrictions $P \mapsto ({}_P \varphi|_{[\U_V]_{P_{\U}}})$
belong to $\cT^\Delta_{\cF_V}(\U_V)$.

\item For $({}_P \varphi) \in \cT^\Delta_{\cF_V}(\U_V)$, then the family of
    restrictions $P \to ({}_P \varphi|_{[\U_V']_{P'}})$
    belongs to $\cT_{\cF_V}(\U_V')$.
\end{enumerate}
\end{lemma}

\begin{proof}
The first follows from the characterizations~\eqref{eq:TDelta_equivalent}
of elements in $\cT^\Delta_{\cF_V}(\U_V)$ and
$\cT^\Delta_{\cF_V}(\widetilde{\U_V},\psi)$. The second follows from the
additional fact that $d_{P}^{Q, \Delta}|_{[\U_V']_{P'}} = d_{P'}^{Q'}$.
\end{proof}

We fix a $\phi_0 \in \cS(L^\vee(\bA))$ such that $\aabs{\phi_0}_{L^2} = 1$. For
$f \in \cS(\U_V(\bA))$ and $\phi_1 \in \cS(L^\vee(\bA))$, we define a function
$\widetilde{f} \in \cS(\widetilde{\U_V}(\bA), \psi)$ by
    \begin{equation} \label{eq:u_tildef}
    \widetilde{f}(gs) = f(g^{-1}) \langle \omega(s) \phi_0, \overline{\phi_1} \rangle,
    \quad s \in S(V)(\bA), \ g \in \U_V(\bA).
    \end{equation}

For $\varphi \in \cT^0([\U_V]_{P_V})$, as in the case of general linear groups, we define a measure $\varphi \cdot \overline{\prescript{}{P}{\theta(\cdot,
\phi_0)}} \in \cT^0([\widetilde{\U_V}]_{\widetilde{P}},\psi)$ by
\begin{equation}
\label{eq:radon_defi_uni}
    \langle \varphi \cdot \overline{\prescript{}{P}{\theta(\cdot, \phi_0)}},
    \beta \rangle =
    \int_{[\widetilde{\U_V}]_{P_{\U}}} \int_{[S(V)]_{P_S}}
    \beta(sg)  \overline{\prescript{}{P}{\theta(sg, \phi_0)}} \rd s \varphi(g),
\end{equation}
where $\beta \in C_c([\widetilde{\U_V}]_{\widetilde{P}})$.

\begin{lemma}   \label{lemma:smoothened_constant_term_U}
Let $\varphi \in  \cT^0([\U_V'])$, then the family
    \[
    P \in \cF_V \mapsto \mathrm{R}(\widetilde{f})
    \left( \varphi_{P'} \cdot
    \overline{\prescript{}{P}{\theta(\cdot, \phi_0)}} \right)
    \]
belongs to $\cT^\Delta_{\cF_V}(\widetilde{\U_V}, \psi^{-1})$. Moreover if
$y \in [\U_V]_{P_{\U}}$ we have
    \[
    \mathrm{R}(\widetilde{f})
    (\varphi \cdot \overline{\prescript{}{P}{\theta(\cdot, \phi_0)}})(y) =
    \int_{[\U_V']_{P'}} K_{f, P}(x, y) \theta^\vee(x, \phi_1) \varphi(x).
    \]
In particular the composition of this map followed by the restriction to
$[\U_V]_{P_{\U}}$ is independent from the choice of $\phi_0$.
\end{lemma}

\begin{proof}
This is proved in the same way as in the case of general linear groups, and in
particular Lemma~\ref{lemma:distribution_extension_by_theta},
Lemma~\ref{lem:smoothed_constant_term} and
Lemma~\ref{lemma:interpretation_of_the_kernel}.
\end{proof}

\subsection{A modified kernel}

\label{subsec:u_main_thm_spectral}

Let $\chi \in \fX(\U_V)$ be a cuspidal datum. For $f_+=f \otimes
\phi_1 \otimes \phi_2$ as above, we put
\begin{equation}
\label{eq:unitary_kernel_defi}
      K_{f \otimes \phi_1 \otimes \phi_2, P, \chi}
    (\widetilde{x}, \widetilde{y})
    = K_{f, P_{\U}, \chi}(x, y)
    \prescript{}{P}{\theta}^\vee(\widetilde{x}_2, \phi_1)
    \prescript{}{P}{\theta}(\widetilde{y}_2, \phi_2),
    \quad \widetilde{x},\widetilde{y} \in [\widetilde{\U_V}]_{\widetilde{P}}.
\end{equation}

Where we denote $x,y$ to be the image of $\widetilde{x},\widetilde{y}$ in $[\U_V]_{P_{\U}}$ respectively and $\widetilde{x}_2,\widetilde{y}_2 \in [J(V)]$ are the second component of $\widetilde{x}$ and $\widetilde{y}$.
Let us now explain that the function $K_{f \otimes \phi_1 \otimes
\phi_2, P, \chi}$ extends continuously to a smooth function $K_{f_+, P, \chi}$ for
all $f_+ \in \cS(\U_{V, +}(\bA))$. By Lemma~\ref{lemma:estimate_kernel} and
Proposition~\ref{prop:u_theta_property}, there exists an $N_0>0$ such that for
any $N>0$ and $X \in \cU(\widetilde{\fu}_{V,\infty})$, there is a continuous semi-norm $\aabs{\cdot}_\cS$ on $\cS(\U_{V,
+}(\bA))$ such that
    \begin{equation}  \label{eq:any_tensor_estimate}
    \sum_{\chi \in \fX(\U)}
    \Abs{ \mathrm{R}_{\widetilde{y}}(X)K_{f_+,P,\chi}( \widetilde{x},  \widetilde{y})}
    \le \|f_+\|_{\cS} \|  \widetilde{x} \|_{\widetilde{P}}^{-N}
    \aabs{\widetilde{y}}_{\widetilde{P}}^{N+N_0}.
    \end{equation}
holds for all $f_+ \in \cS(\U_V(\bA)) \otimes \cS(L^\vee(\bA)) \otimes
\cS(L^\vee(\bA))$ (algebraic tensor product). Now let $f_+ \in \cS(\U_{V,
+}(\bA))$ and $f_{+, n} \in \cS(\U_V(\bA)) \otimes \cS(L^\vee(\bA)) \otimes
\cS(L^\vee(\bA))$ a sequence of functions approaching $f_+$. Because of the
estimate~\eqref{eq:any_tensor_estimate}, the sequence
$K_{f_{+, n}, P, \chi}$
is convergent to a function on $[\widetilde{\U_V}]_{\widetilde{P}} \times [\widetilde{\U_V}]_{\widetilde{P}}$, and this convergence is locally
uniform for $(\widetilde{x}, \widetilde{y}) \in [\widetilde{\U_V}]_{\widetilde{P}} \times [\widetilde{\U_V}]_{\widetilde{P}}$. We denote this function by $K_{f_+,P,
\chi}$. It is clearly independent of the choice of the sequence approximating
$f_+$. Because the convergence is locally uniform, $K_{f_+, P, \chi}$ is a
smooth function. Moreover the estimate~\eqref{eq:any_tensor_estimate}
continues to hold for $K_{f_+, P, \chi}$. By the symmetry of $\widetilde{x}$
and $\widetilde{y}$, the estimate~\eqref{eq:any_tensor_estimate} holds when
$\widetilde{x}$ and $\widetilde{y}$ on the right hand side are switched (and with a possibly different $\| \cdot \|_\cS$).

For $T \in \fa_0$, $\chi \in \fX(\U_V)$, and $x, y \in [\widetilde{\U_V}]$, we
define modified kernels
    \[
    K_{f_+, \chi}^T(x,y)= \sum_{P \in {\cF_V}} \epsilon_P
    \sum_{\substack{\gamma \in P'(F) \backslash \U_V'(F)
    \\ \delta \in P'(F)
    \backslash \U_V'(F)}}
    \widehat{\tau}_{P'}(H_{P'}(\delta y)-T_{P'})
    K_{f_+,P, \chi}(\gamma x,\delta y).
    \]
and
    \[
    K_{f_+}^T(x,y) = \sum_{P \in {\cF_V}}
    \epsilon_P
    \sum_{\substack{\gamma \in P'(F) \backslash \U_V'(F)
    \\ \delta \in P'(F) \backslash \U_V'(F)}}
    \widehat{\tau}_{P'}(H_{P'}(\delta y)-T_{P'})
    K_{f_+, P}(\gamma x, \delta y).
    \]
Here $\epsilon_P = (-1)^{\dim \fa_{P'}}$ and $\widehat{\tau}_{P'}$ is the
characteristic function of a certain cone in $\fa_{P'}$ defined in
Subsection~\ref{subsec:reduction_theory}. In these definitions, the
convergence of the inner sum can be seen as follows. For fixed $x$ and $y$,
there are only finitely many $\delta \in P'(F) \backslash \U_V'(F)$ (depending
on $y$) such that $\widehat{\tau}_{P'}(H_{P'}(\delta y)-T_{P'}) \ne 0$,
cf.~\cite{Arthur3}*{Lemma~5.1}. Hence the
estimate~\eqref{eq:any_tensor_estimate} implies that the summations defining $K_{f_+}^T$ and
$K_{f_+,\chi}^T$ are absolutely convergent.

\begin{prop} \label{prop:u_asymptotic_modified_kernel}
For every $N>0$, there exists a continuous semi-norm $\| \cdot \|_{\cS,N}$ on
$\cS(\U_{V, +}(\bA))$ such that for every $f_+ \in \cS(\U_{V, +}(\bA))$ and
$T$ sufficiently positive, we have
    \begin{equation}    \label{eq:u_asymptotic_modified_kernel}
    \sum_{\chi \in \fX(\U_V)}
    \left| K_{f_+,\chi}^T(x,y)-K_{f_+,\chi}(x,y) F^{\U_V'}(x,T) \right|
    \le e^{-N\|T\|} \| x \|_{\U_V'}^{-N} \| y \|_{\U_V'}^{-N}
    \| f_+ \|_{\cS,N}.
    \end{equation}
In particular for $f_+ \in \cS(\U_{V,+}(\bA))$ and $T$ sufficiently positive, the
expression
    \[
    \sum_{\chi \in \fX(\U_V)} \int_{[\U_V'] \times [\U_V']}
    \left| K_{f_+,\chi}^T(x,y) \right| \rd x \rd y
    \]
is finite and defines a continuous semi-norm on $\cS(\U_{V,+}(\bA))$.
\end{prop}

\begin{proof}
First, since the center of $\U'_V$ is anisotropic, for a fixed $T$ the
function $x \mapsto F^{\U'_V}(x, T)$ is compactly supported. Therefore the
second assertion on the absolute convergence follows from the
estimate~\eqref{eq:u_asymptotic_modified_kernel}.

Next we note that each summand in~\eqref{eq:u_asymptotic_modified_kernel} is
continuous in $f_+$. Thus by continuity we only need to
prove~\eqref{eq:u_asymptotic_modified_kernel} when $f_+ = f \otimes \phi_1 \otimes \phi_2$ where $f \in \cS(\U_V(\bA))$, $\phi_1,
\phi_2 \in \cS(L^\vee(\bA))$. We will assume that this is the case from now on.

We fix a $\phi_0 \in \cS(L^\vee(\bA))$ with $\aabs{\phi_0}_{L^2} = 1$. Consider
the following sequence of map
    \[
    \begin{tikzcd}
    {\cT^0([\U_V'])} \arrow[r] & \cT^\Delta_{\cF_V}(\widetilde{\U_V},\psi^{-1})
    \arrow[r,"|_{\U_V}"] & \cT^\Delta_{\cF_V}(\U_V) \arrow[r,"|_{\U'_V}"]
    & \cT_{\cF_V}(\U'_V) \arrow[r,"\cdot \theta"]
    & \cT_{\cF_V}(\U'_V) \arrow[r, "\Lambda^T", bend left]
    \arrow[r, "\Pi^T"', bend right] & {\cS^0([\U'_V])},
    \end{tikzcd}
    \]
where the first map sends $\varphi \in \cT^0([\U_V'])$ to the family
    \[
    P \mapsto \mathrm{R}(\widetilde{f})
    \left( \varphi_{P'}
    \cdot \overline{\prescript{}{P}{\theta(\cdot, \phi_0)}} \right).
    \]
We recall that $\widetilde{f}$ is defined in ~\eqref{eq:u_tildef}.
The second and the third maps are to restrict the family $\left(
\prescript{}{P}{\varphi}\right)_{P \in {\cF_V}}$ to
$\left(\prescript{}{P}{\varphi}|_{[\U_V]_{P_{\U}}} \right)_{P \in {\cF_V}}$ and then
further to $\left(\prescript{}{P}{\varphi}|_{[\U'_V]_{P'}} \right)_{P \in
{\cF_V}}$. The fourth map sends a family $(\prescript{}{P'}{\varphi})_{P \in
{\cF_V}}$ to the family $\left(\prescript{}{P}{\varphi}
\prescript{}{P}{\theta(\cdot, \phi_2)}\right)_{P \in {\cF_V}}$. By
Lemma~\ref{lemma:smoothened_constant_term_U},
Lemma~\ref{lem:T_space_relation} and
Lemma~\ref{lemma:multiplication_by_theta}, the targets of these maps are as
described. Using the closed graph theorem, one checks that all these maps are
continuous, cf.~Remark~\ref{remark:closed_graph_theorem}.
The last map is one of the truncation operators defined in
Proposition~\ref{prop:relative_truncation}.
We denote by $L_{f,\phi_1,\phi_2}$ (resp. $P_{f, \phi_1, \phi}$) the composite of the sequence of the maps
above, where we use $\Lambda^T$ (resp. $\Pi^T$) in the last step.

Similarly, for $\chi \in \fX(\U_V)$, we consider the same chain of continuous
linear maps as above, but we project to the $\chi^\vee$-component before
restricting to $\U_V'$. The resulting maps are denoted by
$L_{f,\phi_1,\phi_2,\chi}$ and $P_{f,\phi_1,\phi_2,\chi}$ respectively.

As in the proof of Theorem~\ref{thm:convergence_first_GL}, the functions
$K^T_{f \otimes \phi_1 \otimes \phi_2}(x,y)$ and $K_{f \otimes \phi_1 \otimes
\phi_2}(x,y) F^{\U'_V}(x,T)$ are the kernel functions of $L_{f,\phi_1,\phi_2}$
(resp. $P_{f,\phi_1,\phi_2}$) respectively. The same is true with a
$\chi \in \fX(\U_V)$ in the subscript. By Proposition
~\ref{prop:relative_truncation}, for any fixed $N$ and for all $\varphi \in
\cT^0([\U'_V])$, we have
    \[
    \sum_{\chi \in \fX(U)}
    \| L_{f,\phi_1,\phi_2,\chi}(\varphi)
    - P_{f,\phi_1,\phi_2,\chi}(\varphi) \|
    \ll e^{-N\|T\|} \| \varphi \|_{1,-N}.
    \]
The rest of the argument is to show that the implicit constant in this
estimate can be taken to be continuous seminorms of $f$, $\phi_1$ and
$\phi_2$. This is a consequence of the uniform boundedness principle, and the
argument is the same as the proof of Theorem~\ref{thm:convergence_first_GL}.
\end{proof}

\subsection{The coarse spectral expansion}  \label{subsec:coarse_spectral_U}
For $\chi \in \fX(\U_V)$, $f_+ \in \cS(\U_{V,+}(\bA))$ and $T \in
\fa_0$ a truncation parameter, we put
    \[
    J^T(f_+) = \int_{[\U_V'] \times [\U_V']}
    K_{f_+}^T(x,y) \rd x \rd y, \quad
    J^T_\chi(f_+) = \int_{[\U_V'] \times [\U_V']}
    K_{f_+,\chi}^T(x,y) \rd x \rd y.
    \]
By Proposition~\ref{prop:u_asymptotic_modified_kernel}, these integrals are
absolutely convergent when $T$ is sufficiently positive.

We define an action $\mathrm{L}_+$ of $\U_V'(\bA)$ on $\cS(\U_{V,
+}(\bA))$ by
    \[
    \mathrm{L}_+(h)f_+ (m, l_1, l_2) =
    \left(\omega^\vee(h) f_+(h^{-1}m, \cdot, l_2)\right)(l_1), \quad
    m \in \U_V(\bA), \ l_1, l_2 \in L^\vee(\bA).
    \]
The right hand side means that we first evaluate $f_+$ at $h^{-1}m$ and $l_2$
to obtain a Schwartz function in the variable $l_1$. We apply the Weil
representation $\omega^\vee(h)$ to this Schwartz function and finally
evaluate at $l_1$. We similarly define an action $\mathrm{R}_+$ of
$\U_V'(\bA)$ on $\cS(\U_{V, +}(\bA))$ by
    \[
    \mathrm{R}_+(h)f_+ (m, l_1, l_2) =
    \left(\omega(h) f_+(m h, l_1, \cdot)\right)(l_2), \quad
    m \in \U_V(\bA), \ l_1, l_2 \in L^\vee(\bA).
    \]
The right hand side is interpreted similarly as in the case $\mathrm{L}_+$.

Using the action $L_+$ and $R_+$, for $P \in \cF_V$ and $f_+ \in \cS(\U_{V,+}(\bA))$, the kernel function $K_{f_+,P}$ can also be written as
\begin{equation} \label{eq:u_kernel_L_+_R_+}
    K_{f_+,P}(x,y) = \sum_{m \in M_{P_+}(F)} \int_{N_{P_+}(\bA)} \mathrm{L}_+(x)\mathrm{R}_+(y)f_+(mn) \rd n
\end{equation}
where $x,y \in [\U'_V]$.

\begin{theorem} \label{thm:coarse_spectral_u}
As a function of $T$, the functions $J^T(f_+)$ and $J^T_\chi(f_+)$ are the restrictions of
exponential polynomials whose purely polynomial term are constants denoted by
$J(f_+)$ and $J_\chi(f_+)$ respectively. The linear forms $f_+ \mapsto J(f_+)$ and $f_+ \mapsto
J_{\chi}(f_+)$ are continuous and bi-$\U_V'(\bA)$-invariant, i.e.
    \[
    J_{\chi}(\mathrm{L}_+(h_1) \mathrm{R}_+(h_2) f_+) =
    J_{\chi}(f_+), \quad
    J(\mathrm{L}_+(h_1) \mathrm{R}_+(h_2) f_+) =
    J(f_+), \quad h_1, h_2 \in \U_V'(\bA).
    \]
Finally we have
    \[
    J(f_+) = \sum_{\chi \in \fX(\U_V)} J_\chi(f_+)
    \]
where the sum is absolutely convergent.
\end{theorem}

Before we delve into the proof of this theorem, let us first explain a
variant of the construction of the modified kernel for parabolic subgroups.
Let us take $Q = M_Q N_Q \in \cF_V$. Then $Q_{\U} = Q_1 \times Q_2$ where $Q_1 =
Q_2$ are parabolic subgroup of $\U(V)$ and $Q' = Q_1 = Q_2$ is a parabolic
subgroup of $\U_V'$. Assume that $Q'$ is the stabilizer of the isotropic flag
    \begin{equation}    \label{eq:isotropic_flag_Q'}
    0 = X_0 \subset X_1 \subset \cdots \subset X_r
    \end{equation}
in $V$. Recall that we constructed a decomposition
$V = X \oplus V' \oplus X^\vee$ in
Subsection~\ref{subsec:u_Jacobi}, where $X = X_r$,
$V'=M_{Q_V}$, and $V'$ is perpendicular to $X \oplus X^\vee$ . We have a polarizations $\Res V = L \oplus L^\vee$ and another polarization $\Res V' =  L' \oplus L'^{\vee}$ where $L' = V' \cap L$ and $L'^\vee = V' \cap L^\vee$. In particular $L = L' \oplus X$
and $L^\vee = L'^\vee + X^\vee$.

The Levi subgroup $M_{Q'}$ is isomorphic to
    \[
    \prod_{i = 1}^{r} \GL(X_i/X_{i-1}) \times \U(V_0).
    \]
We write $m \in M_{Q'}$ as $(m_{\bullet}, m_*)$ where $m_{\bullet} = (m_1,
\hdots, m_r)$, $m_i \in \GL(X_i/X_{i-1})$, and $m_* \in \U(V_0)$. Recall that in Subsection ~\ref{subsec:spec_unitary_setup}, we have defined
    \[
    M_{Q_2, +}= M_{Q_2} \times M_{Q_{L^\vee}} \times M_{Q_{L^\vee}}  =  \left( \prod_{i = 1}^{r} \GL(X_i/X_{i-1}) \times
    \U(V') \right) \times L'^\vee \times L'^\vee.
    \]
We have the D-Levi component $ M_{\widetilde{Q}} = M_{Q_1} \times M_Q$
of $\widetilde{Q}$, where
    \[
    M_{Q} = \prod_{i = 1}^{r} \GL(X_i/X_{i-1}) \times J(V')
    \]
is the D-Levi component of $Q$. If $\widetilde{m}
\in M_Q$, we still denote by $\widetilde{m}_*$ its component
in $J(V_0)$.

Take a $P\in \cF_V$ with $P \subset Q$. Then $P \cap M_{Q_{\U}}$ is a parabolic subgroup of $M_{Q_{\U}}$. Denote temporarily by $ \iota_{Q_{\U}}: \fX(M_{Q_{\U}}) \to \fX(\U_V)$
the natural finite-to-one map. For $\chi \in \fX(\U_V)$  and $f \in \cS(M_{Q_{\U}}(\bA))$, we put
    \[
    K_{f, P \cap M_{Q_{\U}}, \chi} = \sum_{\chi' \in \iota_{Q_{\U}}^{-1}(\chi)}
    K_{f, P \cap M_{Q_{\U}}, \chi'}
    \]
where $K_{f, P \cap M_{Q_{\U}}, \chi'}$ is the kernel function on $M_{Q_{\U}}$.
Let $\phi_1,\phi_2 \in \cS(L'^\vee(\bA))$ and $f_+'= f \otimes \phi_1 \otimes \phi_2 \in \cS(M_{Q, +}(\bA))$ be a test function.
We define
    \[
    K_{f_+',P \cap M_Q,\chi}(x, y) =
    K_{f, P \cap M_{Q_{\U}}, \chi}(x, y) \cdot
    {}_{P \cap J(V')}  \theta^\vee(x,\phi_1) \cdot
    {}_{P \cap J(V')} \theta(y, \phi_2)  \mu^{-1}(\det x_\bullet) \mu(\det y_\bullet),
    \]
where $(x, y) \in [M_{Q'}]_{P' \cap M_{Q'}} \times [M_{Q'}]_{P' \cap M_{Q'}}$, and the intersection $P \cap J(V')$ is taken inside $J(V)$.
As in
Subsection~\ref{subsec:u_main_thm_spectral}, we define $K_{f'_+,P \cap
M_Q,\chi}$ for all $f'_+ \in \cS(M_{Q_+}(\bA))$ by continuity.

For $T \in \fa_0$, define
    \[
     K_{f'_+,\chi}^{M_Q, T}(x, y) =
     \sum_{\substack{P \subset Q \\ P \in \cF_V}}
     \epsilon_P^Q \sum_{\substack{\gamma \in (M_{Q'} \cap P')(F)
     \backslash M_{Q'}(F) \\
     \delta \in (M_{Q'} \cap P')(F) \backslash M_{Q'}(F)}}
     \widehat{\tau}_{P'}^{Q'}(H_{P'}(\delta y)- T_{P'})
     K_{f_+', P,\chi}(\gamma x,\delta y).
    \]
where $(x, y) \in [M_{Q'}] \times [M_{Q'}]$. Using similar methods as the
proof of Proposition~\ref{prop:u_asymptotic_modified_kernel}, we can show
that when $T$ is sufficiently positive, for any $f_+' \in \cS(M_{Q_+}(\bA))$,
the integral
    \[
    \int_{ A_{Q'}^\infty \backslash [M_{Q'}] \times [M_{Q'}]}
    \left| K_{f'_+,\chi}^{M_Q,T}(x, y) \right| \rd x \rd y
    \]
is convergent and defines a continuous seminorm on $\cS(M_{Q_+}(\bA))$. Here $A_{Q'}^\infty$ embeds in $M_{Q'} \times M_{Q'}$ diagonally. We
thus define a distribution
    \[
    J_{\chi}^{M_{Q}, T}(f_+') =
    \int_{ A_{Q'}^\infty \backslash [M_{Q'}] \times [M_{Q'}]}
     K_{f'_+,\chi}^{M_Q,T}(x, y) \rd x \rd y.
    \]

\begin{proof}[Proof of Theorem~\ref{thm:coarse_spectral_u}]

We define $\underline{\rho}_Q$ be the unique element in $\fa_{Q'}^*$ such that for any $m \in M_{Q'}(\bA)$ we have
    \[
    e^{\langle \underline{\rho}_Q, H_{Q'}(m) \rangle} =
    \abs{\det m_\bullet}.
    \]
One check directly that $\underline{\rho}_Q$ concides with the definition of ~\cite{Zydor3}*{Lemma 4.3}.

By~\cite{Arthur1}*{Section~2}, there exist functions $\Gamma'_{Q'}$ on
$\fa_{Q'}^{\U_V'} \times \fa_{Q'}^{\U_V'}$, for $Q \in \cF_V$, that are
compactly supported in the first variable when the second variable stays in a
compact and such that
    \begin{equation}    \label{eq:Gamma_Q_U}
    \widehat{\tau}_{Q'}(H-X) = \sum_{P \subset Q \in \cF_V}
    \epsilon_{Q'}^{\U_V'} \widehat{\tau}_{P'}^{Q'}(H)
    \Gamma'_{Q'}(H, X).
    \end{equation}
Define a function $p_Q$ on $\fa_{Q'}$ by
    \begin{equation}    \label{eq:polynomial_exponential_Q'}
    p_Q(X) = \int_{\fa_{Q'}}
    e^{\langle \underline{\rho}_{Q}, H \rangle} \Gamma'_{Q'}(H,X) \rd H.
    \end{equation}
By~\cite{Zydor3}*{Lemma 4.3} $p_Q$ is an exponential polynomial on
$\fa_{Q'}$ with exponents contained in the set $\{\underline{\rho}_R
\mid R \supset Q \}$ and the pure polynomial term is the constant $\epsilon_Q
\widehat{\theta}_{Q'}(\underline{\rho}_Q)^{-1}$ where $\widehat{\theta}_{Q'}$ is a homogeneous polynomial on $\fa_{Q'}$ defined
in~\cite{Arthur1}*{Section~2}.

For $Q \in \cF_V, f_+ \in \cS(\U_{V,+}(\bA))$ and $ T \in \fa_0$, define
    \[
    K_{f_+,\chi}^{Q,T}(x,y) =
    \sum_{\substack{P \subset Q \\ P \in \cF_V}}
    \epsilon_P^Q \sum_{\substack{\gamma \in P'(F) \backslash Q'(F) \\
    \delta \in P'(F) \backslash Q'(F)}}
    \widehat{\tau}_{P'}^{Q'}(H_{P'}(\delta y)-T_{P'})
    K_{f_+,P,\chi}(\gamma x,\delta y),
    \]
where $(x,y) \in [\U_V']_{P'} \times [\U_V']_{P'}$. Here $\epsilon_P^Q =
(-1)^{\dim \fa_{P'}^{Q'}}$ and $\widehat{\tau}_{P'}^{Q'}$ is the characteristic
function defined in Subsection~\ref{subsec:reduction_theory}.
Using the inversion formula~\eqref{eq:Gamma_Q_U}, for $T,T'
\in \fa_0$ we have
    \[
    K_{f_+,\chi}^T(x,y) = \sum_{Q \in \cF_V}
    \sum_{ \substack{\gamma \in Q'(F) \backslash \U'(F) \\
    \delta \in Q'(F) \backslash \U'(F)}}
    \Gamma'_{Q'}(H_{Q'}(\delta y)-T'_{Q'}, T_{Q'}-T'_{Q'})
    K_{f_+,\chi}^{Q,T'}(\gamma x,\delta y).
    \]

We now relate $K_{f_+,\chi}^{Q,T'}$ to $K_{f_+,\chi}^{M_Q,T'}$ via parabolic
descent. For $f_+ \in \cS(\U_{V,+}(\bA))$, we define its parabolic descent as
    \[
    f_{+, Q}(m, l_1, l_2) =
    e^{\langle -2 \rho_{Q'} + \frac 12 \underline{\rho}_Q  ,H_{Q'}(m_1) \rangle}
    \int_{K'} \int_{K'} \int_{N_{Q_{\U}}(\bA)}
    \left( \mathrm{L}_+(k_1)\mathrm{R}_+(k_2)
    f_+ \right)(m n_Q, l_1, l_2) \rd n_Q \rd k_1 \rd k_2,
    \]
where $m \in M_{Q_{\U}}(\bA)$, $l_1, l_2 \in L'(\bA)$ . The element $m_1$ stands for the first component of $m$ and we regard it as an element of $M_{Q'}$ under the natural identification $M_{Q_1} \cong M_{Q'}$. We have $f_{+,Q} \in
\cS(M_{Q_+}(\bA))$, and for $x, y \in [M_{Q'}]$ and $P \subset Q\in \cF_V$, we
have
    \[
    \int_{K'} \int_{K'} K_{f_+,P, \chi}(x k_1, y k_2) \rd k_1 \rd k_2
    =e^{ \langle  2 \rho_{Q'}  , H_{Q'}(x) \rangle}
    e^{ \langle  2 \rho_{Q'} + \underline{\rho}_Q , H_{Q'}(y) \rangle}
    K_{f_{+,Q}, P \cap M_Q, \chi}(x, y).
    \]
Indeed using ~\eqref{eq:u_kernel_L_+_R_+} and the mixed model described in Subsection ~\ref{subsec:theta_series_U}, direct calculations give the identity without $\chi$. The argument
in~\cite{Zydor3}*{Lemma~1.3} shows that this implies the identity with the
$\chi$. It follows that
    \[
    \int_{K'} \int_{K'} K^T_{f_+,\chi}(x k_1, y k_2) \rd k_1 \rd k_2
    = e^{ \langle   2 \rho_{Q'} , H_{Q'}(x) \rangle}
    e^{ \langle  2 \rho_{Q'}  +  \underline{\rho}_Q , H_{Q'}(y) \rangle}
    K_{f_{+,Q},\chi}^{M_Q, T}(x, y).
    \]

The rest of the calculation is the same as that in the proof of
Theorem~\ref{thm:coarse_spectral_expansion_GL}. We omit the details and only record the final
outcome. We have
    \[
    J_\chi^T(f_+) = \sum_{Q \in \cF_V}
    e^{ \langle \underline{\rho}_{Q'},T'\rangle}
    p_Q(T_{Q'}-T'_{Q'}) J_\chi^{M_Q,T'}(f_{+,Q}).
    \]
It is an exponential polynomial in $T$ whose purely polynomial term is a constant that equals
    \[
    \sum_{Q \in \cF_V}  \epsilon_{Q}
    \widehat{\theta}_{Q'}(\underline{\rho}_Q)^{-1}
    e^{ \langle \underline{\rho}_{Q'},T'\rangle}
    p_Q(T_{Q'}-T'_{Q'}) J_\chi^{M_Q,T'}(\widetilde{f_{+,Q}}).
    \]

The invariance of the linear form $J_\chi$ is shown as that of
Proposition~\ref{prop:invariant_I_chi}. This concludes the proof.
\end{proof}

\subsection{An alternative truncation operator}
\label{sec:truncation_IY}

To compute explicit spectral terms of the relative trace formula in ~\cite{BLX3}, we require a new truncation operator for Fourier-Jacobi periods. As its construction relies on the machinery developed herein, we present it in this independent subsection. This operator serves as an analogue to the regularized periods introduced by Ichino and Yamana ~\cite{IY2}.

Take $P \in \cF_V$. Then $\U(V) \times P$ is a D-parabolic subgroup of
$\widetilde{\U_V}$. For $\varphi \in
\cT([\widetilde{\U_V}], \psi)$, we can speak of the constant term
$\varphi_{\U(V) \times P}$. This amounts to viewing the function
$\varphi$ as a function in two variables $(x_1, \widetilde{x_2})
\in [\U(V)] \times [J(V)]$, fixing
$x_1$, and taking the constant term along $P$ in the second variable
$\widetilde{x_2}$.

For $\varphi \in \cT([\widetilde{\U_V}], \psi)$ and $x \in [\U'_V]$, we
define
\begin{equation}
\label{eq:unitary_truncation}
     \Lambda^T_u \varphi(x) =
    \sum_{P \in \cF_V} \epsilon_P
    \sum_{\delta \in P'(F) \backslash \U'_V(F)}
    \widehat{\tau}_{P'}(H_{P'}(\delta x)-T_{P'})
    \varphi_{\U(V) \times P}(\delta x).
\end{equation}

\begin{prop}    \label{prop:IY_properties}
We have the following assertions.
\begin{enumerate}
    \item $\Lambda_u^T \varphi \in \cS^0([\U'_V])$. And the map $\varphi
        \mapsto \Lambda_u^T \varphi$ induces a continuous map
        $\cT([\widetilde{\U_V}],\psi) \to \cS^0([\U_V'])$.

    \item For every $N>0$, there exists a continious seminorm $\aabs{\cdot}_N$ on
        $\cT([\widetilde{\U_V}],\psi)$ such that for all $x \in [\U_V']$ and
        $\varphi \in \cT([\widetilde{\U_V}],\psi)$, we have
            \[
            \left| \Lambda^T_u \varphi(x)-F^{\U_V'}(x,T)\varphi(x) \right|
            \le e^{-N\|T\|} \|x\|_{\U_V'}^{-N} \aabs{\varphi}_N.
            \]
\end{enumerate}
\end{prop}

\begin{proof}
First note $ d_{\U(V) \times P}^{\U(V) \times Q}|_{\U'_V}  = d_{P}^{Q}$. Thus by Theorem~\ref{thm:approximation_by_constant_term} and
Lemma~\ref{lem:d_P^Q_unitary_equivalent}, there exists an $N_0>0$ such that for any $P \subset Q \in \cF_V$ and $X \in \cU
((\fu'_V)_\infty)$, we have
    \begin{equation} \label{eq:IY_suffices}
    \abs{ \mathrm{R}(X) \varphi_{\U(V) \times P}(x)
    - \mathrm{R}(X) \varphi_{\U(V) \times Q}(x)}
    \ll d_{P'}^{Q'}(x)^{-r} \|x\|_{P'}^{N_0}.
    \end{equation}
By~\eqref{eq:characterization_T_U'_V},
this is equivalent to the fact that the family $P \mapsto  \left( \varphi_{\U(V) \times P}\right) |_{[\U_V']_{P'}}$
belongs to $\cT_{\cF_V'}(\U_V')$. The second assertion follows from this by Proposition~\ref{prop:relative_truncation}. The first assertion follows from the second since the function $x \mapsto F^{\U_V'}(x,T)$ is compactly supported.
\end{proof}

Recall that for $f_+ \in \cS(\U_{V,+}(\bA))$ and $\chi \in \fX(\U_V)$, we have a
kernel function $K_{f_+,\chi}(\widetilde{x},\widetilde{y})$ for
$\widetilde{x},\widetilde{y} \in [\widetilde{\U_V}]$ (cf. \eqref{eq:unitary_kernel_defi} for the definition for pure tensors). Applying $\Lambda_u^T$ to
the second variable, we get a function on $[\widetilde{\U_V}] \times [\U_V']$
denoted by $K_{f_+,\chi} \Lambda^T_u$.

\begin{prop}    \label{prop:comparison_of_asymptotics}
We have the following assertions.
    \begin{enumerate}
        \item For $f_+ \in \cS(\U_{V, +}(\bA))$ and $T$ sufficiently positive,
            the expression
        \[
        \sum_{\chi \in \fX(\U_V)} \int_{[\U_V'] \times [\U_V']}
        \left| K_{f_+,\chi}\Lambda^T_u (x,y) \right| \rd x \rd y
        \]
        is finite and defines a continuous seminorm on $\cS(\U_{V,
        +}(\bA))$.

        \item For any $r>0$, there exists a continuous seminorm
            $\aabs{\cdot}$ on $\cS(\U_{V, +}(\bA))$ such that
        \[
        \left| J_\chi^T(f_+) - \int_{[\U_V'] \times [\U_V']}
        K_{f_+,\chi}\Lambda^T_u (x,y) \rd x \rd y  \right|
        \le e^{-r\|T\|} \|f_+\|.
        \]
        for all $T$ sufficiently positive and $f_+ \in \cS(\U_{V,
        +}(\bA))$. In particular, the absolutely convergent integral
        \[
        \int_{[\U_V'] \times [\U_V']} K_{f_+,\chi}
        \Lambda^T_u (x,y) \rd x \rd y
        \]
        is asymptotic to an exponential-polynomial in $T$, whose purely
        polynomial term is a constant that equals $J_\chi(f_+)$.
    \end{enumerate}
\end{prop}

\begin{proof}
By Proposition~\ref{prop:IY_properties}, for every $N>0$ there exists a
continuous semi-norm $\| \cdot \|_\cT$ on $\cT_N([\widetilde{\U_V}], \psi)$ such for all $x, y
\in [\U_V']$, we have
    \[
    \sum_{\chi \in \fX(\U_V)}
    \left| K_{f_+,\chi}\Lambda_u^T(x,y)
    - F^{\U_V'}(y,T)K_{f_+,\chi}(x,y) \right|
    \le  e^{-N\|T\|} \| y \|_{\U_V'}^{-N}
    \sum_{\chi \in \fX(\U_V)} \| K_{f_+,\chi}(x,\cdot)  \|_\cT,
    \]
where $K_{f_+,\chi}(x,\cdot)$ is regarded as an element in $\cT_N([\widetilde{\U_V}], \psi)$.
By~\eqref{eq:any_tensor_estimate}, there exists $N_0>0$ and a continuous
semi-norm $\| \cdot \|_\cS$ on $\cS(\U_{V, +}(\bA))$ such that
    \[
    \sum_{\chi \in \fX(\U_V)} \| K_{f_+,\chi}(\widetilde{x},\cdot) \|_\cT
    \ll \|f_+\|_{\cS} \|\widetilde{x}\|_{\widetilde{\U_V}}^{-N+N_0}.
    \]
Combining the above two equations we obtain
    \[
    \sum_{\chi \in \fX(\U_V)}
    \left| K_{f_+,\chi}\Lambda_u^T(x,y) - F^{\U_V'}(y,T)
    K_{f_+,\chi}(x,y) \right|
    \le e^{-N\|T\|} \|f_+\|_{\cS} \|x\|_{\U_V'}^{-N+N_0} \|y\|_{\U_V'}^{-N}.
    \]
This proves the first assertion. The second assertion follows from the first
and Proposition~\ref{prop:u_asymptotic_modified_kernel}.
\end{proof}

We also introduce a variant of $\Lambda_u^T$ for parabolic subgroup. Let $R \in \cF_V$ and $\varphi \in \cT([\widetilde{\U_V}]_{\U(V) \times R},\psi)$ and $x \in [\U_V']_{R'}$, we define
\begin{equation*}
    \Lambda^{R,T}_u \varphi(x) = \sum_{ \substack{P \in \cF_V \\ P \subset R}} \epsilon_P^R \sum_{\delta \in {P'(F) \backslash R'(F)}} \widehat{\tau}_{P'}^{R'}(H_{P'}(\delta x)-T_{P'}) \varphi_{\U(V) \times P}(\delta x).
\end{equation*}

It has similar properties as listed in Proposition \ref{prop:IY_properties} and moreover, for $T$ and $T'$ sufficiently positive and $\varphi \in \cT([\widetilde{\U_V}],\psi)$, we have

\begin{equation} 
    \Lambda^{T+T'}_{u}\varphi(x) = \sum_{R \in \cF_V} \sum_{\gamma \in R'(F) \backslash \U'(F)} \Gamma'_{R'}(H_{R'}(\gamma x)-T_{R'},T'_{R'}) \Lambda^{R,T}_u \varphi_{\U(V) \times R}(\gamma x),
\end{equation}
where $\Gamma'_{R'}$ is a function on $\fa_{R'} \times \fa_{R'}$ introduced in \cite{Arthur1}*{Section 2}.

\section{The coarse geometric expansion: unitary groups}
\label{sec:geo_U}

\subsection{Geometric modified kernels}
We keep the notation from the previous section. We also need the following
additional notation.

\begin{itemize}
\item We put $\U_V^+ = \U_V \times V$. The group structure is given by the
    product of those of $\U_V$ and of the additive group $V$. Note that this is not a subgroup
    of $\widetilde{\U_V}$ but merely a closed subvariety.
    The group $\U_V' \times \U_V'$
    acts on $\U_V^+$ from the right by $(g, v) \cdot (x, y) = (x^{-1} g y, y^{-1} v)$.

\item Let $P  \in \cF_V$, we put
        \[
        P^+ = P_{\U} \times P_V, \quad
        M_P^+ = M_{P_{\U}} \times M_{P_V}, \quad
        N_P^+ = N_{P_{\U}} \times N_{P_V} .
        \]
    These are subgroups of $\U_V^+$.

\item By~\cite{CZ}*{Lemma~15.1.4.1} the categorical quotient $\U_V^+//(\U'_V \times \U'_V)$ is canonically identified with $\cA = G^+//(H \times G')$. The canonical morphism $q_V: \U_V^+ \to \cA$ is given by
        \[
        ((g_1, g_2), v) \mapsto (a_1, \hdots, a_n; b_1, \hdots, b_n)
        \]
    where
        \[
        a_i = \Trace \wedge^i (g_1^{-1} g_2), \quad
        b_i = 2 (-1)^{n-1} \tau^{-1} q_V( g_1^{-1} g_2 v, v).
        \]
    Here $\cA$ is viewed as a locally closed subscheme of $\Res_{E/F} \mathbf{A}_{2n, E}$ as in Subsection~\ref{subsec:geo_GL_modified_kernel}.

\item For $\alpha \in \cA(F)$, let $\U^+_{V, \alpha}$ be the preimage of $\alpha$ in
    $\U^+_V$ as a closed subscheme. We also put
        \[
        M^+_{P, \alpha} = \U^+_{V, \alpha} \cap M^+_P.
        \]
\end{itemize}

For $f^+ \in \cS(\U_V^+(\bA))$ and $P \in \cF_V$, we define
    \[
    k_{f^+,P}(x,y) = \sum_{m^+ \in M_P^+(F)}
    \int_{N_P^+(\bA)} f^+(m^+ n^+ \cdot (x,y)) \rd n^+, \quad
    x, y \in [\U_V']_{P'}.
    \]
For $\alpha \in \cA(F)$, we define similarly
    \[
    k_{f^+,P,\alpha}(x,y) = \sum_{m^+ \in M_{P,\alpha}^{+}(F)}
    \int_{N_P^+(\bA)} f^+(m^+ n^+ \cdot (x,y)) \rd n^+.
    \]

\begin{lemma}
There are an integer $N$ and a semi-norm $\aabs{\cdot}_{\cS}$ on
$\cS(\U_V^+(\bA))$ such that for all
    \[
    \sum_{m \in M_P^+(F)}
    \int_{N_P^+(\bA)} \Abs{f^+(mn \cdot (x,y))} \rd n
    \leq \aabs{f^+}_{\cS} \aabs{x}_{\U_V'}^N \aabs{y}_{\U_V'}^N.
    \]
In particular the defining expressions of $k_{f^+, P,\alpha}$ and $k_{f^+,
P}$ are absolutely convergent and we have
    \[
    \sum_{\alpha \in \cA(F)} k_{f^+,P, \alpha}(x, y) = k_{f^+, P}(x, y).
    \]
\end{lemma}

\begin{proof}
The proof is the same as
Lemma~\ref{lemma:convergence_geometric_kernel_definition}.
\end{proof}

Similar to what we have done in Subsection ~\ref{subsec:geo_GL_modified_kernel}, we identify $\Res_{E/F} \mathbf{A}_{2n, E}$ with the affine space $\mathbf{A}^{4n}$ over $F$, and denote again by $q_V$ the morphism $\U_V^+ \to \mathbf{A}^{4n}$. We extend the definition of $k_{f^+, P, \alpha}$ to all $\alpha \in F^{4n}$ by setting $k_{f^+, P, \alpha} = 0$ if $\alpha \not\in \cA(F)$.
There is a $d \in F^\times$ such that if $k_{f^+, P, \alpha}$ is not identically zero for some $\alpha \in
F^{4n}$ and $P \in \cF_V$, then $\alpha \in (d\cO_F)^{4n}$. Put $\Lambda = (d\cO_F)^{4n} \subset F^{4n}$. For each
$\alpha \in \Lambda$, take $u_{\alpha} \in C_c^\infty(F^{4n}_\infty)$ and define
the function $f^+_\alpha = f^+ u_{\alpha}$ as in
Subsection~\ref{subsec:decomposition_of_Schwartz_functions} .

\begin{lemma}   \label{lem:kernel_function_of_truncated_function_U}
We have
    \[
    k_{f^+_{\alpha}, P}(x, y) =
    k_{f^+_{\alpha}, P, \alpha}(x, y) = k_{f^+, P, \alpha}(x, y)
    \]
for all $x, y \in [\U_V']_{P'}$.
\end{lemma}

The proof is the same as
Lemma~\ref{lem:kernel_function_of_truncated_function_linear}.

We now relate the kernel function $k_{f^+, P}$ and the kernel function
$K_{f_+, P}$ defined in Subsection~\ref{subsec:u_main_thm_spectral}. We
first define a partial Fourier transform, cf.~\cite{Li92}*{Section~2}.
For $\phi_1, \phi_2 \in
\cS(L^\vee(\bA))$ we define a partial Fourier transform
    \begin{equation}    \label{eq:weil_U_partial_FT}
    \cS(L^\vee(\bA)) \otimes \cS(L^\vee(\bA)) \to
    \cS(V(\bA)), \quad \phi_1 \otimes \phi_2 \mapsto (\phi_1 \otimes \phi_2)^\ddag,
    \end{equation}
by
    \begin{equation}    \label{eq:weil_U_partial_FT_function}
    (\phi_1 \otimes \phi_2)^{\ddag}(v) =
    \int_{L^\vee(\bA)} \phi_1(x + l') \phi_2(x- l')
    \psi(- 2 \Tr_{E/F} q_V(x, l)) \rd x,
    \end{equation}
where we write $v = l + l'$ where $l \in L(\bA)$ and
$L^\vee(\bA)$. In particular we have
    \[
    (\phi_1 \otimes \phi_2)^{\ddag}(0) = \langle \phi_1, \overline{\phi_2} \rangle_{L^2}
    \]
where $\langle-, -\rangle_{L^2}$ stands for the $L^2$-inner product on $L^\vee(\bA)$.
We also have
    \[
    (\omega^\vee(g) \phi_1 \otimes \omega(g) \phi_2)^\ddag(v) =
    (\phi_1 \otimes \phi_2)^\ddag(g^{-1}v)
    \]
for $g \in \U(V)(\bA)$ and $v \in V(\bA)$.

The partial Fourier transform~\eqref{eq:weil_U_partial_FT} extends to a
continuous isomorphism
    \[
    \cS(\U_{V,+}(\bA)) \to \cS(\U_{V}^+(\bA)),
    \]
which we still denote by $-^\ddag$. Let $f^+ \mapsto f^+_{\ddag}$ be its
inverse. Concretely for $f_+ \in \cS(\U_{V, +}(\bA))$,
we have
    \begin{equation}    \label{eq:global_ddag_map}
    f_+^\ddag((g_1, g_2), v) = \int_{L^\vee}
    \omega^\vee_{(1)}(g_1) f_+((g_1, g_2), x+l', x-l')
    \psi( - 2 \Tr_{E/F} q_V(x, l) ) \rd x,
    \end{equation}
where we write $v = l + l'$ where $l \in L(\bA)$ and
$L^\vee(\bA)$. The notation is interpreted as follows. We first evaluate $f_+$ at
$(g_1, g_2)$ to obtain a Schwartz function on $L^\vee(\bA) \times L^\vee(\bA)$.
Then the Weil representation $\omega^\vee(g_1)$ acts on this Schwartz function
on the first variable (this is what the subscript $(1)$ indicates).
Finally we take the partial Fourier transform.

\begin{lemma}   \label{lemma:k=K_U}
For all $x, y \in [\U_V']_{P'}$, we have
    \[
    K_{f_+, P}(x, y) = k_{f_+^\dag, P}(x, y).
    \]
\end{lemma}

The proof is the same as Lemma~\ref{lemma:k=K}.

Recall that we have defined in Subsection~\ref{subsec:coarse_spectral_U} the
actions $\mathrm{L}_+$ and $\mathrm{R}_+$ of $\U_V'(\bA)$ on $\cS(\U_{V,
+}(\bA))$. We also have the left and right translation of $\U_V'(\bA)$ on
$\cS(\U_V^+(\bA))$ which we denote by
$\mathrm{L}^+$ and $\mathrm{R}^+$ respectively. More precisely we have
    \[
    \mathrm{L}^+(h) f^+(g, v) = f^+(h^{-1} g,  v), \quad
    \mathrm{R}^+(h) f^+(g, v) = f^+(gh, h^{-1} v),
    \]
for $h \in \U_V'(\bA)$.

\begin{lemma}   \label{lemma:FT_invariance_U}
Let $f^+ \in \cS(G^+(\bA))$ and $x, y \in \U_V'(\bA)$. We have
    \[
    (\mathrm{L}^+(x)\mathrm{R}^+(y) f^+)_{\dag} =
    \mathrm{L}_+(x)\mathrm{R}_+(y) (f^+_\dag) .
    \]
\end{lemma}

\begin{proof}
This is a direct computation.
\end{proof}

For $T \in \fa_0$ and $x, y \in [\U_V']$, we define the modified kernel
    \[
    k_{f^+}^T(x,y) =
    \sum_{P \in \cF_V} \epsilon_P
    \sum_{\substack{\gamma \in P'(F) \backslash \U_V'(F)
    \\
    \delta \in P'(F) \backslash \U_V'(F)}}
    \widehat{\tau}_{P'}(H_{P'}(\delta y)-T_{P'})
    k_{f^+,P}(\gamma x,\delta y),
    \]
and for $\alpha \in \cA(F)$ we put
    \[
     k_{f^+,\alpha}^T(x,y) =
    \sum_{P \in \cF_V} \epsilon_P
    \sum_{\substack{\gamma \in P'(F) \backslash \U_V'(F)
    \\
    \delta \in P'(F) \backslash \U_V'(F)}}
    \widehat{\tau}_{P'}(H_{P'}(\delta y)-T_{P'})
    k_{f^+, \alpha, P}(\gamma x, \delta y).
    \]

\begin{lemma}   \label{lem:integration_of_modified_kernel_seminorm_U}
For $T$ sufficiently positive, the integral
    \[
    \int_{[\U_V'] \times [\U_V']} \Abs{k_{f^+}^T(x,y)} \rd x \rd y
    \]
is absolutely convergent and defines a continuous seminorm on
$\cS(\U_V^+(\bA))$. Put
    \[
    j^T(f^+) := \int_{[\U_V'] \times [\U_V']} k_{f^+}^T(x,y) \rd x \rd y.
    \]
Then $j^T$ is the restriction of an exponential-polynomial function of $T$
whose  purely polynomial term is a constant that equals $J((f^+)_\ddag)$.
\end{lemma}

\begin{proof}
The absolute convergence follows from Lemma~\ref{lemma:k=K_U} and
Proposition~\ref{prop:u_asymptotic_modified_kernel}, the second follows from
Lemma~\ref{lemma:k=K_U} and Theorem~\ref{thm:coarse_spectral_u}.
\end{proof}

\subsection{The coarse geometric expansion}

The following theorem gives the geometric expansions of relative trace formulae on
unitary group.

\begin{theorem} \label{thm:geometric_u}
We have the following assertions.
\begin{enumerate}
    \item For $T$ sufficiently positive, the expression
        \[
        \sum_{\alpha \in \cA(F)} \int_{[\U_V'] \times [\U_V']}
        \Abs{k_{f^+, \alpha}^T(x,y)} \rd x \rd y
        \]
        is finite and defines a continuous seminorm on
        $\cS(\U_V^+(\bA))$.

    \item For $\alpha \in \cA(F)$ and $T$ sufficiently positive, put
        \[
        j^T_\alpha(f^+) = \int_{[\U_V'] \times [\U_V']}
        k_{f^+, \alpha}^T(x,y) \rd x \rd y.
        \]
        Then $j^T_\alpha$ coincides with the restriction of an exponential-polynomial function of $T$
        whose purely polynomial term is a constant $j_\alpha(f^+)$.

    \item The linear form $f^+ \mapsto j(f^+)$ is continuous and satisfies
        the invariant property that
            \[
            j(\mathrm{L}^+(x)\mathrm{R}^+(y)f^+) = j(f^+)
            \]
        for all $x, y \in \U'_V(\bA)$.
\end{enumerate}
\end{theorem}

The proof is the same as Theorem~\ref{thm:geometric_linear} and make use of
Proposition~\ref{prop:decomposition_of_Schwartz}.

\subsection{Synthesis of the results: the coarse relative trace formula}
We now summarize what we have done. For $f_+ \in \cS(\U_{V,+}(\bA))$ we put
    \[
    J_{\alpha}^T(f_+) = j_{\alpha}^T(f_+^\ddag), \quad
    J_{\alpha}(f_+) = j_{\alpha}(f_+^\ddag).
    \]
Then Theorem~\ref{thm:geometric_u} tells us that if $T$ is sufficiently
positive then $J_{\alpha}^T(f_+)$ is the restriction of a polynomial
exponential and the purely polynomial part is a constant that equals
$J_{\alpha}(f_+)$. By Lemma~\ref{lemma:FT_invariance_U}, the distribution
$J_{\alpha}$ is bi-$\U_V'(\bA)$-invariant, i.e. for $x, y \in \U_V'(\bA)$
we have
    \[
    J_{\alpha}(\mathrm{L}_+(x) \mathrm{R}_+(y) f_+ )
    = J_{\alpha}(f_+).
    \]

We summarize the coarse relative trace formula on the unitary groups as the
following theorem.

\begin{theorem} \label{thm:coarse_U}
Let $f_+ \in \cS(\U_{V,+}(\bA))$ be a test function. Then we have
    \[
    \sum_{\chi \in \fX(\U_V)} J_{\chi}(f_+) =
    \sum_{\alpha \in \cA(F)} J_{\alpha}(f_+).
    \]
Each summand on both sides are bi-$\U_V'(\bA)$-invariant.
\end{theorem}

This is simply a combination of Theorem~\ref{thm:coarse_spectral_u},
Lemma~\ref{lem:integration_of_modified_kernel_seminorm_U} and
Theorem~\ref{thm:geometric_u}.

\end{document}